\documentclass[reqno,a4paper,11pt]{amsart}
\usepackage{amsmath,amsthm,amssymb,mathrsfs}
\usepackage[
	left=2cm, 
        right=2cm, 
        top=2cm, 
]{geometry}
\usepackage{graphicx,xcolor,tikz,mathtools}

\usepackage{soul}
\usepackage{bm,stmaryrd}
\usepackage{amsaddr}
\usepackage{enumitem}
\usepackage[toc,page]{appendix}
\usepackage[
	pdfencoding = auto,
	colorlinks = true,
	citecolor = black,
	linkcolor = black,
	anchorcolor = black,
	urlcolor = black, 
	filecolor = black,
	pdfkeywords={}
]{hyperref}
\usepackage{bbm}

\theoremstyle{plain}
\newtheorem{theorem}{Theorem}[section]
\newtheorem{corollary}[theorem]{Corollary}
\newtheorem{definition}[theorem]{Definition}
\newtheorem{lemma}[theorem]{Lemma}
\newtheorem{proposition}[theorem]{Proposition}
\newtheorem{remark}[theorem]{Remark}
\theoremstyle{definition}
\newtheorem{assumption}[theorem]{Assumption}

\numberwithin{equation}{section}

\newcommand{\bb}[1]{\mathbb{#1}}
\newcommand{\bbr}{\bb{R}}
\newcommand{\bbi}{\bb{I}}
\newcommand{\bbn}{\bb{N}}
\newcommand{\bbs}{\bb{S}}
\newcommand{\bbt}{\bb{T}}

\newcommand{\bbd}{\bb{D}}

\newcommand{\bbone}{\mathbbm{1}}
\newcommand{\bbtwo}{\mathbbm{2}}

\newcommand{\bu}{\mathbf{u}}
\newcommand{\bv}{\mathbf{v}}

\newcommand{\bn}{\mathbf{n}}
\newcommand{\btau}{\bm{\tau}}

\newcommand{\be}{\mathbf{e}}

\newcommand{\bX}{{X}}
\newcommand{\bU}{\mathbf{U}}

\newcommand{\bF}{\mathbf{F}}

\newcommand{\cA}{\mathcal{A}}
\newcommand{\cB}{\mathcal{B}}

\newcommand{\cD}{\mathcal{D}}
\newcommand{\cE}{\mathcal{E}}
\newcommand{\cF}{\mathcal{F}^w}
\newcommand{\cR}{\mathcal{R}}
\newcommand{\cN}{\mathcal{N}}
\newcommand{\cH}{\mathcal{H}}

\newcommand{\bphi}{\bm{\varphi}}
\newcommand{\bPhi}{\mathbf{\Phi}}
\newcommand{\bPsi}{\mathbf{\Psi}}

\newcommand{\fp}{\mathfrak{p}}

\newcommand{\vr}{\varrho}
\newcommand{\vphi}{\varphi}

\renewcommand{\Bar}[1]{\overline{#1}}

\newcommand{\ess}[1]{\left[#1\right]_{\mathrm{ess}}}
\newcommand{\res}[1]{\left[#1\right]_{\mathrm{res}}}
\newcommand{\tin}{\quad \text{in }}
\newcommand{\ton}{\quad \text{on }}
\renewcommand{\d}{\mathrm{d}}
\newcommand{\dx}{\,\d x}
\newcommand{\dt}{\,\d t}
\newcommand{\ddt}{\frac{\d}{\d t}}
\newcommand{\dxdt}{\,\d x \d t}
\newcommand{\dtau}{\,\d \tau}
\newcommand{\dxdtau}{\,\d x \d \tau}
\newcommand{\dH}{\,\d \cH}
\newcommand{\ptial}[1]{ \partial_{#1} }
\newcommand{\pt}{\ptial{t}}

\newcommand{\onehalf}{\frac{1}{2}}
\newcommand{\onequater}{\frac{1}{4}}
\newcommand{\tr}{\mathrm{tr}}
\newcommand{\rv}[1]{\left. #1 \right\vert}
\newcommand{\rvm}[1]{#1 \vert}
\newcommand{\tran}[1]{ #1^{\top}}
\newcommand{\inv}[1]{ #1^{-1}}
\newcommand{\invtr}[1]{ #1^{-\top}}

\newcommand{\abs}[1]{\left\vert #1 \right \vert}
\newcommand{\absm}[1]{\vert #1 \vert}
\newcommand{\norm}[1]{\left\Vert #1 \right \Vert}
\newcommand{\normm}[1]{\Vert #1 \Vert}

\DeclareMathOperator*{\Div}{\mathrm{div}}
\DeclareMathOperator*{\esssup}{\mathrm{esssup}}

\newcommand{\dist}{\mathrm{dist}}

\allowdisplaybreaks[4]

\begin{document}
	
	\title[Compressible fluid-structure interaction]{Weak solutions and singular limits for a compressible fluid-structure interaction problem with slip boundary conditions
	}

        \author[Y. Liu, S. Mitra, \& \v{S}. Ne\v{c}asov\'a]{
            \small
            Yadong Liu$^{\ast}$, 
            Sourav Mitra$^{\dagger}$, and 
            \v{S}\'arka Ne\v{c}asov\'a$^\ddagger$
            }
        \address{
	   $^\ast$School of Mathematical Sciences, Ministry of Education Key Laboratory of NSLSCS, \\and Key Laboratory of Jiangsu Provincial Universities of FDMTA, \\Nanjing Normal University, Nanjing 210023, P. R. China
        }
        \email{ydliu@njnu.edu.cn}
        \address{
            $^\dagger$Department of Mathematics, 
            Indian Institute of Technology Indore, \\ 
            Simrol, Indore, 453552, Madhya Pradesh, India
        }
        
        \email{souravmitra@iiti.ac.in}
        
        \address{
            $^\ddagger$Institute of Mathematics,
            Czech Academy of Sciences,
            \v{Z}itn\'{a} 25, 115 67 Praha, Czech Republic
            }
        
    \email{matus@math.cas.cz}
 
	
	
	\date{\today}
	

    \begin{abstract}
		We study a system describing the compressible barotropic fluids interacting with (visco) elastic solid shell/plate. In particular, the elastic structure is part of the moving boundary of the fluid, and the Navier-slip type boundary condition is taken into account. Depending on the reference geometry (flat or not), we show the existence of weak solutions to the coupled system provided the adiabatic exponent satisfies $\gamma > \frac{12}{7}$ without damping and $\gamma > \frac{3}{2}$ with structure damping, utilizing the domain extension and regularization approximation. Moreover, via a modified relative entropy method in time-dependent domains, we give a rigorous justification of the incompressible inviscid limit of the compressible fluid-structure interaction problem with a flat reference geometry, in the regime of low Mach number, high Reynolds number, and well-prepared initial data. As a byproduct, with a fixed Reynolds number, we derive the incompressible limit without extra assumption.
        To the best of our knowledge, this is the first result concerning the singular limit problem for compressible fluids interacting with elastic structures.
    \end{abstract}
	
    \maketitle
	

\section{Introduction}
\label{sec:introduction}

This contribution considers a fluid-structure interaction problem, which arises in many fields, such as aerodynamics, hemodynamics, and engineering. More precisely, the model is a 3D-2D coupling problem, consisting of 3D compressible fluids and a 2D thin elastic structure.
\subsection{Model description}
Let $ T > 0 $. We denote by $ \Omega^w(t) $ the variable-in-time bounded domain in $ \bbr^3 $ with boundary $ \Gamma^w(t) $ {contains some part of the elastic structure depending on $w$ which will be specified later}. Let $ Q_T^w \coloneqq \bigcup_{t \in (0,T)} \Omega^w(t) \times \{t\} $, $ \Gamma_T^w \coloneqq \bigcup_{t \in (0,T)} \Gamma^w(t) \times \{t\} $ and $ \Gamma_T = \Gamma \times (0,T) $ be the corresponding space-time cylinders {with $\Gamma = \bbt^2$.} The motion of the fluid is governed by the compressible Navier--Stokes equation
\begin{subequations}
	\label{eqs:FSI-Model}
	\begin{alignat}{3}
		\pt \vr + \Div (\vr \bu) & = 0, && \tin Q_T^w, \\ 
		\pt (\vr \bu) + \Div (\vr \bu \otimes \bu) + \nabla p(\vr) & = \Div \bbs, && \tin Q_T^w,
	\end{alignat}
where $ (\bu, \vr) : Q_T^w \to \bbr^3 \times \bbr_{\geq 0} $ are the unknown velocity and density, $ p(\vr) = \vr^\gamma $ with $ \gamma > 1 $ denotes the barotropic pressure and 
\begin{equation*}
	\bbs(\nabla \bu) = \mu (\nabla \bu + \nabla \bu^\top) + \lambda \Div \bu \bbi
	= 2 \mu \bigg(\frac{\nabla \bu + \nabla \bu^\top}{2} - \frac{1}{3} \Div \bu \bbi\bigg) + \bigg(\lambda + \frac{2 \mu}{3}\bigg) \Div \bu \bbi,
\end{equation*}
with $ \mu > 0 $, $ \lambda + 2 \mu / 3 \geq 0 $ the shear viscosity and bulk viscosity.

In the following, we would like to differ from two different geometries, since there is a significant difference if we consider the Navier-slip boundary condition. The reason is that 
the geometric quantities, such as mean curvature, come into play as it is varying with the dynamic moving of the elastic boudary. Moreover, compared to the general geometry, one could gain more regularity of the structure displacement in a flat geometric setting (see Proposition \ref{prop:more-spatial-regularity} below). Overall, one can make use of a mapping $\bPhi_w$ depending on the displacement of the structure $w$ to describe the moving domain.

\begin{definition}[{Two geometries, Figure \ref{fig:two-geometries}}]
    Let $\Gamma = \bbt^2$, we say 
    \begin{itemize}
        \item $\Omega \in \bbone$, if $\Omega \subset \bbr^3$, and the boundary $\partial \Omega$ is parametrized by a $C^4$ injective mapping $\Phi: \Gamma \to \bbr^3$, i.e.,
        \begin{equation*}
            \partial \Omega
             = \left\{
                \Phi(x) \in \bbr^3: x \in \Gamma
               \right\},
        \end{equation*}
        \item $\Omega \in \bbtwo$, if $\Omega = \Gamma \times [0,1]$.
    \end{itemize}
\end{definition}

\begin{figure}[ht!]
	\centering
    \begin{tikzpicture}[line width=0.5pt,scale=0.9]

        \definecolor{aqua}{rgb}{0.0, 1.0, 1.0}
        
        \begin{scope}[shift={(-6,1)}]		
    		\draw[blue,fill=aqua!20] plot[smooth cycle, tension=0.7] coordinates{(0.1,0.1) (1,1.5) (1.8,2) (2.8,1.2) (3.6,1) (3.8,-1) (3,-1.35) (1,-2) (0.5,-1)};
    		\draw [dashed] plot[smooth cycle, tension=0.7] coordinates{(0,0) (0.2,1) (1,1.25) (2,2) (3,1.5) (3.75,0.6) (4,-0.3) (3.3,-0.7) (3,-1.75) (2,-2.05) (0.5,-1)};
        \end{scope}

		\begin{scope}[declare function={
			f(\x)=.2*sin(deg(\x*pi/2))+2;
		}]
					
            \fill[aqua!20] plot[domain = 0:4,samples=100](\x, {f(\x)}) -- (4,0) -- (0,0) -- (0,2);
            
		  	\draw[-stealth] (-.5,0) -- (4.5,0) node[right] {};
		  	\draw[-stealth] (0,-.25) -- (0,2.75);
    
		  	\draw[blue,domain = 0:4,samples=100] plot(\x, {f(\x)});

		  	\draw[dashed] (0,2) node[left] {1} -- (4,2);
            
		  	\node[below left] at (0,0) {$ 0 $};

            \node[below] at (4,0) {$ \Gamma \times \{0\} $};
		\end{scope}
    \end{tikzpicture}
    \caption{Sketch of two different geometries: general reference domain (left) and flat reference domain (right) {(the dashed lines denotes the reference domain, while the solid ones are deformed domain)}.}
    \label{fig:two-geometries}
\end{figure}
Let $ w : \Gamma \times (0,T) \to \bbr $ be the transverse displacement of the {elastic structure as part of the moving boundary illustrated above}. Then we characterize $\Gamma^w$ by an injective mapping $\Phi_w$, such that
\begin{equation*}
    \Phi_w(t,x) = \Phi(x) + w(t,x) \bn(\Phi(x)),
\end{equation*}
where 
\begin{equation*}
    \bn(y) = \frac{\ptial{1} \Phi(x) \times \ptial{2} \Phi(x)}{\abs{\ptial{1} \Phi(x) \times \ptial{2} \Phi(x)}}
\end{equation*}
is a smooth unit outer normal to $\partial \Omega$ at $y = \Phi(x)$, $x \in \Gamma$. Namely, for any $t > 0$,
\begin{equation*}
    \Gamma^w(t) = \left\{
        \Phi(x) + w(t,x) \bn(\Phi(x)): x \in \Gamma
    \right\}.
\end{equation*}
It is a well-known fact that from differential geometry (see e.g. {\cite[Theorem 10.19]{Lee2013}}) there exists $a_{\partial \Omega}$, $b_{\partial \Omega}$ such that for $w \in (a_{\partial \Omega}, b_{\partial \Omega})$, $\Phi_w(t,\cdot)$ is a bijective parametrization of the surface $\Gamma^w(t)$. Moreover, it defines a domain $\Omega^w(t)$ as its interior with boundary $\partial \Omega^w(t) = \Gamma^w(t)$. 
On the middle surface $\Gamma^w(t)$, one defines the weighted normal vector as
\begin{equation*}
    \bn^w(t,y) = \ptial{1} \Phi_w(t,x) \times \ptial{2} \Phi_w(t,x),
\end{equation*}
for $ y = \Phi_w(t,x) \in \Gamma^w(t) $ with $ x \in \Gamma $.

If $\Omega \in \bbtwo$, i.e., the reference geometry is flat, we define the middle surface of the structure from $\Gamma \times \{1\}$, as
\begin{equation*}
    \Gamma^w(t) = \left\{(x, x_3) \in \bbr^3: x_3 = w(t,x)+1, \ x \in \Gamma\right\},
\end{equation*}
for all $t > 0$.
Moreover, we denote by $\Gamma_0 \coloneqq \Gamma \times \{0\}$ the rigid bottom, which only appears accordingly with the flat reference domain. Then the domain $\Omega^w(t)$, with $\partial \Omega^w(t) = \Gamma^w(t) \cup \Gamma_0$ for $t > 0$, is of a periodic slab-like geometry.
In this case, $\Phi_w(t,x) = w(t,x) \be_3 + (x,1)^\top$ for $x \in \Gamma$, $\bn = \mathbf{e}_3$, the area Jacobian is $\sqrt{1 + \absm{\nabla w}^2}$, the dynamic (weighted) normal vector becomes $\bn^w \circ \Phi_w = (- \nabla w,1)^\top$, and then one can easily compute
    \begin{equation*}
        \bn \cdot \bn^w \circ \Phi_w = 1.
    \end{equation*}

Now we are able to describe the rest part of the model. The elastic structure is described by a linear fourth-order (bending) equation in Lagrangian coordinates as
\begin{alignat}{3}
	\label{eqs:w-equation}
	\pt^2 w + \Delta^2 w - \nu_s \Delta \pt w & = \mathcal{F}^w, && \ton \Gamma_T, 
\end{alignat}
and $ \mathcal{F}^w = \mathcal{F}^w(\vr, \bu) $ denotes the external stress from the fluid, which will be determined.

The \textit{kinematic coupling condition} is of the so-called \textit{Navier-slip} type,
\begin{alignat}{3}
    \label{eqs:velocity-boundary}
    (\bu \cdot \bn^w) \circ \Phi_w
    & = (\pt w \bn) \cdot (\bn^w \circ \Phi_w), && \ton \Gamma_T, \\
    \label{eqs:slip}
    (\bbs(\nabla \bu) \bn^w)_{\btau^w} \circ \Phi_w & = - \alpha (\bu - \pt w \bn \circ \Phi_w^{-1})_{\btau^w} \circ \Phi_w, && \ton \Gamma_T,
\end{alignat}
and
\begin{alignat}{3}
    \label{eqs:velocity-bottom}
    - \bu \cdot \be_3 & = 0, && \ton \Gamma_T^0, \\
    \label{eqs:slip-bottom}
    (\bbs(\nabla \bu) \be_3 + \alpha_0 \bu)_{\btau^0} & = 0, && \ton \Gamma_T^0,
\end{alignat}
where $\Gamma_T^0 \coloneqq \Gamma_0 \times (0,T)$, $ (\cdot)_{\btau^w} \coloneqq (\bbi - \bn^w \otimes \bn^w) \cdot $ and $ (\cdot)_{\btau^0} \coloneqq (\bbi - \be_3 \otimes \be_3) \cdot $ are the tangential projections with respect to the normal $ \bn^w $ on $ \Gamma^w $ and $\bn$ on $\Gamma_0$, respectively. The constant $\alpha, \alpha_0 \geq 0$ are the friction coefficient.
From \eqref{eqs:velocity-boundary} and \eqref{eqs:slip}, we are not requiring that fluid particles adhere at the boundary and there is a slip between the fluid and the elastic boundary.
Moreover, the \textit{dynamic coupling condition} reads
\begin{equation}
	\label{eqs:dynamicboundary}
 	\mathcal{F}^w(\vr, \bu) = - [(\bbs(\nabla \bu) - p(\vr) \bbi) \bn^w] \circ \Phi_w \cdot \bn, \ton \Gamma_T,
\end{equation}
which allows the current (Eulerian) normal traction from the fluid to act on the structure in the reference (Lagrangian) direction $ \bn $.
Finally, the system is subjected to the initial data
\begin{alignat}{4}
	\vr(0) & = \vr_0, \quad (\vr \bu)(0) && = \mathbf{m}_0, && \tin \Omega^{w_0}, \\
	w(0) & = w_0, \quad \pt w(0) && = w_1, && \ton \Gamma.
\end{alignat}

\end{subequations}

\begin{remark}
    \label{rem:discussion-boundary}
    If $\Omega \in \bbtwo$, we know $\bn \cdot \bn^w \circ \Phi_w = 1$. Then \eqref{eqs:velocity-boundary} simply reduces to
        \begin{equation}
            \label{eqs:pt_w=un}
            \pt w = (\bu \cdot \bn^w) \circ \Phi_w, \ \text{ on } \Gamma_T.
        \end{equation}
        This is a free boundary type kinematic condition, which, in our case, plays a significant role in further analysis (see Lemma \ref{lem:equi-integrability-with} and Proposition \ref{prop:more-spatial-regularity}), concerning the spatial regularity of $w$. 

		From the viewpoint of the fluid dynamics with free boundaries, \eqref{eqs:pt_w=un} is a usual condition along with flat reference geometry. In particular, we want to mention some related work \cite{CCS2007,CS2010,RT2019}, where general surface energies were considered and the boundary conditions were in fact reformulated as 
		\begin{equation}
            \label{eqs:free-boundary}
            (\bbs(\nabla \bu) - p(\vr) \bbi) \bn^w
		  = - \mathbf{t}, \quad
		  \bu \cdot \bn^w = \pt w \circ \Phi_w^{-1}, \ton \Gamma_T^w. 
		\end{equation}	
		for some traction vector $ \mathbf{t} $ (general surface energy). 
        This condition was also adopted in \cite{GGP2012} to prove the continuous dependence for weak solutions of a fluid-structure interaction problem in two dimensions. 
\end{remark}
\begin{remark}	
    Concerning the slip fact in applications, e.g., cardiovascular tissue and cell-seeded tissue constructs, Muha--\v{C}ani\'{c} \cite{MC2016} studied a fluid-structure interaction problem with \emph{Navier-slip} boundary conditions, allowing full displacements $ \bm{w}: (0,T) \times \Gamma \to \bbr^3 $, instead of only transverse displacement. The \emph{Navier-slip} boundary conditions there were 
		\begin{alignat*}{3}
		  (\bu \cdot \bn^w) \circ \Phi_w
		  & = \pt \bm{w} \cdot \bn^w \circ \Phi_w, && \ton \Gamma_T, \\
		  (\bbs(\nabla \bu) \bn^w)_{\btau^w} \circ \Phi_w & = \alpha (\pt \bm{w} - \bu)_{\btau^w} \circ \Phi_w, && \ton \Gamma_T.
		\end{alignat*}
        From above if we only assume the transverse displacement, i.e., substituting $\pt \bm w$ by $\pt w \bn$, one recovers exactly the boundary conditions \eqref{eqs:velocity-boundary} and \eqref{eqs:slip}.

\end{remark}

\subsection{Literature Review}
In the past few decades, mathematical analysis in fluid dynamics has developed significantly, leading to a wealth of literature. Particularly, in the last twenty years, extensive exploration has been conducted on the interaction between fluid and elastic structures.

The investigation of weak solutions to incompressible Navier--Stokes interacting with elastic structures (moving boundary) starts from the beginning of this century. The first result went back to Chambolle--Desjardins--Esteban--Grandmont \cite{CDEG2004}, where they proved the existence of weak solutions. The solid equation there is assumed to be of fourth order damped, which was later improved by Grandmont \cite{Grandmont2008} without damping. 
Since Muha--\v{C}ani\'c \cite{MC2013}, the existence of weak solutions was extended to the case with a one-dimensional cylindrical Koiter shell model. A time discretization method combined with an operator splitting method was proposed to prove the existence of weak solutions, inspired by numerical studies. See also \cite{MC2016} for the extension to the Navier-slip boundary condition. In the meanwhile, Lengeler--R\r{u}\v{z}i\v{c}ka \cite{LR2014} considered the general geometry via Hanzawa transformation, by which the weak solution was showed to be existed by a regularization procedure and a set-valued fixed-point argument. Recently, Muha--Schwarzacher \cite{MS2022} took a nonlinear Koiter shell model into account, for which they explored more spatial regularity of the displacement hidden in the model, such that the existence of weak solutions is available. We would like to mention that there was an interesting work \cite{BKS2023}, which took the variational viewpoint to prove the existence of weak solutions for the bulk fluid-solid interactions.
Concerning the well-posedness, Cheng--Coutand--Shkoller \cite{CCS2007} studied the Navier--Stokes equations interacting with a nonlinear elastic biofluid shell via higher-order energy methods, obtaining the existence and uniqueness of strong solutions to such model, also in Cheng--Shkoller \cite{CS2010}. The strong solutions of the (regularized) linear model were investigated by Beir\~ao da Veiga \cite{BdV2004}, Lequeurre \cite{Lequeurre2011,Lequeurre2013}, and later improved/extended in \cite{Breit2023,BMSS2023,DT2019,GH2016,GHL2019,KT2022,MRR2020}. In particular, Djebour--Takahashi \cite{DT2019} proved the strong well-posedness of an incompressible fluid-structure interaction problem with the Navier-slip boundary condition. Kukavica--Tuffaha \cite{KT2022} showed the local well-posedness of a model coupling the \textit{inviscid incompressible} fluids and a viscoelastic structure, which is exactly one of the motivations to study the singular limit problem of our system, {to rigorously derive the model in \cite{KT2022}}.

With the development of effective techniques for handling fluid-structure interaction problems, the work on compressible coupled problems arose. The first result of global weak solutions started from Breit--Schwarzacher \cite{BS2018}, concerning compressible isentropic Navier--Stokes equation interacting with a linear Koiter shell model, which provides a compressible counterpart of the results in Lengeler--R\r{u}\v{z}i\v{c}ka \cite{LR2014}. In \cite{BS2018}, they developed the theory in \cite{LR2014} with techniques for compressible fluids, e.g. Lions \cite{Lions1998}, Feireisl et al. \cite{Feireisl2004,FNP2001}. Later on, Breit--Schwarzacher \cite{BS2021} considered the Navier--Stokes--Fourier equations with a nonlinear Koiter shell. Recently, M\'{a}cha--Muha--Ne\v{c}asov\'{a}--Roy--Trifunovi\'{c} \cite{MMNRT2022} generalized the compressible fluid-structure interaction problems to the heat conducting case, via a domain extension technique and time discretization scheme. A compressible multicomponent fluid-structure interaction problem was addressed by Kalousek--Mitra--Ne\v{c}asov\'a \cite{KMN2023} with some extensions in this direction. The existence and uniqueness of strong solutions of compressible fluids with linear solid equations are referred to \cite{MRT2021,MT2021NARWA,Mitra2020}.

{In the study of compressible fluids, a particularly intriguing problem involves singular limits -- specifically, the incompressible and inviscid limit, where a compressible viscous fluid becomes effectively incompressible and inviscid. This occurs in the regime of low Mach number and high Reynolds number (with $ \mathrm{Ma} = {U_f}/{\sqrt{p_f/\varrho_f}} $ and $ \mathrm{Re} = {\varrho_f U_f L}/{\nu_f} $ respectively, see Section \ref{sec:rescaled-system} for more details), while the fluid dynamics are described by the incompressible Euler equations}. {Considering them simultaneously allows us to directly derive the target model. Moreover, studying the incompressible inviscid limit enhances our understanding of how solutions to the Navier–Stokes equations behave as viscosity tends to zero, and how compressible fluids transition to incompressibility, especially near solid boundaries.}
There are tons of literature in this direction and we only want to mention several typical ones.
Let us start with the seminal work of Klainerman and Majda \cite{KM1}, where the low Mach number limit for systems of equations describing a motion of fluid  is shown. Studying various types of singular limits allows us to eliminate unimportant or unwanted modes of motion as a consequence of scaling and asymptotic analysis. 
There were used two possible ways to introduce the Mach number into the system, which are different from the physical point of view, but  from the mathematical one - completely equivalent. The first approach considers a varying equation of state as well as the transport coefficients see the works of Ebin \cite{EB1},  Schochet \cite{SCH2}. 
The second way is to  use the dimensional analysis, see Klein \cite{Kl}.
The mathematical analysis of singular limits can be divided into two types of frame. In the frame of strong solutions we can refer to works of Gallagher \cite{Gallag}, Schochet \cite {SCH2}, Danchin \cite {Da} or Hoff \cite{Ho}. 
In the frame of weak solution, because of the seminal works of Lions \cite{Lions1998} and its extension by Feireisl et al. \cite{FNP2001}, we can mention e.g. works of Desjardins--Grenier \cite{DesGre}, Desjardins--Grenier--Lions--Masmoudi \cite{DGLM}, Feireisl--Novotn\' y  \cite{FN2017}. Moreover, the singular limits where we start with the weak solutions finally derive the strong solutions, using so-called weak strong uniqueness, see \cite{FN2017}. This provides us with more information on the convergence rate.

In contrast to the growing literature on (compressible) fluid-structure interaction problems, and singular limits in compressible fluids, the study of compressible fluid-structure interaction problems with slip boundary condition is unavailable prior to the weak solutions with no-slip condition \cite{BS2018,MMNRT2022,KMN2023}, and more interestingly how are the limiting process of compressible fluids happening, encountering elastic structure, is not clear at all. \textbf{To the best of our knowledge, this paper presents the first result about the weak solutions of compressible fluids interacting with elastic boundaries under the \textit{Navier-slip boundary condition}, and the first rigorous justification of the low Mach number and high Reynolds number limit for the fluid-structure interaction problems}.

\subsection{Main results}
Let us now state the main results of the article. The first theorem is the existence of weak solutions to the compressible fluid-structure interaction model with the Navier-slip boundary condition.
\begin{theorem}[Weak solution, proved in {Section \ref{sec:weak}}]
    \label{thm:weak}
    Assume that the initial data satisfies \eqref{eqs:initial}. Let either $ \nu_s > 0 $, $\gamma > \frac{3}{2}$ or $\Omega \in \bbtwo$, $\nu_s = 0$, $ \gamma > \frac{12}{7} $. Then there exists $ T > 0 $ and a weak solution to \eqref{eqs:FSI-Model} in the sense of Definition \ref{def:bounded-weak}. Moreover, either $ T = + \infty $ or the domain $ \Omega^w(t) $ degenerates in the sense of self-contact as $ t \rightarrow T^- $.
\end{theorem}
\begin{remark}
    Here different restrictions on elastic damping $\nu_s$ and adiabatic exponent $\gamma$ come essentially from geometric settings and boundary conditions. In the case of general reference domain with $\nu_s = 0$, it is still unknown about the existence of weak solutions to such models with \emph{Navier-slip boundary condition}. 
\end{remark}
\begin{remark}
	It is noticed that the weak-strong uniqueness principle for the weak solutions constructed in Theorem \ref{thm:weak} is addressed in a subsequent manuscript \cite{LMN-WSU} by authors.
\end{remark}

The second result concerns the singular limit problems in fluid dynamics. In particular, after certain rescalings, we justify the limit system rigorously in the regime of \textit{low Mach number and high Reynolds number}, i.e., incompressible inviscid limit. The following is an abbreviated version. 
For a precise description of the limit system and convergent rate, see Section \ref{sec:singular-limits} and Theorem \ref{thm:singular-limit-full}.
\begin{theorem}[Incompressible Inviscid limit]
    \label{thm:singular-limit}
    Assume that the initial data is well-prepared. Let $\Omega \in \bbtwo$, $ \nu_s > 0 $, $ \gamma > \frac{3}{2} $, and $ (\vr, \bu, w) $ be a weak solution to the compressible fluid-structure interaction problem \eqref{eqs:FSI-Model} constructed in Theorem \ref{thm:weak}, with $ w $ satisfying an additional regularity. Then up to a certain time $ T > 0 $, the weak solution $ (\vr, \bu, w) $ converges to the strong solution of an {incompressible} Euler-plate interaction system, as the Mach number goes to zero and the Reynolds number tends to infinity. Moreover, we have a certain convergent rate depending on the initial data. 
\end{theorem}
\begin{remark}
    In this result, we only consider the case $\nu_s > 0$. This is essential as in the limit passage we have to control the error term containing $\nabla \pt w$, cf. \eqref{eqs:F^epsilon-L^2}, whose regularity is not available from the \textit{a priori} energy estimate. On the other hand, with the inviscid limit, i.e., $\nu \to 0$, one cannot employ any regularity of $\bu$, which means we may not gain the regularity of $\pt w$ via the kinematic condition \eqref{eqs:pt_w=un}.
\end{remark}
\begin{remark}
    Note that we require an additional regularity assumption on the weak solution $w$, cf. \eqref{eqs:w-assumption-thm-3}. Indeed, it is necessary in the current framework. Since the domain depends on $w$, we later carry out a transform of the strong limit solution back to the weak domain, which causes much trouble as the \textit{a priori} regularity of $w$ is not sufficient to define a uniform Lipschitz domain. Consequently, a lot of techniques do not apply and no matter how smooth the strong solution is, it will eventually have a relatively low regularity comparable to $w$, cf. Corollary \ref{coro:regularity-loss}. 
    
    Moreover, here the incompressible inviscid limit is investigated, meaning the uniform regularity of the weak solution is much lower. This is another reason we require the additional assumption of $w$, cf. Remark \ref{rmk:not-avoid-assumtion}.

    In a related context of incompressible fluid-structure interaction problems, an extra Lipschitz regularity was assumed by Breit--Mensah--Schwarzacher--Su \cite{BMSS2023} to deal with the so-called \textit{Ladyzhenskaya--Prodi--Serrin} condition.
\end{remark}

\begin{remark}[Incompressible limit]
    \label{rem:incompressible-limit}
    If we only consider the low Mach number limit, the convergence holds as well. Namely, the compressible fluid-structure interaction system will converge to a system coupled the incompressible Navier--Stokes equations and a viscoelastic plate equation with slip boundary conditions. In particular, one can show that the limit passage is valid for $\gamma > 3$ without any regularity assumption. cf. Theorem \ref{coro:IncompressibleLimit}. 
\end{remark}


\subsection{Technical Discussions}
Concerning the \textit{existence of weak solutions}, it is the `bad' elastic boundary that causes much trouble in the approximate scheme. The idea would be either regularizing the elastic boundary (cf. \cite{LR2014,BS2018}), or approximating the boundary with a lager smooth domain (cf. \cite{KMN2023,MMNRT2022}) in the literature. In this work, we will complete the proof via the argument in \cite{MMNRT2022}, cf. Section \ref{sec:weak}. Note that the domain now is not uniform-in-time Lipschitz anymore, due to the low regularity of the displacement. Most of the steps of the proof (cf. Section \ref{sec:ideas-proof}) go through quite well compared to \cite{BS2018,MMNRT2022}, while the significant differences come out, encountering the \textit{Navier-slip} boundary conditions. 
Particularly, here it is not very convincing to show the existence of weak solutions in the general geometry without structure damping ($\nu_s = 0$), as we have to recover the nonlinear kinematic condition on the boundary. 
Instead, we differ from the type of reference geometries (flat or not) and include the structure viscosity to compensate for the boundary conditions. More precisely, in flat geometry, $\Omega = \Gamma \times [0,1]$, we observe that \eqref{eqs:velocity-boundary} is actually can be rewritten as $\pt w = (\bu \cdot \bn^w) \circ \Phi_w$, cf. \eqref{eqs:pt_w=un}. Then one gains more regularity of $\pt w$ by the multiplication of Sobolev--Slobodeckij spaces, as well as the trace regularity of $\bu$. This essentially facilitates further analysis for \textit{additional spatial regularities for $w$} and for \textit{equi-integrability of $\vr$ near the boundary}, cf. Proposition \ref{prop:more-spatial-regularity} and Remark \ref{rem:equi-integrability}. Note that the \textit{additional spatial regularity} is required here to recover the kinematic condition, which involves the deformed normal vector $\bn^w$, depending nonlinearly on $ \nabla w $. However, in general geometry, we are only able to use \eqref{eqs:velocity-boundary}, the \textit{Navier-slip} kinematic condition $(\bu \cdot \bn^w) \circ \Phi_w = (\pt w \bn) \cdot (\bn^w \circ \Phi_w)$, by which no more regularity for $\pt w$ can be achieved. Hence, to compensate for the loss, we require $\nu_s > 0$, i.e., $\pt w \in L^2(0,T;W^{1,2}(\Gamma))$, which plays the same role as above. In this case, to prove the equi-integrability of $\vr$ near the boundary (Lemma \ref{lem:equi-integrability-with}), one can use this regularity directly and the range of $\gamma$ can be relaxed to $\gamma > \frac{3}{2}$ (compared to the case $\nu_s = 0$ with $\gamma > \frac{12}{7}$).

In order to investigate the \textit{incompressible inviscid limit} (Mach number $\varepsilon \to 0$, reciprocal of the Reynolds number $\nu \to 0$, simultaneously), we employ a modified \textit{relative entropy method}, which is usually valid for fixed domains. In our case, one shall compare the weak solutions and the limit strong solution that are in fact defined in different domains. If we perform the usual transformation $ \bPsi = \bPhi_\eta \circ \bPhi_w^{-1} $ that simply composes two natural mappings, cf. \eqref{eqs:bPhi_w}, \eqref{eqs:mapping-Psi} below, the transformed velocity is not divergence-free anymore, while the limit velocity should be solenoidal doe to low Mach number. Moreover, there will be an error term on the boundary by transforming, as the \textit{Navier-slip boundary condition} depends on $ \bn^w $, which is not invariant along with the transform above. 
To overcome the issue, one needs to correct the transform both on the boundary and in the bulk.
Thus, a \textit{Piola transforma} based on $\bPsi$ is proposed, inspired by \cite{GGP2012,SS2022}. Precisely, for the transformed velocity $\bv$ one takes the form of $\bv = (\det\nabla \bPsi) \nabla \bPsi^{-1} (\tilde{\bv} \circ \bPsi)$ for $\tilde{\bv}$ defined in $\Omega^\eta$ with respect to $\eta$. In this case, the \textit{Piola identity} is valid $\Div \big((\det\nabla \bPsi) \nabla \bPsi^{-1}\big) = 0$ and hence the transformed variable is divergence-free (cf. \eqref{eqs:div-v=0}). In the meantime, in view of the transformation between outer normal vectors, it is exactly satisfied with a factor $(\det\nabla \bPsi) \nabla \bPsi^{-\top}$ (cf. \eqref{eqs:v-n-w}). 
Note that the high Reynolds number limit prevents us from cooking the argument up with the regularity of the velocity (of weak solution).
Though we consider a flat reference geometry here and hence $\pt w$ could inherit the trace regularity of $\bu$ formally, it is not applied as no spatial regularity of $\bu$ is available. {To close the estimates, it is demanded the high integrability of the error terms, cf. \eqref{eqs:rhovF^varepsilon}, which contains the Jacobian and the deformation gradient of $\bPsi$ by the Piola transform. In this case, one may require more regularity of the weak solutions, which essentially depends on the elastic damping $\nu_s > 0$ and additional regularities assumed, cf. Remark \ref{rmk:not-avoid-assumtion}.}

\subsection{Outline}
The rest parts are organized as follows. In Section \ref{sec:preli}, some notations, auxiliary results, and geometric settings are introduced. In Section \ref{sec:weak}, the concept of weak solutions and the main ideas of the proof for it are presented, as well as the crucial steps. Based on the weak solutions constructed in Section \ref{sec:weak}, the low Mach number and high Reynolds number limit of the system in flat geometry is justified for well-prepared initial data in Section \ref{sec:singular-limits}, via a modified relative entropy method.

\section{Preliminaries}
\label{sec:preli}

\subsection{Function spaces and auxiliary results.}

In this paper, we recall the usual Lebesgue and Sobolev spaces as $L^p$ and $W^{k,p}$ respectively with their corresponding induced norms, for $k \in \bbn$, $p \in [1,\infty]$. Particularly, $L^p = W^{0,p}$. If $p = 2$, we recognize $H^k = W^{k,2}$.

\begin{lemma}[Multiplication of fractional Sobolev spaces]
	\label{lem:multiplication}
	Let $ \Omega \subset \bbr^d $, $ d \in \bbn $ be a bounded domain with Lipschitz boundary. Let $ 1 \leq p < \infty $ and $ s_1, s_2 \geq s \geq 0 $. If $ u \in W^{s_1,p}(\Omega) $, $ v \in W^{s_2,p}(\Omega) $, then $ uv \in W^{s,p}(\Omega) $, for all $ s < s_1 + s_2 - \frac{d}{p} $. 
	Moreover,
	\begin{equation*}
		\norm{uv}_{W^{s,p}(\Omega)}
		\leq C \norm{u}_{W^{s_1,p}(\Omega)} \norm{v}_{W^{s_2,p}(\Omega)},
	\end{equation*}
	where $ C > 0 $ depends only on $ \Omega $.
\end{lemma}
\begin{proof}
	We refer to \cite[Theorem 7.4]{BH2021} for more general cases.
\end{proof}

\subsection{Geometric settings and regularities}
\label{sec:geometric-setting}

In this section, we introduce the geometric settings and corresponding regularities. Since the case of flat reference configuration is a special case of general geometry, here we only present the general one. One is also referred to e.g., \cite{BS2018,KMN2023,LR2014,MMNRT2022}. We define the tubular neighborhood of $\partial \Omega$ as
\begin{equation*}
    \cN_a^b = \left\{
            y \in \bbr^3 : y = \Phi(x) + \bn(\Phi(x)) z, x \in \Gamma, z \in (a_{\partial \Omega}, b_{\partial \Omega})
        \right\}.
\end{equation*}
The projection $\pi : \cN_a^b \to \partial \Omega$ denotes the mapping that assigns to each $x \in \cN_a^b$ a unique $\pi(x) \in \partial \Omega$ such that there is $z \in (a_{\partial \Omega}, b_{\partial \Omega})$ satisfying
\begin{equation*}
    x - \pi(x) = \bn(\pi(x)) z,
\end{equation*}
and the signed distance function is defined as $d: \cN_a^b \to (a_{\partial \Omega}, b_{\partial \Omega}) $ such that for $x \in \cN_a^b$,
\begin{equation*}
    d(x) = (x - \pi(x)) \cdot \bn(\pi(x)).
\end{equation*}
For a point $x \in \bbr^3$, we introduce a similar distance function {$\mathfrak{d}: \bbr^3 \to \bbr$} to the boundary $\partial \Omega$ fulfilling
\begin{equation*}
    \mathfrak{d}(x, \partial \Omega)
    = \left\{
        \begin{aligned}
            & - \dist(x, \partial \Omega) \quad && \text{if } x \in \Bar{\Omega}, \\
            & \dist(x, \partial \Omega) \quad && \text{if } x \in \bbr^3 \setminus \Omega.
        \end{aligned}
    \right.
\end{equation*}
Note that $d$ and $\mathfrak{d}$ agree in the tubular $\cN_a^b$. Since the mapping $\Phi$ is assumed to be $C^4(\Gamma)$, it is known that the projection $\pi$ is well-defiend and possesses the $C^3$-regularity, and the distance function $d$ assures the $C^4$-regularity in a neighborhood of $\partial \Omega$ containing $\cN_a^b$. Let $w: [0,T] \times \Gamma \to \bbr$ be a given displacement function with $a_{\partial \Omega} < m < w < M < b_{\partial \Omega}$ with $m, M > 0$ being fixed.

Now let us define a flow map $\bPhi_w: [0,T] \times \bbr^3 \to \bbr^3$ such that
\begin{equation}
    \label{eqs:bPhi_w}
    \bPhi_w(t,x) = x + f_\Gamma(\mathfrak{d}(x)) w(t, \inv{\Phi}(\pi(x))) \bn(\pi(x)),
\end{equation}
where the cut-off function $f_\Gamma \in C_c^\infty(\bbr; [0,1])$ is given by $ f_\Gamma(s) = (f * \omega_\alpha)(s) $
with a standard mollifying kernel $\omega_\alpha$ supporting in $(-\alpha, \alpha)$ for $0 < \alpha < \onehalf \min \{m' - m'', M'' - M'\}$,
and a linear function $f \in W^{1, \infty}(\bbr; [0,1])$ such that $ f = 1 $ in $ (m'', m'' - m'] $ and $ f = 0 $ in $ (- \infty, m''] \cap (M'', \infty) $, as well as $ f' \in \left[- \frac{1}{M'}, - \frac{1}{m'}\right] $. The constructure can be found in \cite[Section 2.1]{KMN2023}.
Here we suppose that
\begin{equation*}
    a_{\partial \Omega} < m'' < m' < m < 0 < M < M' < M'' < b_{\partial \Omega}.
\end{equation*}
Then for $X \in \cN_b^a$, one can rewrite the map $\bPhi_w$ with the help of $X - \pi(X) = \bn(\pi(X)) d(X)$, as
\begin{equation*}
    \bPhi_w(t,X) = (1 - f_\Gamma(d(X)))X + f_\Gamma(d(X)) \big(\pi(X) + \big(d(X) + w(t, \inv{\Phi}(\pi(X)))\big) \bn(\pi(X))\big).
\end{equation*}
Then the corresponding inverse $\inv{\bPhi_w}$ follows
\begin{equation*}
    \inv{\bPhi_w}(t,x) = (1 - f_\Gamma(d(x)))x + f_\Gamma(d(x)) \big(\pi(x) + \big(d(x) - w(t, \inv{\Phi}(\pi(x)))\big) \bn(\pi(x))\big),
\end{equation*}
which implies a mapping $\inv{\bPhi_w} : [0,T] \times \bbr^3 \to \bbr^3$ such that
\begin{equation}
    \label{eqs:bPhi_w^-1}
    \inv{\bPhi_w}(t,x) = x - f_\Gamma(\mathfrak{d}(x)) w(t, \inv{\Phi}(\pi(x))) \bn(\pi(x)).
\end{equation}
In view of the explicit expressions of $\bPhi_w$ and $\inv{\bPhi_w}$, we know that they roughly inherit the same regularity from $w$.

\begin{remark}
    By the constructions above, particularly one has the restrictions
    \begin{alignat*}{3}
        \bPhi_w(t,X) & = \Phi_w(t,\inv{\Phi}(X)) && \ \text{ for all } X \in \partial \Omega, \\
        \bPhi_w(t,X) \circ \Phi(z) & = \Phi_w(t,z) && \ \text{ for all } z \in \Gamma. 
    \end{alignat*}
    With a slight abuse of notations when there is no danger of confusion, we will write $f \circ \bPhi_w \circ \Phi$ as $f \circ \bPhi_w$ for $f \in \Gamma^w$.
\end{remark}


Now we define the generalized Bochner spaces defined on time-dependent domains, cf. \cite{LR2014,BS2018,KMN2023}, and record several necessary lemmas. For $T > 0$, $\eta \in C([0,T] \times \Gamma)$ we define the function spaces on variable domains in the following way for $p,r \in [1,\infty]$
\begin{align*}
    L^p(0,T;L^r(\Omega^w(t)))
    & \coloneqq \big\{
        v \in L^1(Q_T^w):
        v(t) \in L^r(\Omega^w(t)) \text{ for a.e. } t \in (0,T), \\
        & \phantom{\coloneqq \big\{v \in L^1(Q_T^w):}
        \normm{v(t)}_{L^r(\Omega^w(t))} \in L^p(0,T)
    \big\}, \\
    L^p(0,T;W^{1,r}(\Omega^w(t))) 
    & \coloneqq \big\{
        v \in L^p(0,T;L^r(\Omega^w(t))):
        \nabla v \in L^p(0,T;L^r(\Omega^w(t)))
    \big\}, \\
    W^{1,p}(0,T;W^{1,r}(\Omega^w(t))) 
    & \coloneqq \big\{
        v \in L^p(0,T;L^r(\Omega^w(t))):
        \pt v \in L^p(0,T;W^{1,r}(\Omega^w(t))) 
    \big\}.
\end{align*}
Throughout the context, we denote by
\begin{gather*}
	K_t^{s,p} M_y^{r,q} \coloneqq K^{s,p}(0,t; M^{r,q}(\mathcal{O})), \quad s,r \in \bbr, 1 \leq p,q \leq \infty, K,M \in \{W,H\},
\end{gather*}
where either $ \mathcal{O} = \Omega $ if $ y = x $, or $ \mathcal{O} = \Omega^w(t) $ if $ y = w $, or $ \mathcal{O} = \Gamma $ if $ y = \Gamma $. Moreover, $ K_{ty}^{s,p} \coloneqq K^{s,p}(0,t; K^{s,p}(\mathcal{O})) $.

The next lemma quantifies the loss in the regularity for transformations.
\begin{lemma}
    \label{lem:embedding}
    Let $ 1 \leq p \leq \infty $, $1 < q \leq \infty$, and $ w \in L_t^\infty H_\Gamma^2 \cap W_t^{1,\infty} L_\Gamma^2 $ with $ \normm{w}_{L_t^\infty H_\Gamma^2 \cap W_t^{1,\infty} L_\Gamma^2} < \kappa $. Then the mapping $\bv \mapsto \bv \circ \bPhi_w$ is continuous from $L^p(0,T;L^q(\Omega^w(t)))$ to $L^p(0,T;L^r(\Omega))$ and from $L^p(0,T;W^{1,q}(\Omega^w(t)))$ to $L^p(0,T;W^{1,r}(\Omega))$ for any $1 \leq  r < q$. The same one holds from $\bPhi_w^{-1}$.
\end{lemma}
\begin{proof}
    For the proof, see \cite[Lemma 2.2]{KMN2023}.
\end{proof}

Concerning the trace on $ \Gamma $ with respect to $ \bPhi_w $, we have the following Lemma, which is referred to as \cite[Lemma 2.3]{KMN2023}.
\begin{lemma}
	\label{lem:trace-Phi_w}
	Let $ 1 \leq p \leq \infty $, $1 < q \leq \infty$, and $ w \in L_t^\infty H_\Gamma^2 \cap W_t^{1,\infty} L_\Gamma^2 $ with $ \norm{w}_{L_t^\infty H_\Gamma^2 \cap W_t^{1,\infty} L_\Gamma^2} < \kappa $. Then the linear mapping $ \tr_{\Gamma^w} : \bv \mapsto \rvm{(\bv \circ \bPhi_w)}_{\partial \Omega} $ is well defined and continuous from $ L^p(0,T;W^{1,q}(\Omega^w(t))) $ to $ L^p(0,T;L^r(\Gamma)) $ for all $r \in (1, \frac{2q}{3-q})$, respectively from $L^p(0,T;W^{1,q}(\Omega^w(t)))$ to $L^p(0,T;W^{1-\frac{1}{r},r}(\Gamma))$ for any $1 \leq r < q$.
\end{lemma}

Now we are into the regularity of the map. Recall that we have {\textit{a priori}}
\begin{equation*}
	w \in L^\infty(0,T; H^2(\Gamma)) \cap W^{1,\infty}(0,T; L^2(\Gamma)).
\end{equation*}
Then as explained above, $ \bPhi_w $ is endowed with the same regularity of $ w $, i.e.,
\begin{equation*}
	\bPhi_w \in L^\infty(0,T; H^2(\Omega)) \cap W^{1,\infty}(0,T; L^2(\Omega)),
\end{equation*}
and for a.e. $ t \in (0,T) $,
\begin{equation}
	\label{eqs:estimate-Phi}
	\norm{\bPhi_w(t)}_{H_x^2} + \norm{\bPhi_w(t)}_{L_x^2}
	\leq C \norm{w(t)}_X,
\end{equation}
where $ \normm{w(t)}_X \coloneqq \normm{\pt w(t)}_{L_\Gamma^2} + \normm{\Delta w(t)}_{L_\Gamma^2} $. Since $ H^1(\Gamma) \hookrightarrow L^p(\Gamma) $ for all $1 \leq p < \infty$ by the two dimensional Sobolev embedding, we know that
\begin{equation}
	\label{eqs:nabla-Phi_w-1}
	\norm{\nabla \bPhi_w(t)}_{L_x^p} + \norm{J_w(t)}_{L_x^p} + \norm{\cA_w(t)}_{L_x^p}
	\leq C \norm{w(t)}_X \text{ for all } 1 \leq p < \infty,
\end{equation}
as well as $ \inv{\bPhi_w}, \pt \inv{\bPhi_w}, \inv{J_w} $, possessing the same regularity classes and estimates in $ \Omega^w(t) $. 

If $ \nu_s > 0 $, one has $ \pt w \in L^2(0,T; H^1(\Gamma)) \hookrightarrow L^2(0,T; L^p(\Gamma)) $. Then
\begin{equation*}
    \pt \bPhi_w \in L^2(0,T; L^p(\Omega)) \text{ for all } 1 \leq p < \infty
\end{equation*}
and for a.e. $ t \in (0,T) $,
\begin{equation}
    \label{eqs:pt-Phi_w}
    \int_0^t \norm{\pt \bPhi_w(s)}_{L_x^p}^2 \,\d s \leq C \int_0^t \norm{\nabla \pt w(s)}_{L_\Gamma^p}^2 \,\d s
    \leq C(\nu_s^{-1},w) \text{ for all } 1 \leq p < \infty.
\end{equation}
Here for $ J_w $, one can employ the Hadamard's inequality (see e.g. \cite[Equation (A.2.23)]{KR2019})
\begin{equation*}
	\abs{\det \mathbf{B}} \leq d^{d/2} \abs{\mathbf{B}}^d
\end{equation*} 
for the determinant of a bounded matrix $ \mathbf{B} \in \bbr^{d \times d} $ in the sense of Frobenius.

\section{Existence of Weak Solutions}
\label{sec:weak}
In this section, we are going to investigate the existence of weak solutions to \eqref{eqs:FSI-Model}. Namely, we will prove Theorem \ref{thm:weak} following the scheme developed in \cite{MMNRT2022}. More precisely, rough ideas for the proof will be presented first in Section \ref{sec:ideas-proof}, while the approximate (extended) system and uniform estimates are shown in Sections \ref{sec:extended-problem} and \ref{sec:uniform-estimates} respectively. Then the proof of existence is demonstrated step-by-step in Section \ref{sec:weak-ext}. 

Before the main argument, let us first give the necessary assumption for the initial data.
\begin{assumption}
    The initial data $ (\vr_0, \bu_0, w_0, w_{0,1}) $ is supposed to fulfill
	\begin{equation}
		\label{eqs:initial}
		\begin{gathered}
			\vr_0 \in L^\gamma(\Omega^{w_0}), \quad \vr_0 \geq 0, \quad \vr_0 \not\equiv 0, \quad {\vr_0}_{|\bbr^3 \backslash \Omega^{w_0}} = 0, \\
			\vr_0 > 0 \text{ in } \{x \in \Omega^{w_0}: \mathbf{m}_0(x) > 0 \}, \quad \frac{\mathbf{m}_0^2}{\vr_0} \in L^1(\Omega^{w_0}), \\
			w_0 \in H^2(\Gamma), \quad w_{0,1} \in L^2(\Gamma), \\
			E_0 \coloneqq \int_{\Omega^{w_0}} \left( \frac{1}{2\vr_0} \abs{\mathbf{m}_0}^2 + \frac{p(\vr_0)}{\gamma - 1} \right) \dx
			+ \int_\Gamma \left( \onehalf \abs{w_{0,1}}^2 + \onehalf \abs{\Delta w_0}^2 \right) \dH^2 < \infty,
		\end{gathered}
	\end{equation}
\end{assumption}
Now we record the definition of \textit{bounded energy} weak solutions to \eqref{eqs:FSI-Model}.
\begin{definition}
	\label{def:bounded-weak}
	We say that $ (\vr, \bu, w) $ is a \textbf{(bounded energy) weak solution} to \eqref{eqs:FSI-Model} with initial data $ (\vr_0, \mathbf{m}_0, w_0, w_{0,1}) $,
	if the following conditions hold:
	\begin{enumerate}
		\item $ (\vr, \bu, w) $ satisfies 
		\begin{gather}
			\label{eqs:weak-rho}
			\vr \geq 0, \  \vr \in C_w([0,T]; L^\gamma (\bbr^3)) \cap L^q_{loc}([0,T] \times \Omega^w(t)), \\
			\label{eqs:weak-u}
			\bu \in L^2(0,T; W^{1,q}(\Omega^w(t))), \  q < 2, \quad
			\vr \abs{\bu}^2 \in L^\infty(0,T; L^1(\bbr^3)), \\
			\label{eqs:weak-w}
			w \in L^\infty(0,T; H^2(\Gamma)) \cap H^1(0,T; H^1(\Gamma)), \quad \pt w \in L^\infty(0,T; L^2(\Gamma)).
		\end{gather}
		\item The coupling condition $ (\pt w \bn) \cdot (\bn^w \circ \bPhi_w) = \mathrm{tr}_{\Gamma^w} \bu \cdot (\bn^w \circ \bPhi_w) $ {holds} on $ \Gamma_T $, where the trace operator $ \tr_{\Gamma^w} $ is defined in Lemma \ref{lem:trace-Phi_w}.
		\item The renormalized continuity equation
		\begin{equation}
			\label{eqs:weak-renormal}
			\begin{aligned}
				\int_{Q^w_t} \big( b(\vr) (\pt \vphi + \bu \cdot \nabla \vphi) 
				+ ((b(\vr) & - b'(\vr) \vr) \Div \bu \vphi) \big) \dxdtau \\
				& = \int_{\Omega^{w}(t)} b(\vr) \vphi (t, \cdot) \dx - \int_{\Omega^{w_0}} b(\vr_0) \vphi (0, \cdot) \dx,
			\end{aligned}
		\end{equation}
		holds for all $t \in [0,T]$ and $ \vphi \in C^\infty([0,T] \times \bbr^3) $, and any $ b(\cdot) \in C^1[0,\infty) $, $ b(0) = 0 $, $ b'(r) = 0 $ for large $ r $.
		\item The coupled momentum equation is satisfied in the sense of
		\begin{equation}
			\label{eqs:weak-momentum}
			\begin{aligned}
				& \int_{Q^w_t} \Big( 
                    \vr \bu \cdot \pt \bphi 
    			+ (\vr \bu \otimes \bu) : \nabla \bphi
    			+ p(\vr) \Div \bphi - \bbs(\nabla \bu) : \nabla \bphi
                \Big) \dxdtau \\
				& \qquad + \int_{\Gamma_t} \big(
                \pt w \pt \psi - \Delta w \Delta \psi - \nu_s \nabla \pt w \cdot \nabla \psi
                \big) \dH^2 \d t \\
				& \qquad + \alpha \int_{\Gamma^w_T}  \big( (\bu - \pt w \bn \circ \inv{\bPhi_w})_{\btau^w} \big) \cdot \big( (\bphi - \psi \bn \circ \inv{\bPhi_w})_{\btau^w} \big) \dH^2 \d t \\
				& = \int_{\Omega^{w}(t)} (\vr \bu) \cdot \bphi(t, \cdot) \dx + \int_{\Gamma} \pt w \psi(t, \cdot) \dH^2 \\
                & \qquad - \int_{\Omega^{w_0}} \mathbf{m}_0 \cdot \bphi(0, \cdot) \dx - \int_{\Gamma} w_{0,1} \psi(0, \cdot) \dH^2,
			\end{aligned}
		\end{equation}
		for all $t \in [0,T]$, $ \bphi \in C^\infty(\Bar{Q_T^w}) $ and $ \psi \in C^\infty(\Bar{\Gamma_T}) $ satisfying $ (\psi \bn) \cdot (\bn^w \circ \bPhi_w) = \tr_{\Gamma^w} \bphi \cdot (\bn^w \circ \bPhi_w) $ on $ \Gamma_T $. 
		\item The energy inequality
		\begin{equation}
			\label{eqs:weak-energy-inequality}
			\begin{aligned}
				& \int_{\Omega^w(t)} \left( \onehalf \vr \abs{\bu}^2 + \frac{p(\vr)}{\gamma - 1} \right)(t) \dx
				+ \int_{\Gamma} \left( \onehalf \abs{\pt w}^2 + \onehalf \abs{\Delta w}^2 \right)(t) \dH^2 \\
				& \qquad + \int_{0}^{t} \int_{\Omega^w(\tau)} \bbs(\nabla \bu) : \nabla \bu \dxdtau \\
				& \qquad + \alpha \int_{0}^{t} \int_{\Gamma^w(\tau)} \abs{(\bu - \pt w \bn \circ \inv{\bPhi_w})_{\btau^w}}^2 \dH^2 \d \tau 
				+ \nu_s \int_{0}^{t} \int_{\Gamma} \abs{\nabla \pt w}^2 \dH^2 \d \tau \\
				& \qquad \quad \leq \int_{\Omega^{w_0}} \left( \frac{1}{2 \vr_0} \abs{\mathbf{m}_0}^2 + \frac{p(\vr_0)}{\gamma - 1} \right) \dx
				+ \int_{\Gamma} \left( \onehalf \abs{w_{0,1}}^2 + \onehalf \abs{\Delta w_0}^2 \right) \dH^2,
			\end{aligned}
		\end{equation}
		holds for all $ t \in [0, T] $. 
	\end{enumerate}
\end{definition}

\begin{remark}
    Here \eqref{eqs:weak-momentum} and \eqref{eqs:weak-energy-inequality} are based on the case $\Omega \in \bbone$. If the reference geometry is flat, there will be no essential difference, except for the boundary data on $\Gamma_0$, namely, we would additionally have $\tr_{\Gamma_0} \bphi \cdot \be_3 = 0$, and
    \begin{equation*}
        \alpha_0 \int_{\Gamma_T^0} \bu_{\btau^0} \cdot \bphi_{\btau^0} \dH^2 \dt \text{ and }
        \alpha_0 \int_0^t \int_{\Gamma_0} \abs{\bu_{\btau^0}}^2 \dH^2 \dtau    
    \end{equation*}
    will appear on the left-hand side of \eqref{eqs:weak-momentum} and \eqref{eqs:weak-energy-inequality}, respectively.
\end{remark}

\begin{remark}
    In the weak formulation \eqref{eqs:weak-momentum}, one assumes a pair of test functions $(\bphi, \psi) \in C^\infty(\Bar{Q_T^w}) \times C^\infty(\Bar{\Gamma_T}) $ satisfying $ (\psi \bn) \cdot (\bn^w \circ \bPhi_w) = \tr_{\Gamma^w} \bphi \cdot (\bn^w \circ \bPhi_w) $ on $ \Gamma_T $. The choice of smooth test functions is compatible with the continuity of normal velocity in the presence of $\bn^w$, due to the Navier-slip boundary condition. We comment that if one considers the no-slip boundary, i.e., continuity of velocities $ \psi \bn = \tr_{\Gamma^w} \bphi $, the regularity class of $(\bphi, \psi)$ are not matched due to the transformation $\bPhi_w$. The same holds for the flat reference domain case as $\bn \cdot (\bn^w \circ \bPhi_w) = 1$, i.e., $ \psi = \tr_{\Gamma^w} \bphi \cdot (\bn^w \circ \bPhi_w) $.
    Instead, we should consider $(\bphi, \psi) \in C^\infty(\Bar{Q_T^w}) \times (L^\infty(0,T; H^2(\Gamma)) \cap W^{1,\infty}(0,T; L^2(\Gamma))) $. See e.g. Kalousek--Mitra--Ne\v{c}asov\'{a} \cite[Definition 1.1]{KMN2023} for the multicomponent fluid-structure interaction problem with no-slip boundary conditions. 
\end{remark}


\subsection{Rough ideas of the proof}
\label{sec:ideas-proof}
To show Theorem \ref{thm:weak}, we will follow the ideas and techniques introduced in \cite{MMNRT2022}, where a Navier--Stokes--Fourier system coupled with a damping shell and heat exchange was considered. The strategy there is robust and can be adapted to our problem, with the corresponding weak solution argument of compressible barotropic Navier--Stokes equation \cite{NS2004}. To avoid lengthy repetitions of the arguments analogous to \cite{KMN2023,MMNRT2022} and make the manuscript a reasonable size, in the following we will conclude the main steps and emphasize the main points coming from the compressible Navier--Stokes equation and the Navier--slip boundary condition. The sketch of the proof reads as follows:
\begin{enumerate}
    \item \textbf{Existence of approximate solutions}. 
    In Section \ref{sec:extended-problem}, we extend the fluid domain $\Omega^w(t)$ to a large domain $ B $ such that $ \Bar{\Omega^w(t)} \subset B \coloneqq \{\abs{x} < 2R\} $ for any fixed $t$ and $R$, while an approximation parameter $ \kappa > 0 $ is employed to extend the viscosities $ \mu $, $ \lambda $ to $ \mu_\kappa $, $ \lambda_\kappa $ satisfying $ \mu_\kappa, \lambda_\kappa \in C_c^\infty([0,T] \times \bbr^3) $, $ \mu_\kappa(t, \cdot ) _{|\Omega^w(t)} = \mu $, $ \lambda_\kappa(t, \cdot ) _{|\Omega^w(t)} = \lambda $ and $ \mu_\kappa, \lambda_\kappa \rightarrow 0 $ as $ \kappa \rightarrow 0 $. Moreover, as the usual techniques in the framework of compressible Navier--Stokes equation, see e.g. \cite{FN2017,NS2004}, the artificial (regularized) pressure $ \delta \vr^\beta $ with $ \delta, \beta > 0 $ is employed to approximate the pressure to get higher integrability of the density $ \vr $. The main goal here is to prove the existence of weak solutions to the approximate system, see Theorem \ref{thm:weak-ext}. The proof will {rely} on a combination of a time discretization and a operator splitting method, where one solves the subproblems respectively by standard methods, see \cite{KMN2023,MMNRT2022,MC2013,MC2016}.
    Note that the boundary is missing in the fluid subproblem. So we penalize the kinematic boundary condtion as 
    \begin{equation*}
        \frac{1}{\kappa} \int_0^T \int_{\Gamma^w(t)} \big( (\bu - \pt w \bn \circ \inv{\bPhi_w}) \cdot \bn^w \big) \big( (\bphi - \psi \bn \circ \inv{\bPhi_w}) \cdot \bn^w \big) \dH^2 \dt,
    \end{equation*}
    for $ \kappa > 0 $, initiated by \cite{FKNNS2013}. Further, we apply {a fundamental lemma} which roughly says that if the initial density is extended by zero to a large domain $B\setminus \Omega^{w_0}$, also the density is zero on $B \setminus \Omega^w(t)$ in the case of the smooth domain and sufficiently regular density, see \cite{KMN2023,MMNRT2022}.
    \item \textbf{Uniform bounds and equi-integrability}. Once solving the approximate system, one may derive the uniform both in $ \kappa $ and $ \delta $ estimates for $ (\vr^\delta, \bu^\delta, w^\delta) $ (only labeled with $ \delta $ for simplicity). Depending on the presence of structure damping, i.e., $ \nu_s > 0 $ or $ \nu_s = 0 $, {we derive uniform estimates with different restrictions}. In particular, the equi-integrability of the density $ \vr^\delta $ in this case replies strongly on the adiabatic exponent, due to the \textit{Navier-slip boundary condition}, cf. Lemma \ref{lem:equi-integrability-with} and Remark \ref{rem:equi-integrability}.
	\item \textbf{Passing to the limit as $ \kappa, \delta \rightarrow 0 $}.
	Based on the uniform (in $ \delta $ and $ \kappa $) bounds, we are hence able to extract subsequences to justify the limit functions, as $ \kappa = \delta \rightarrow 0 $, see Section \ref{sec:weak-ext}, which finally completes the proof of Theorem \ref{thm:weak}. Note that again one shall distinguish different reference geometries and elastic damping $ \nu_s \geq 0 $, to recover boundary data, cf. Proposition \ref{prop:more-spatial-regularity} and Section \ref{sec:recover-boundary}, resulting from the presence of the \textit{Navier-slip boundary condition}.
\end{enumerate} 

In the following, we will discuss some necessary details about each step above. For the full argument, we refer to \cite{KMN2023,MMNRT2022}.

\subsection{Extended problem with artificial pressure}
\label{sec:extended-problem}
Following the extension argument in \cite{FKNNS2013,KMN2023,MMNRT2022}, for a fixed $0 < \kappa \ll 1$ we approximate the viscosity coefficients $\mu, \lambda$ by
\begin{equation*}
    \mu_\kappa \coloneqq \chi_\kappa \mu, \quad 
    \lambda_\kappa \coloneqq \chi_\kappa \lambda,
\end{equation*}
where the function $\chi_\kappa \in C_c^\infty(\bbr^3 \times [0,T]; [0,1])$ fulfills
\begin{gather*}
    0 < \kappa < \chi_\kappa \leq 1 \text{ in } B \times [0,T], \\
    {\chi_\kappa}_{|\Omega^w(t)} = 1 \text{ for all } t \in [0,T], \\
    \norm{\chi_\kappa}_{L^p((B \times (0,T)) \setminus Q_T^w)} \leq C \kappa \text{ for some } p \geq 1, \\
    \text{the mapping } w \mapsto \chi_\kappa \text{ is Lipschitz}.
\end{gather*}

Moreover, for $\delta > 0$ we introduce an approximate pressure $p_\delta(\vr)$ with an artificial pressure $\delta \vr^\beta$,
\begin{equation*}
    p_\delta(\vr) = \vr^\gamma + \delta \vr^\beta.
\end{equation*}
The initial data $\vr_0$ and $\mathbf{m}_0$ are extended to $\vr_{0,\delta}$, $(\vr \bu)_{0,\delta}$ in $B$ such that
\begin{equation}
    \label{eqs:initial-reg}
    \begin{gathered}
        \vr_{0,\delta} \geq 0, \ 
        \vr_{0,\delta} \not\equiv 0, \ 
        {\vr_{0,\delta}}_{|\bbr^3 \setminus \Omega^{w_0}} = 0, \ 
        \int_B \left(\vr_{0,\delta}^\gamma + \delta \vr_{0,\delta}^\beta\right) \dx \leq C, \\
        \vr_{0,\delta} \to \vr_0 \text{ in } L^{\gamma} (\Omega^{w_0}), \ 
        (\vr \bu)_{0,\delta} \to \mathbf{m}_0 \text{ in } L^1(\Omega^{w_0}), \text{ as } \delta \to 0, \\
        \int_B \frac{1}{\vr_{0,\delta}} \abs{(\vr \bu)_{0,\delta}}^2 
        \dx \leq C, \ 
        \int_B \frac{\abs{(\vr \bu)_{0,\delta}}^2}{\vr_{0,\delta}} \dx 
        \to \int_{\Omega^{w_0}} \frac{\abs{\mathbf{m}_0}^2}{\vr_0} \dx, \text{ as } \delta \to 0.
    \end{gathered}
\end{equation}

Now we give the definition of weak solutions to the extended problem with pressure regularization as follows.
\begin{definition}
    \label{def:weak-extended}
    We say that $ (\vr, \bu, w) $ is a weak solution to the extended compressible Navier--Stokes equations interacting with a viscoelastic structure with the Navier-slip boundary in $B$
	if the following conditions hold:
	\begin{enumerate}
		\item $ (\vr, \bu, w) $ satisfies
		\begin{gather}
			\label{eqs:weak-rho-ext}
			\vr \geq 0, \text{ a.e. in } B, \  \vr \in L^{\infty}(0,T; L^\beta (B)), \ p_\delta(\vr) \in L^1((0,T) \times B), \\
			\label{eqs:weak-u-ext}
			\bu \in L^2(0,T; H_0^1(B)), \  
			\vr \abs{\bu}^2 \in L^\infty(0,T; L^1(B)), \\
			\label{eqs:weak-w-ext}
			w \in L^\infty(0,T; H^2(\Gamma)) \cap H^1(0,T; H^1(\Gamma)), \ \pt w \in L^\infty(0,T; L^2(\Gamma)).
		\end{gather}
		\item The renormalized continuity equation
		\begin{equation}
			\label{eqs:weak-renormal-ext}
			\begin{aligned}
				\int_0^t \int_B \vr \cB (\vr)\vr (\pt \vphi & + \bu \cdot \nabla \vphi) \dxdtau 
				- \int_0^t \int_B b (\vr) \Div \bu \vphi (t,\cdot) \dxdtau\\
				& = \int_B \vr(t,\cdot) \cB (\vr(t,\cdot)) \vphi (t,\cdot) \dx - \int_B \vr_{0,\delta} \cB (\vr_{0,\delta}) \vphi (0,\cdot) \dx,
			\end{aligned}
		\end{equation}
		holds for all $t \in [0,T]$, $ \vphi \in C^\infty([0,T] \times \bbr^3) $, and $b \in L^\infty(0,\infty) \cap C[0,\infty)$ such that $b(0)=0$ with $\cB(\vr) = \cB(1) + \int_1^\vr \frac{b(z)}{z^2} \mathrm{d}z$.
		\item The coupled momentum equation is satisfied in the sense of
		\begin{align}
            \nonumber
            & \int_0^t \int_B \Big(
                \vr \bu \cdot \pt \bphi 
				+ (\vr \bu \otimes \bu) : \nabla \bphi
				+ p_\delta(\vr) \Div \bphi - \bbs_\kappa(\nabla \bu) : \nabla \bphi \Big) \dxdtau \\
				\nonumber
		  & 
            \qquad + \int_0^t \int_\Gamma \big(
                \pt w \pt \psi - \Delta w \Delta \psi - \nu_s \nabla \pt w \cdot \nabla \psi
                \big) \dH^2 \d \tau
                + \alpha \int_0^t \int_{\Gamma^w(\tau)} \bu_{\btau^w} \cdot \bphi_{\btau^w} \dH^2 \d \tau \\
                \label{eqs:weak-momentum-ext}
           & \qquad + \frac{1}{\kappa} \int_0^T \int_{\Gamma^w(\tau)} \big( (\bu - \pt w \bn \circ \inv{\bPhi_w}) \cdot \bn^w \big) \big( (\bphi - \psi \bn \circ \inv{\bPhi_w}) \cdot \bn^w \big) \dH^2 \dtau \\
		  \nonumber
            & = \int_B (\vr \bu) \cdot \bphi(t, \cdot) \dx + \int_{\Gamma} \pt w \psi(t, \cdot) \dH^2 
            - \int_B (\vr \bu)_{0,\delta} \cdot \bphi(0, \cdot) \dx - \int_{\Gamma} w_{0,1} \psi(0, \cdot) \dH^2,
		\end{align}
		for all $t \in [0,T]$ and $ \bphi \in C^\infty([0,T] \times \bbr^3) $ and $ \psi \in C^\infty([0,T] \times \Gamma) $ satisfying $ (\psi \bn) \cdot (\bn^w \circ \bPhi_w) = \tr_{\Gamma^w} \bphi \cdot (\bn^w \circ \bPhi_w) $ on $ \Gamma_T $. Here the tensor $\bbs_\kappa$ is defined by $ \bbs_\kappa(\nabla \bu) = \mu_\kappa(\nabla \bu + \nabla \tran{\bu}) + \lambda_\kappa \Div \bu \bbi $.
		\item The energy inequality
		\begin{align}
			\nonumber
			& \int_B \left( \onehalf \vr \abs{\bu}^2 + \frac{p_\delta(\vr)}{\gamma - 1} \right)(t) \dx
			+ \int_{\Gamma} \left( \onehalf \abs{\pt w}^2 + \onehalf \abs{\Delta w}^2 \right)(t) \, \d \mathcal{H}^2 \\
			\nonumber
			& \qquad + \int_{0}^{t} \int_B \bbs_\kappa(\nabla \bu) : \nabla \bu \dxdtau
			+ \frac{1}{\kappa} \int_0^T \int_{\Gamma^w(\tau)} \abs{(\bu - \pt w \bn \circ \inv{\bPhi_w}) \cdot \bn^w}^2 \dH^2 \dtau \\
			\nonumber
			& \qquad + \alpha \int_{0}^{t} \int_{\Gamma^w(\tau)} \abs{(\bu - \pt w \bn \circ \inv{\bPhi_w})_{\btau^w}}^2 \dH^2 \d \tau
			+ \nu_s \int_{0}^{t} \int_{\Gamma} \abs{\nabla \pt w}^2 \dH^2 \d \tau \\
			\label{eqs:weak-energy-inequality-ext}
			& \qquad \quad \leq \int_{\Omega^{w_0}} \left( \frac{1}{2 \vr_{0,\delta}} \abs{(\vr \bu)_{0,\delta}}^2 + \frac{p_\delta(\vr_{0,\delta})}{\gamma - 1} \right) \dx
			+ \int_{\Gamma} \left( \onehalf \abs{w_{0,1}}^2 + \onehalf \abs{\Delta w_0}^2 \right) \dH^2,
		\end{align}
		holds for a.e. $ t \in [0, T] $. 
	\end{enumerate}
\end{definition}

The aim is to prove the existence of the approximate solutions.
\begin{theorem}
    \label{thm:weak-ext}
    Assume that the initial data satisfies \eqref{eqs:initial} and \eqref{eqs:initial-reg}. Then there exists $ T > 0 $ and a weak solution $(\vr, \bu, w)$ to the approximate system in the sense of Definition \ref{def:weak-extended}.
\end{theorem}
\begin{proof}
    The proof relies on a time discretization scheme, together with an operator splitting method, which was employed in \cite{KMN2023,MMNRT2022} for other related fluid-structure interaction problems. For the splitting method, we decouple the problem into two subproblems: fluid and structure, then we proceed with the method from \cite{KMN2023,MMNRT2022} and passing with time layer to zero we get the existence of the approximate solution. The differences result from the Navier-slip boundary, which is not an issue to prove the existence of weak solutions to the extended regularized problem. For simplicity, we omit the details here.
\end{proof}


\subsection{Uniform estimates}
\label{sec:uniform-estimates}

In this section, we are going to derive uniform in $\kappa$ and $\delta$ estimates so that one is able to further pass to the limit as $\kappa = \delta \to 0$ step by step, i.e., the extension limit and the vanishing artificial pressure limit. {Let $(\vr^\delta, \bu^\delta, w^\delta)$ be a pair of approximate solutions constructed in Theorem \ref{thm:weak-ext}}, for which we suppress the dependence on $ \kappa $ for simplicity. Then in view of \eqref{eqs:weak-energy-inequality-ext}, we conclude that following uniform (in $\delta$ and $ \kappa $) boundedness
\begin{alignat*}{3}
    w^\delta & \quad  \text{ is bounded in } \quad && L^\infty(0,T; H^2(\Gamma)), \\
    \pt w^\delta & \quad  \text{ is bounded in } \quad && L^\infty(0,T; L^2(\Gamma)), \\
    \sqrt{\nu_s} \pt w^\delta & \quad  \text{ is bounded in } \quad && L^2(0,T; H^1(\Gamma)), \\
    \vr^\delta & \quad  \text{ is bounded in } \quad && L^\infty(0,T; L^\gamma(B)), \\
    \delta^{\frac{1}{\beta}} \vr^\delta & \quad  \text{ is bounded in } \quad && L^\infty(0,T; L^\beta(B)), \\
    \sqrt{\vr^\delta} \bu^\delta & \quad  \text{ is bounded in } \quad && L^\infty(0,T; L^2(B)), \\
    \bbs_\delta(\nabla \bu^\delta) : \nabla \bu^\delta & \quad  \text{ is bounded in } \quad && L^1((0,T) \times B). 
\end{alignat*}
Moreover, a direct use of Korn's inequality on H\"older domain (e.g., \cite[Lemma 3.9]{MMNRT2022}) yields
\begin{alignat}{3}
	\label{eqs:v-delta-regularity}
    \nabla \bu^\delta & \quad  \text{ is bounded in } \quad && L^2(0,T; L^q(\Omega^{w^\delta}(t))), \text{ for } 1 \leq q < 2,
\end{alignat}
which by Sobolev embedding implies
\begin{alignat*}{3}
	\bu^\delta & \quad  \text{ is bounded in } \quad && L^2(0,T; L^p(B)), \text{ for } 1 \leq p < 6.
\end{alignat*}
Indeed, 
\begin{align*}
    \int_{Q_T^{w^\delta}} & \Big( \mu \abs{\nabla \bu^\delta + \tran{(\nabla \bu^\delta)}}^2  + \lambda \abs{\Div \bu^\delta \bbi}^2 \Big) \dxdt \\
    & = \int_{Q_T^{w^\delta}} \Big( \mu (\nabla \bu^\delta + \tran{(\nabla \bu^\delta)})  + \lambda \Div \bu^\delta \bbi \Big) : \nabla \bu^\delta \dxdt
    \leq \int_{(0,T) \times B} \bbs_\delta(\nabla \bu^\delta) : \nabla \bu^\delta \dxdt \leq C,
\end{align*}
where $ C > 0 $ does not depend on $\delta > 0$. Then by \cite[Lemma 2.4]{KMN2023}, one has
\begin{alignat*}{3}
	\bu^\delta & \quad  \text{ is bounded in } \quad && L^2(0,T; L^{p}(\bbr^3)), \text{ for } 1 \leq p < 6.
\end{alignat*}

Now we derive the integrability of the density. Since the domain now is not uniformly Lipschitz due to elastic structure regularity, the usual technique of the Bogovski\u{i} operator (e.g., \cite{FN2017,NS2004}) does not work in general. However, one can overcome the problem by considering the domain separately. More precisely, away from the `bad' boundary, it is easy to derive the equi-integrability by the usual Bogovski\u{i} operator, i.e.,
\begin{equation}
	\label{eqs:Bogovskii-D}
	\int_D \left( (\vr^\delta)^{\gamma + \theta} + \delta (\vr^\delta)^{\beta + \theta} \right) \dxdt \leq C(D),
\end{equation}
for some $ \theta > 0 $ and any $ D \subset \subset \bigcup_{t \in [0,T]} \{t\} \times (\overline{B} \setminus \Gamma^{w^\delta}(t)) $, where $ C > 0 $ is independent of $\delta$. 
Moreover, we need to exclude the possible concentration of pressure at the moving boundary as follows:
\begin{lemma}[Equi-integrability]
	\label{lem:equi-integrability-with}
    Let either $ \nu_s > 0 $, $\gamma > \frac{3}{2}$ or $\Omega \in \bbtwo$, $\nu_s = 0$, $ \gamma > \frac{12}{7} $. For any $\varepsilon > 0$, there exists a $\varepsilon_0 > 0$ and $A_\varepsilon \subset \subset (0,T) \times B$ such that for all $\varepsilon \leq \varepsilon_0$ the following holds
    \begin{equation}
    	\label{eqs:equi-integrability}
        A_\varepsilon \cap (\Gamma^{w^\delta}_T) = \emptyset, \quad 
        \int_{((0,T) \times B) \setminus A_\varepsilon} \left( (\vr^\delta)^\gamma + \delta (\vr^\delta)^\beta\right) \dxdt \leq \varepsilon.
    \end{equation}
\end{lemma}


\begin{remark}
	\label{rem:equi-integrability}
	For the proof, we refer to \cite{BS2018,KMN2023,MMNRT2022}. This problem was first handled in \cite[Lemma 7.4]{BS2018} for compressible fluid-structure interaction problems initiated by \cite{Kukucka2009}. 
    Recently in \cite[Lemma 3.4]{MMNRT2022}, it was proved again for the Navier--Stokes--Fourier system coupled with a thermoelastic shell with viscoelasticity and inertial rotation. 
    Following the proof of \cite[Lemma 6.4]{BS2018}, one can shows there is no extra restriction on $\gamma > \frac{3}{2}$ with the structure damping $\nu_s \Delta \pt w$.
    In the case of $\nu_s = 0$, Breit--Schwarzacher \cite{BS2018} improved the result by means of the no-slip boundary condition $ \pt w \bn = \bu \circ \bPhi_w $. When encountering the slip condition \eqref{eqs:velocity-boundary}, we can still make it by adapting it to the proof of \cite[Lemma 7.4]{BS2018}, but only in the case $\Omega \in \bbtwo$. More precisely, in this case we know $ \pt w = (\bu \cdot \bn^w) \circ \bPhi_w $, then
    in \cite[(7.15)]{BS2018}, instead of using no-slip boundary condition, one replaces $ \pt w $ by $ (\bu \cdot \bn^w) \circ \bPhi_w $, which lies in the space $ L^2(0,T;L^r(\Gamma)) $ for any $ r < q < 4 $, as $ \tr_{\Gamma^w} \bu \in L^2(0,T;L^q(\Gamma)) $ for any $ q < 4 $ and $ \bn^w \in L^\infty(0,T;L^p(\Gamma)) $, for any $ p < \infty $. Then it is needed that
    \begin{equation*}
         \vr \bu \in L^2(0,T; L^{q}(\Omega^w(t))),
    \end{equation*}
    with $ q > \frac{6 \gamma}{\gamma + 6} > \frac{4}{3} $, for $\gamma > \frac{12}{7}$.

\end{remark}

\subsection{Existence of weak solutions}
\label{sec:weak-ext}

Based on the uniform estimates derived in the previous section, we are now devoted to proving Theorem \ref{thm:weak} by passing to the limit as $\kappa = \delta \to 0$. Up to a non-relabeled subsequence, we have
\begin{alignat*}{4}
    w^\delta & \rightarrow w, \quad && \text{weakly-}* \quad && \text{in } L^\infty(0,T; H^2(\Gamma)), \\
    \pt w^\delta & \rightarrow \pt w, \quad && \text{weakly-}* \quad && \text{in } L^\infty(0,T; L^2(\Gamma)), \\
    \sqrt{\nu_s} \pt w^\delta & \rightarrow \sqrt{\nu_s} \pt w, \quad && \text{weakly} \quad && \text{in } L^2(0,T; H^1(\Gamma)), \\
    \vr^\delta & \rightarrow \vr, \quad && \text{weakly-}* \quad && \text{in } L^\infty(0,T; L^\gamma(B)), 
     \\
     \bu^\delta & \rightarrow \bu, \quad && \text{weakly} \quad && \text{in } L^2(0,T; W^{1,q}(\bbr^3)), \text{ for } 1 \leq q < 2.
\end{alignat*}
Note that by a similar argument to that in \cite[Lemma 4.4]{KMN2023}, we conclude $ \vr_{|B \setminus \Omega^{w^\delta}(t)} = 0 $ for a.e. $ t \in (0,T) $.

A direct application of the Kondrachov embedding theorem and the Aubin--Lions lemma tells us
\begin{alignat*}{4}
	w^\delta & \rightarrow w, \quad && \text{strongly} \quad && \text{in } C([0,T]; H^{r}(\Gamma)),
\end{alignat*}
for $ 1 < r < 2 $.
and 
\begin{alignat*}{4}
	w^\delta & \rightarrow w, \quad && \text{strongly} \quad && \text{in } C([0,T] \times \Gamma).
\end{alignat*}
By definition, we know that
\begin{alignat}{4}
	\label{eqs:Phi_w-strong-convergence}
	\bPhi_{w^\delta} & \rightarrow \bPhi_w, \quad && \text{strongly} \quad && \text{in } C([0,T] \times \Gamma).
\end{alignat}
Now we are going to give an additional regularity for the displacement $ w $, which was also obtained originally in \cite{MS2022} to tackle the incompressible Navier--Stokes equation interacting with a nonlinear Koiter shell, and then developed in \cite{BS2021,KMN2023,Trifunovic2023} for related compressible fluid-structure interaction problems.
\begin{proposition}[Additional spatial regularity]
	\label{prop:more-spatial-regularity}
	Let $ (\vr^\delta, \bu^\delta, w^\delta) $ be a weak solution of the extended problem in the sense of Definition \ref{def:weak-extended} on $ (0,T) $.
	If $ \nu_s > 0 $, $ \gamma > \frac{3}{2} $,
	Then the displacement $ w^\delta $ admits an additional regularity
	\begin{equation}
		\label{eqs:w^delta-addtional-regularity}
		w^\delta \in L^2(0,T; H^{2 + s}(\Gamma)) \  \text{ for some } 0 < s < \onehalf,
	\end{equation}
	which depends on $ \Gamma $ and the initial data, but is independent of $ \delta > 0 $.
	If $ \Omega \in \bbtwo $, $ \nu_s = 0 $, $ \gamma > \frac{12}{7} $, then it holds
	\begin{align}
		& \pt w^\delta \in L^2(0,T; W^{s, r}(\Gamma)) \  \text{ for all } 0 < s  < 2 - \frac{3}{r},\  \frac{3}{2} < r < 2, \\
		\label{eqs:w^delta-addtional-regularity-t}
		& w^\delta \in L^2(0,T; H^{2 + s}(\Gamma)) \  \text{ for some } 0 < s < \onequater,
	\end{align}
	which depends on $ \Gamma $ and the initial data, but is independent of $ \delta > 0 $.
	
\end{proposition}
\begin{proof}
    For the case $ \nu_s > 0 ,$ in view of the regularity $\partial_{t}w^{\delta}\in L^{2}(H^{1}(\Gamma))$ one can adapt the arguments from \cite{MS2022,Trifunovic2023} in a rather straight forward manner to show \eqref{eqs:w^delta-addtional-regularity}.\\
    If $ \nu_s = 0 $, in order to use the arguments from \cite{MS2022,Trifunovic2023}, we need some spatial regularity of $\partial_{t}w^{\delta}$. We obtain that regularity in the following discussion by exploiting the interface coupling condition which allows to transfer the effect of fluid dissipation to the structure. Let us make use of the kinematic boundary condition \eqref{eqs:pt_w=un} for $\Omega \in \bbtwo$ to gain regularity from velocity, i.e., $$\pt w^\delta = (\bu^\delta \cdot \bn^w) \circ \bPhi_{w^\delta}.$$ More precisely, in view of \eqref{eqs:v-delta-regularity} and Lemma \ref{lem:trace-Phi_w}, we know $ \bu^\delta \circ \bPhi_{w^\delta} \in L^2(0,T; W^{1 - \frac{1}{r},r}(\Gamma)) $ for all $ 1 \leq r < 2 $. Moreover, the normal $ \bn^w \circ \bPhi_{w^\delta} \in L^\infty(0,T; W^{1,r}(\Gamma)) $, which inherits from the regularity of $ \nabla w^\delta $. Then by Lemma \ref{lem:multiplication}, the multiplication property of Sobolev spaces, it follows
	\begin{equation*}
		\normm{\pt w^\delta}_{L_t^2 W_\Gamma^{\sigma,r}}
		\leq C \normm{\bu^w \circ \bPhi_{w^\delta}}_{L_t^2 W_\Gamma^{1 - \frac{1}{r},r}}
		\normm{\bn^w \circ \bPhi_{w^\delta}}_{L_t^\infty W_\Gamma^{1,r}}
		< \infty,
	\end{equation*}
	with any $ 0 < \sigma < 2 - \frac{3}{r} $ and $ \frac{3}{2} < r < 2 $, where $ C > 0 $ depends on the $ \Gamma $, the solution $ \bu^\delta $ and $ w^\delta $.
	Now proceeding with the similar argument as in \cite{MS2022,Trifunovic2023}, one is able to recover the regularity of $ w^\delta $ for some $ 0 < s < \onequater $ and all $ \gamma > \frac{12}{7} $.
	Indeed, in this case $ s $ can be chosen precisely as
	\begin{equation}
		\label{eqs:add-s-nodamping}
		0 < s < \min\left\{\frac{5 - \frac{6}{r} + 3\sigma}{6} - \frac{1}{\gamma}, \frac{\sigma}{2}\right\}, \quad 0 < \sigma < 2 - \frac{3}{r}, \ \frac{3}{2} < r < 2.
	\end{equation}
	If one only takes some $ r $ sufficiently close to $ 2 $, we get lower bound $ \gamma > \frac{12}{7} $. For the derivation of \eqref{eqs:add-s-nodamping}, we refer to \cite[Section 3.1]{Trifunovic2023} for more details, with a slight modification. The key point here is that due to technical reasons (constructing suitable test functions for $\bu^\delta$ based on $\pt w^\delta$) as in \cite[Section 3.1]{Trifunovic2023}, one has to ensure the validation of the embedding
	\begin{equation*}
		W^{\sigma,r}(\Gamma) \hookrightarrow W^{2s,p}(\Gamma), \ \text{ for } p \geq \frac{6 \gamma}{5\gamma - 6},
	\end{equation*}
	for $ 0 < \sigma < 2 - \frac{3}{r} < \onehalf $ from additional spatial regularity of $ \pt w^\delta $ with $ \frac{3}{2} < r < 2 $.  Then by Sobolev embedding, it has to hold $ \sigma - \frac{2}{r} \geq 2s - \frac{2}{p} $, entailing
	\begin{equation*}
		0 < 2s < \sigma \ \text{ and } \  \frac{6\gamma}{5\gamma - 6} \leq p < \frac{2}{2s + \frac{2}{r} - \sigma}.
	\end{equation*}
	To find some suitable $ p $, one arrives at \eqref{eqs:add-s-nodamping}.
\end{proof}

\subsubsection{Recover the boundary identity}
\label{sec:recover-boundary}
This subsection is devoted to recovering the continuity of the limiting normal velocities at the boundary. In general, it is rather easier for the case of Dirichlet boundary condition, which does not require more compactness of the structure displacement; see, e.g. \cite[Section 6.1.2]{KMN2023}. However, encountering the \emph{Navier-slip} boundary condition, this is more difficult, since we need more compactness to justify the limit of the deformed normal $\bn^w$ depending nonlinearly on $\nabla w$.
On noting \eqref{eqs:w^delta-addtional-regularity} and the uniform boundedness of $ \pt w^\delta $, i.e.,
\begin{align*}
	\sqrt{\nu_s} \pt w^\delta & \in L^2(0,T; H^1(\Gamma)), \\
	\pt w^\delta & \in L^2(0,T; W^{s,r}(\Gamma)),\ \text{ for some } 0 < s  < \onehalf,\  \frac{3}{2} < r < 2,
\end{align*}
an employment of the Kondrachov embedding theorem and the Aubin--Lions compactness lemma yields
\begin{alignat*}{4}
	w^\delta & \rightarrow w, \quad && \text{strongly} \quad && \text{in } L^2(0,T; H^{2+\epsilon}(\Gamma)),
\end{alignat*}
where $ 0 < \epsilon < s $, implying
\begin{alignat*}{4}
	\nabla w^\delta & \rightarrow \nabla w, \quad && \text{strongly} \quad && \text{in } L^2(0,T; C(\Gamma)).
\end{alignat*}
Then we record
\begin{alignat}{4}
	\label{eqs:n^w-strong-convergence}
	\bn^{w^\delta} \circ \bPhi_{w^\delta} & \rightarrow \bn^w \circ \bPhi_w, \quad && \text{strongly} \quad && \text{in } L^2(0,T; C(\Gamma)).
\end{alignat}
In fact, the outer normal $ \bn^{w^\delta} $ inherits the same regularity as $ \nabla w^{\delta} $, in view of \eqref{eqs:Phi_w-strong-convergence}, $ \bn^w \circ \bPhi_w = \nabla \invtr{\bPhi_w} \bn \circ \Phi $, and that $ \bn $, $ \Phi $ are smooth. Specifically, by \eqref{eqs:bPhi_w^-1},
\begin{equation*}
	\norm{\invtr{\nabla \bPhi_{w^\delta}} - \invtr{\nabla \bPhi_{w}}}_{L^2(0,T;C(\Gamma))}
	\leq C \norm{\nabla w^\delta - \nabla w}_{L^2(0,T;C(\Gamma))}
	\to 0, \text{ as } \delta \to 0.
\end{equation*}

\begin{remark}
    Note that in Muha--\v{C}ani\'{c} \cite[Corollary 3]{MC2016}, a two-dimensional fluid-structure interaction problem  with full displacement and Navier-slip boundary condition was considered, and the deformed normal has the convergence 
    \begin{alignat*}{4}
        \bn^{w^\delta} & \rightarrow \bn^w, \quad && \text{strongly} \quad && \text{in } L^\infty(0,T; C(\Gamma)).
    \end{alignat*}
    However, in our case of 3D fluids interacting with a 2D structure, this is no longer possible, even with structure dissipation. The reason is that Sobolev embeddings strongly depend on the dimensions and we are not able to obtain the same compactness of $ w^\delta $ as in \cite[Theorem 5 \& Corollary 2]{MC2016}. Thanks to the additional regularity \eqref{eqs:w^delta-addtional-regularity} and \eqref{eqs:w^delta-addtional-regularity-t}, one can finally recover the spatial regularity with a bit of sacrifice of time integrability, compared to \cite{MC2016}.
\end{remark}

Now we justify the kinematic condition. Recall that by the penalization we have
\begin{equation*}
	\tr_{\Gamma^{w^\delta}} \bu^\delta \cdot \big(\bn^{w^\delta} \circ \bPhi_{w^\delta}\big) = \big(\pt {w^\delta} \bn\big) \cdot \big(\bn^{w^\delta} \circ \bPhi_{w^\delta}\big),
\end{equation*}
for all $\delta > 0$. In the following one verifies that the limit solution $ (\bu, \pt w) $ in fact satisfies
\begin{equation*}
	\tr_{\Gamma^w} \bu \cdot \big(\bn^w \circ \bPhi_{w}\big)
	= \pt w \bn \cdot \big(\bn^w \circ \bPhi_{w}\big), \ton \Gamma_T.
\end{equation*}
By virtue of \eqref{eqs:n^w-strong-convergence} and the uniform weakly convergence of $ \{\pt w^\delta\}_{\delta>0} $, we derive
\begin{align*}
	\int_\Gamma & \big(\pt {w^\delta} \bn\big) \cdot \big(\bn^{w^\delta} \circ \bPhi_{w^\delta}\big) \phi \dH^2
	- \int_\Gamma \big(\pt {w} \bn\big) \cdot \big(\bn^{w} \circ \bPhi_{w}\big) \phi \dH^2 \\
	& = \int_\Gamma \big((\pt {w^\delta} - \pt w) \bn\big) \cdot \big(\bn^{w} \circ \bPhi_{w}\big) \phi \dH^2
	+ \int_\Gamma \big(\pt {w^\delta} \bn\big) \cdot \big(\bn^{w^\delta} \circ \bPhi_{w^\delta} - \bn^{w} \circ \bPhi_{w}\big) \phi \dH^2 \\
	& \leq \int_\Gamma \big((\pt {w^\delta} - \pt w) \bn\big) \cdot \Big(\big(\bn^{w} \circ \bPhi_{w}\big) \phi\Big) \dH^2
	+ C \norm{\bn^{w^\delta} \circ \bPhi_{w^\delta} - \bn^{w} \circ \bPhi_{w}}_{L^2(0,T;C(\Gamma))} \\
	& \to 0, \quad \text{as } \delta \to 0,
\end{align*}
for all $ \phi \in L^2((0,T) \times \Gamma) $.
In view of the fact that $ \{\bu^\delta \circ \bPhi_{w^\delta}\}_{\delta > 0} $ is bounded in $ L^2(0,T; W^{1,q}(\bbr^3)) $ uniformly for any $ q \in [1,2) $, there is a function $\overline{\bu \circ \bPhi_{w}}$ in $L^2(0,T; W^{1,q}(\bbr^3))$ such that we are able to extract a non-relabeled subsequence fulfilling
\begin{alignat*}{4}
	\bu^\delta \circ \bPhi_{w^\delta} & \rightarrow \overline{\bu \circ \bPhi_{w}}, \quad && \text{weakly} \quad && \text{in } L^2(0,T; W^{1,q}(\bbr^3)), \text{ for } 1 \leq q < 2.
\end{alignat*}
Then by Lemma \ref{lem:trace-Phi_w}, one infers that
$ \rv{\overline{\bu \circ \bPhi_{w}}}_{\Gamma} \cdot \big(\bn^{w^\delta} \circ \bPhi_{w^\delta}\big) = \pt w \cdot \big(\bn^{w} \circ \bPhi_{w}\big) $ on $ \Gamma_T $, where limit of $ \overline{\bu \circ \bPhi_{w}} $ is identified by $ \bu \circ \bPhi_{w} $ following analogously \cite[Section 6.1.2]{KMN2023}.

\subsubsection{Convergence of the convection term in the momentum equation}
Now we would like to derive more convergences, in order to obtain the limit passage of the momentum equation. In particular, we will show the convergence regarding the convection term
\begin{alignat*}{4}
    \vr^\delta \bu^\delta \otimes \bu^\delta & \to \vr \bu \otimes \bu, \quad && \text{weakly} \quad && \text{in } L^1((0,T) \times B).
\end{alignat*}
First of all, one records
\begin{alignat}{4}
    \label{eqs:rho^delta-convergence-C_wL^gamma}
    \vr^\delta & \rightarrow \vr, \quad && \text{strongly} \quad && \text{in } C_w([0,T]; L^\gamma(B)).
\end{alignat}
Indeed, by the continuity equation \eqref{eqs:weak-renormal-ext} in the sense of distribution, we have
\begin{equation*}
    \frac{\d}{\dt} \int_B \vr^\delta \phi \dx = \int_B \vr^\delta \bu^\delta \cdot \nabla \phi \dx, \text{ in } \cD'(0,T) \text{ for all } \phi \in \cD(B).
\end{equation*}
Consequently, $\vr^\delta$ is uniformly continuous in $W^{-1,\frac{2\gamma}{\gamma + 1}}(B)$. Note that $\vr^\delta \in C_w([0,T]; L^\gamma(B))$ and $\vr^\delta$ is uniformly bounded in $L^\gamma(B)$. Then in view of \cite[Lemma 6.2]{NS2004}, one concludes \eqref{eqs:rho^delta-convergence-C_wL^gamma}. Moreover, it follows from the boundedness of $ \sqrt{\vr^\delta} \bu^\delta $, weak convergence of $ \bu^\delta $ and the strong convergence \eqref{eqs:rho^delta-convergence-C_wL^gamma} that
\begin{alignat*}{4}
	\vr^\delta \bu^\delta & \to \vr \bu, \quad && \text{weakly-}* \quad && \text{in } L^\infty(0,T; L^{\frac{2 \gamma}{\gamma + 1}}(B)).
\end{alignat*}
Employing the similar argument as in \cite{FN2017,NS2004} along with the weak formulation of the approximate momentum equation \eqref{eqs:weak-momentum-ext} gives rise to
\begin{alignat*}{4}
	\vr^\delta \bu^\delta & \to \vr \bu, \quad && \text{strongly} \quad && \text{in } C_w([0,T]; L^{\frac{2 \gamma}{\gamma + 1}}(B)).
\end{alignat*}
As $ \frac{2 \gamma}{\gamma + 1} > \frac{6}{5} = \frac{3 \cdot 2}{3 + 2} $, we have the compact embedding $ L^{\frac{2 \gamma}{\gamma + 1}}(B) \subset \subset W^{-1,2}(B) $, which by \cite[Lemma 6.4]{NS2004} implies
\begin{alignat*}{4}
	\vr^\delta \bu^\delta & \to \vr \bu, \quad && \text{strongly} \quad && \text{in } L^2(0,T; W^{-1,2}(B)).
\end{alignat*}
Then one infers from the weak convergence of $ \bu^\delta $ in $ L^2(0,T; W^{1,q}) $, $ 1 \leq q < 2 $ that
\begin{alignat*}{4}
	\vr^\delta \bu^\delta \otimes \bu^\delta & \to \vr \bu \otimes \bu, \quad && \text{weakly} \quad && \text{in } L^1((0,T) \times B).
\end{alignat*}

\subsubsection{Convergence of the pressure term in momentum equation}
To handle the approximate pressure in the momentum equation, we first conclude an observation by the uniform estimates.
\begin{lemma}
	Under the assumption of Theorem \ref{thm:weak}, there exists an integrable function $ \overline{p(\vr)} $ such that
	\begin{alignat*}{4}
		p_\delta(\vr) & \to \overline{p(\vr)}, \quad && \text{weakly} \quad && \text{in } L^1(D),
	\end{alignat*}
	up to a non-relabeled subsequence for any $ D \subset \subset \bigcup_{t \in [0,T]} \{t\} \times (\overline{B} \setminus \Gamma^{w^\delta}(t)) $. Additionally, for $ \varepsilon > 0 $ arbitrary, there are a $ \varepsilon_0 > 0 $ and a measurable set $A_\varepsilon \subset \subset (0,T) \times B$ such that for all $\varepsilon \leq \varepsilon_0$ the following holds
	\begin{equation}
		\label{eqs:limit-p}
		A_\varepsilon \cap (\Gamma^{w^\delta}_T) = \emptyset, \quad 
		\int_{((0,T) \times B) \setminus A_\varepsilon} \overline{p(\vr)} \dxdt \leq \varepsilon.
	\end{equation}
\end{lemma}
\begin{proof}
	First it follows that, as a direct consequence of \eqref{eqs:Bogovskii-D} and \eqref{eqs:equi-integrability},
	\begin{alignat*}{4}
		\delta \vr^\beta & \to 0, \quad && \text{strongly} \quad && \text{in } L^1((0,T) \times B).
	\end{alignat*}
	Then the first assertion follows directly from \eqref{eqs:Bogovskii-D}, while \eqref{eqs:equi-integrability} implies the second one.
\end{proof}

\subsubsection{An effective viscous flux identity}
In the following, we are devoted to proving some compactness of $ \vr^\delta $ so that one can identify the weak limit pressure term $ \overline{p(\vr)} $ with $ p(\vr) $. To this end, one follows similar ideas as in the general framework of compressible Navier--Stokes equations \cite{FN2017,NS2004}, with corresponding adaptations in compressible fluid-structure interaction problems \cite{BS2018,KMN2023,MMNRT2022}. Namely, we will make use of the so-called \emph{effective viscous flux}, which provides more information than a single term.

We define the $ L^\infty $-truncation
\begin{equation*}
	T_k(z) \coloneqq k T \left(\frac{z}{k}\right) \quad z \in \bbr, k \in \bbn,
\end{equation*}
where $ T $ is a smooth concave function on $ \bbr $ such that $ T(z) = z $ for $ z \in [0,1) $ and $ T(z) = 2 $ for $ z \geq 3 $. Then the \emph{effective viscous flux identity} is stated as follows, exactly similar to \cite{BS2018,KMN2023}. We omit the proof here.
\begin{lemma}
	\label{lem:effective-viscous-flux-identity}
	Up to a non-relabeled subsequence, the following identity holds
	\begin{equation}
		\label{eqs:effective-viscous-flux-identity}
		\int_{Q_T^{w^\varepsilon}} \left(
			p_\delta(\vr^\delta) - (\lambda + 2 \mu) \Div \bu^\delta
		\right) T_k(\vr^\delta) \dxdt
		\to 
		\int_{Q_T^{w}} \left(
		\overline{p(\vr)} - (\lambda + 2 \mu) \Div \bu^\delta
		\right) \overline{T_k(\vr)} \dxdt,
	\end{equation}
        as $\delta \to 0$ {for a fixed $k$}.
\end{lemma}

\subsubsection{Compactness of the density}
By means of the renormalized formulation of the continuity equation, we show the strong convergence of the density. Let $L_k(z) = z \ln z $ for $0 \leq z < k$ and $L_k(z) = z \int_k^z \frac{T_k(s)}{s^2} \,\d s$ for $z \geq k$ with $z L_k'(z) - L_k(z) = T_k(z)$. On noting that $\vr^\delta$ and $\vr$ are defined zero outside $\Omega^{w^\delta}$ and $\Omega^w$ respectively, the functions $L_k(\vr^\delta)$ and $L_k(\vr)$ vanish outside $\Omega^{w^\delta}$ and $\Omega^w$ respectively as well. 
Invoking the renormalized formulation of the continuity equation \eqref{eqs:weak-renormal-ext}, we have
\begin{equation*}
    \begin{aligned}
        \int_{\bbr^3} (L_k(\vr^\delta) - L_k(\vr))(t) \dx
        & = \int_0^t \int_{\bbr^3} (T_k(\vr) \Div \bu - \overline{T_k(\vr)} \Div \bu^\delta) \dxdtau \\
        & \quad 
        + \int_0^t \int_{\bbr^3} (\overline{T_k(\vr)} - T_k(\vr^\delta)) \Div \bu^\delta \dxdtau,
    \end{aligned}
\end{equation*}
for all $t \in [0,T]$. In view of the \emph{effective viscous flux identity} \eqref{eqs:effective-viscous-flux-identity}, with exactly the same argument as in \cite{BS2018,KMN2023,MMNRT2022}, one is able to derive the strong convergence of the density
\begin{equation*}
    \vr^\delta \to \vr,  \quad \text{a.e.~in } (0,T) \times \bbr^3. 
\end{equation*}

Combining all the ingredients above, together with the verification of energy inequality, the attainment of initial data, and the maximal interval of existence, cf. \cite[Section 6.1.7--6.1.10]{KMN2023}, we finally prove Theorem \ref{thm:weak}.
\qed

\section{Relative Entropy and Singular Limits}
\label{sec:singular-limits}
This section is devoted to providing a rigorous justification of the \textit{incompressible inviscid limit} of our compressible fluid-structure interaction problem to an Euler-plate system, in the regime of \textit{low Mach number} and \textit{high Reynolds number}. In the following, we first present a formal rescaling argument in Section \ref{sec:rescaled-system} to derive the limit system desired in Section \ref{sec:target-system}. Then we derive a modified relative energy inequality under the current framework in Section \ref{sec:relative-energy}, which entails the uniform-in-$ \varepsilon $ (and $ \nu $) bounds in Section \ref{sec:uniform-bounds}. To compare the original system, we introduce the basic transform between two variant domains in Section \ref{sec:tranformation}. Finally, the limit passage is justified in Section \ref{sec:singular-limits-final}.

Note that we will investigate the case of slab-like geometry $\Omega = \Gamma \times (0,1)$, i.e., the reference geometry is flat, see Figure \ref{fig:flat-geometry}.
\begin{figure}[ht!]
	\centering
    \begin{tikzpicture}[line width=0.5pt,scale=0.9]
		\begin{scope}[declare function={
		  	f(\x)=.2*sin(deg(\x*pi/2))+2;
		},scale=1.4]
			\definecolor{aqua}{rgb}{0.0, 1.0, 1.0}
			
		  	\draw[-stealth] (-.5,0) -- (4.5,0) node[right] {};
		  	\draw[-stealth] (0,-.25) -- (0,2.75);
		  	\draw[blue,domain = 0:4,samples=100] plot(\x, {f(\x)});
			
			\fill[aqua!20] plot[domain = 0:4,samples=100](\x, {f(\x)}) -- (4,0) -- (0,0) -- (0,2);
				
		  	\draw[dashed] (0,2) node[left] {1} -- (4,2);
						
		  	\node[below left] at (0,0) {$ 0 $};

            \node[below] at (4,0) {$ \Gamma_0 = \bbt^2 \times \{0\} $};
		\end{scope}
    \end{tikzpicture}
    \caption{A slab-like geometry.}
    \label{fig:flat-geometry}
\end{figure}

\subsection{Rescaled system}
\label{sec:rescaled-system}
First, let us introduce the scaled equations of our problem, in terms of the Mach number $ \varepsilon > 0 $ and Reynolds number $ 1/\nu $ with $ \nu > 0 $. Denote by $ (\vr^\varepsilon, \bu^\varepsilon, w^\varepsilon) $ the weak solution depending on $ \varepsilon $ (also depending on $\nu$, which is ignored in the notation for the sake of readability) satisfying 
\begin{subequations}
	\label{eqs:FSI-model-rescale}
	\begin{alignat}{3}
		\pt \vr^\varepsilon + \Div (\vr^\varepsilon \bu^\varepsilon) & = 0, && \tin Q_T^{w^\varepsilon}, \\ 
		\pt (\vr^\varepsilon \bu^\varepsilon) + \Div (\vr \bu^\varepsilon \otimes \bu^\varepsilon) + \frac{1}{\varepsilon^2} \nabla (p(\vr^\varepsilon) - p(\bar{\vr})) & = \nu \Div \bbs(\nabla \bu^\varepsilon), && \tin Q_T^{w^\varepsilon}, \\
		\pt^2 w^\varepsilon + \Delta^2 w^\varepsilon - \nu_s \Delta \pt w^\varepsilon & = \mathcal{F}^{\varepsilon}, && \ton \Gamma_T, \\
		(\bu^\varepsilon \cdot \bn^{w^\varepsilon}) \circ \bPhi_{w^\varepsilon}
		& = \pt w^\varepsilon, && \ton \Gamma_T, \\
		\nu (\bbs(\nabla \bu^\varepsilon) \bn^{w^\varepsilon})_{\btau^{w^\varepsilon}} \circ \bPhi_{w^\varepsilon} & = - \alpha \nu \bu^\varepsilon_{\btau^{w^\varepsilon}} \circ \bPhi_{w^\varepsilon}, && \ton \Gamma_T, \\
        - \bu^\varepsilon \cdot \be_3 = 0, \quad
        \nu (\bbs(\nabla \bu^\varepsilon) \be_3 + \alpha_0 \bu^\varepsilon)_{\btau^0} & = 0, && \ton \Gamma_T^0,
	\end{alignat}
where $ \bar{\vr} > 0 $ is a constant and the external force on the plate from the fluid is 
\begin{equation}
    \label{eqs:F-scaled}
    \mathcal{F}^\varepsilon(\vr^\varepsilon, \bu^\varepsilon) = - \nu \bbs(\nabla \bu^\varepsilon) \bn^{w^\varepsilon} \circ \bPhi_{w^\varepsilon} \cdot \bn 
    + \frac{p(\vr^\varepsilon) - p(\bar{\vr})}{\varepsilon^2} , \ton \Gamma_T.
\end{equation}
The system is supplemented with the initial data
\begin{equation}
	\label{eqs:SL-initial-data}
	\begin{aligned}
		\vr^\varepsilon(0, \cdot) & = \vr^\varepsilon_0 \coloneqq \bar{\vr} + \varepsilon \vr_0^{\varepsilon, 1}, \\
		\bu^\varepsilon(0, \cdot) & = \bu^\varepsilon_0, \quad 
		w^\varepsilon(0, \cdot) = w^\varepsilon_0, \quad 
		\pt w^\varepsilon(0, \cdot) = w^\varepsilon_1.
	\end{aligned}
\end{equation}
\end{subequations}
In this case, we know $\bn^{w^\varepsilon} 
    = (- \nabla w^\varepsilon, 1)^\top$.

To obtain the scaled system \eqref{eqs:FSI-model-rescale}, we consider the characteristic values 
$U_{f}$, $U_{s}$, $\vr_{f}$, $\vr_{s}$, $p_{f}$, $L$, $T_{f}$, $T_{s}$, $\nu_{f}$, $\nu_{s}$, $W$, $E$, such that
\begin{equation*}
	U_f = \frac{L}{T_f}, \quad
    U_s = \frac{W}{T_s}.
\end{equation*}
The Mach number and Reynolds number are
\begin{equation*}
	\mathrm{Ma} = \frac{U_{f}}{\sqrt{p_{f}/\vr_{f}}} = \varepsilon, \quad 
	\mathrm{Re} = \frac{\vr_{f} U_{f} L}{\nu_{f}} = \frac{1}{\nu},
\end{equation*}
which implies that
\begin{equation*}
	\vr_{f} U_{f}^2 = \mathrm{Ma}^2 p_f, \quad
	\vr_{f} U_{f} L = \mathrm{Re} \nu_{f}.
\end{equation*}
Define new valuables by the characteristic values,
\begin{gather*}
	\vr^\varepsilon(t,x) = \frac{\tilde{\vr}(t T_f,x L)}{\vr_f},\ 
	\bu^\varepsilon(t,x) = \frac{\tilde{\bu}(t T_f,x L)}{U_f},\ 
	w^\varepsilon(t,x) = \frac{\tilde{w}(t T_s,x L)}{W},\\ 
	\mu = \frac{\tilde{\mu}}{N_f},\ 
	\lambda = \frac{\tilde{\lambda}}{N_f},\ 
	\nu_s = \frac{\tilde{\nu}_s}{N_s},\ 
	\alpha = \frac{\tilde{\alpha}}{N_f/L}, \ 
	\alpha_0 = \frac{\tilde{\alpha}_0}{N_f/L},
\end{gather*}
where $ (\tilde{\vr}, \tilde{\bu}, \tilde{w}) $ satisfies the strong formulation
\begin{equation*}
	\left\{
	\begin{alignedat}{3}
		& \pt \tilde{\vr} + \Div (\tilde{\vr} \tilde{\bu}) = 0, && \tin Q_T^{\tilde{w}}, \\ 
		& \pt (\tilde{\vr} \tilde{\bu}) + \Div (\tilde{\vr} \tilde{\bu} \otimes \tilde{\bu}) + \nabla p(\tilde{\vr}) = \Div \widetilde{\bbs}(\nabla \tilde{\bu}), && \tin Q_T^{\tilde{w}}, \\
		& \vr_s h \pt^2 \tilde{w} 
		+ {h^3 E} \Delta^2 \tilde{w} - \tilde{\nu}_s \Delta \pt \tilde{w} = - \big(\widetilde{\bbs}(\nabla \tilde{\bu}) \bn^{\tilde{w}} \big) \circ \bPhi_{\tilde{w}} \cdot \bn + (p(\tilde{\vr}) - p(\bar{\vr})), && \ton \Gamma_T, \\
		& (\tilde{\bu} \cdot \bn^{\tilde{w}}) \circ \bPhi_{\tilde{w}}
		= \pt \tilde{w}, \quad
        (\widetilde{\bbs}(\nabla \tilde{\bu}) \bn^{\tilde{w}} + \tilde{\alpha} (\tilde{\bu} - \pt \tilde{w} \bn \circ \bPhi_{\tilde{w}}^{-1}))_{\btau^{\tilde{w}}} \circ \bPhi_{\tilde{w}} 
		= 0, && \ton \Gamma_T, \\
        & \tilde{\bu} \cdot \be_3 = 0, \quad
        (\widetilde{\bbs}(\nabla \tilde{\bu}^\varepsilon) \be_3 + \tilde{\alpha}_0 \tilde{\bu}^\varepsilon)_{\btau^0} = 0, && \ton \Gamma_T^0, 
	\end{alignedat}
	\right.
\end{equation*}
with $ \widetilde{\bbs}(\nabla \tilde{\bu}) =  \tilde{\mu} (\nabla \tilde{\bu} + \nabla \tilde{\bu}^\top) + \tilde{\lambda} \Div \tilde{\bu} \bbi $. Here the constant $ \vr_s, h, E $ denotes the density, the thickness and Young's modulus of the elastic structure.
Then we get the system for $ (\vr^\varepsilon, \bu^\varepsilon, w^\varepsilon) $
	\begin{alignat*}{3}
            & \frac{\vr_{f}}{T_{f}} \pt \vr^\varepsilon + \frac{\vr_{f} U_{f}}{L} \Div (\vr^\varepsilon \bu^\varepsilon) = 0, && \tin Q_T^{w^\varepsilon}, \\ 
            & \frac{\vr_{f} U_{f}}{T_{f}} \pt (\vr^\varepsilon \bu^\varepsilon) + \frac{\vr_{f} U_{f}^2}{L} \Div (\vr \bu^\varepsilon \otimes \bu^\varepsilon) + \frac{p_{f}}{L} \nabla (p(\vr^\varepsilon) - p(\bar{\vr})) = \frac{U_{f} \nu_{f}}{L^2} \Div \bbs(\nabla \bu^\varepsilon), && \tin Q_T^{w^\varepsilon}, \\
            & \frac{\vr_{s} W^2}{T_{s}^2} \pt^2 w^\varepsilon 
            + \frac{W^4 E}{L^4} \Delta^2 w^\varepsilon - \frac{W N_{s}}{L^2 T_{s}} \nu_s \Delta \pt w^\varepsilon = \mathcal{F}^\varepsilon, && \ton \Gamma_T, \\
            & \mathcal{F}^\varepsilon = \frac{U_{f} \nu_{f}}{L} \big(\bbs(\nabla \bu^\varepsilon)  \bn^{w^\varepsilon} \big) \circ \bPhi_{w^\varepsilon}\cdot \bn^{w^\varepsilon}
            + p_{f}(p(\vr^\varepsilon) - p(\bar{\vr})), && \ton \Gamma_T, \\
            & (\bu^\varepsilon \cdot \bn^{w^\varepsilon}) \circ \bPhi_{w^\varepsilon}
		  = \pt w^\varepsilon, && \ton \Gamma_T, \\
            & \frac{U_{f} \nu_{f}}{L} (\bbs(\nabla \bu^\varepsilon) \bn^{w^\varepsilon})_{\btau^{w^\varepsilon}} \circ \bPhi_{w^\varepsilon}
            + \frac{\alpha N_f U_f}{L} \bu^\varepsilon_{\btau^{w^\varepsilon}} \circ \bPhi_{w^\varepsilon}
            = \frac{\alpha N_f W}{T_s L} (\pt w^\varepsilon \bn \circ \bPhi_{w^\varepsilon}^{-1})_{\btau^{w^\varepsilon}} \circ \bPhi_{w^\varepsilon}, && \ton \Gamma_T, \\
		  & \bu^\varepsilon \cdot \be_3 = 0, \quad
            \frac{U_{f} \nu_{f}}{L} (\bbs(\nabla \bu^\varepsilon) \bn^{w^\varepsilon} + \alpha_0 \bu^\varepsilon)_{\btau^0} \circ \bPhi_{w^\varepsilon}
            = 0, && \ton \Gamma_T,
	\end{alignat*}
by setting $ h = W $, meaning the thickness is of the same order of characteristic displacement.
Multiplying the first equation with $ \frac{T_f}{\vr_f} $ and the rest with $ \frac{1}{\vr_{f} U_{f}^2} $, one ends up with
	\begin{alignat*}{3}
		& \pt \vr^\varepsilon + \Div (\vr^\varepsilon \bu^\varepsilon) = 0, && \tin Q_T^{w^\varepsilon}, \\ 
		& \pt (\vr^\varepsilon \bu^\varepsilon) + \Div (\vr \bu^\varepsilon \otimes \bu^\varepsilon) + \frac{1}{\mathrm{Ma}^2} \nabla (p(\vr^\varepsilon) - p(\bar{\vr})) = \frac{1}{\mathrm{Re}} \Div \bbs(\nabla \bu^\varepsilon), && \tin Q_T^{w^\varepsilon}, \\
		& \frac{\vr_{s} U_s^2}{\vr_{f} U_f^2} \pt^2 w^\varepsilon 
		+ \frac{W^4 E}{L^4 \vr_{f} U_f^2} \Delta^2 w^\varepsilon 
		- \frac{W N_{s}}{L^2 T_{s} \vr_f U_f^2} \nu_s \Delta \pt w^\varepsilon = \mathcal{F}^\varepsilon, && \ton \Gamma_T, \\
		& \mathcal{F}^\varepsilon = \frac{1}{\mathrm{Re}} \big(\bbs(\nabla \bu^\varepsilon) \bn^{w^\varepsilon}\big) \circ \bPhi_{w^\varepsilon} \cdot \bn
        + \frac{1}{\mathrm{Ma}^2} (p(\vr^\varepsilon) - p(\bar{\vr})), && \ton \Gamma_T, \\
		& (\bu^\varepsilon \cdot \bn^{w^\varepsilon}) \circ \bPhi_{w^\varepsilon}
		= \pt w^\varepsilon, && \ton \Gamma_T, \\
		& \frac{1}{\mathrm{Re}} (\bbs(\nabla \bu^\varepsilon) \bn^{w^\varepsilon})_{\btau^{w^\varepsilon}} \circ \bPhi_{w^\varepsilon}
		= - \frac{\alpha}{\mathrm{Re}} (\bu^\varepsilon - \pt w^\varepsilon \bn \circ \bPhi_{w^\varepsilon}^{-1})_{\btau^{w^\varepsilon}} \circ \bPhi_{w^\varepsilon}, && \ton \Gamma_T^0, \\
		  & \bu^\varepsilon \cdot \be_3 = 0, \quad
            \frac{1}{\mathrm{Re}} (\bbs(\nabla \bu^\varepsilon) \bn^{w^\varepsilon} + \alpha_0 \bu^\varepsilon)_{\btau^0} \circ \bPhi_{w^\varepsilon}
            = 0, && \ton \Gamma_T,.
	\end{alignat*}
{If it is further assumed that $ \frac{\vr_{s} U_s^2}{\vr_{f} U_f^2} = \frac{W^4 E}{L^4 \vr_{f} U_f^2} = \frac{W N_{s}}{L^2 T_{s} \vr_f U_f^2} = 1 $, we recover the scaled system \eqref{eqs:FSI-model-rescale}. 
}

\begin{remark}
    
An alternative method is to define the following rescaled new valuables with the help of $ (\tilde{\vr}, \tilde{\bu}, \tilde{w}) $ for $ \varepsilon, \nu > 0 $:
\begin{gather*}
	\tilde{\vr}(t,x) = \vr^\varepsilon(\varepsilon t, x),\ 
	\tilde{\bu}(t,x) = \varepsilon \bu^\varepsilon(\varepsilon t, x),\ 
	\tilde{w}(t,x) = w^\varepsilon(\varepsilon t, x),\\
	\tilde{\mu} = \varepsilon \nu \mu,\ 
	\tilde{\lambda} = \varepsilon \nu \lambda,\ 
	\tilde{\nu}_s = \varepsilon \nu_s,\ 
	\tilde{\alpha} = \varepsilon^2 \nu \alpha,\
        \tilde{\alpha}_0 = \varepsilon^2 \nu \alpha_0,\
	\vr_s h = 1,\ 
	h^3 E = \varepsilon^2. 
\end{gather*} 
\end{remark}

\begin{remark}[Extra pressure]
    Note that in the scaled system \eqref{eqs:FSI-model-rescale}, the dynamic condition \eqref{eqs:F-scaled} contains an additional pressure term $p(\bar{\vr})$. There are two reasons behind this consideration. On the one hand, to pass to the limit of low Mach number, we need to compare the pressure to a constant ``pressure'' (rewrite $\nabla p(\vr)$ as $\nabla(p(\vr) - p(\bar{\vr}))$). On the other hand, in this case, the fluid-structure system is surrounded by air with a constant ``pressure'', which is reasonable from the physical point of view.
\end{remark}

\subsection{Target system (Euler-plate)}
\label{sec:target-system}
Motivated by \cite{Hoff1998}, now we proceed with a formal expansion of $ (\vr^\varepsilon, \bu^\varepsilon, w^\varepsilon, \pt w^\varepsilon) $ as
\begin{align*}
	& \vr^\varepsilon(x,t) = \bar{\vr} + \varepsilon^2 \tilde{\Pi}(x,t) + \zeta_1^\varepsilon(x,t),\\
	& \bu^\varepsilon(x,t) = \tilde{\bv}(x,t) + \zeta_2^\varepsilon(x,t),\\
	& w^\varepsilon(x,t) = \eta(x,t) + \zeta_3^\varepsilon(x,t),\\
	& \pt w^\varepsilon(x,t) = \pt \eta(x,t) + \zeta_4^\varepsilon(x,t),
\end{align*}
where $ \zeta_1^\varepsilon(x,t) = o(\varepsilon^2) $, $ \zeta_2^\varepsilon(x,t) = o(1) $, $ \zeta_3^\varepsilon(x,t) = o(1) $, $ \zeta_4^\varepsilon(x,t) = o(1) $, as $ \varepsilon \to 0 $.
Substituting above expansions into \eqref{eqs:FSI-model-rescale}, one could anticipate the limit system as $ \varepsilon \to 0 $ and $ \nu \to 0 $ to be
\begin{subequations}
	\label{eqs:EulerPlate}
	\begin{alignat}{3}
		\bar{\vr} (\pt \tilde{\bv} + \tilde{\bv} \cdot \nabla \tilde{\bv}) + \nabla \tilde{\Pi} & = 0, && \tin Q_{T'}^\eta, \\
		\Div \tilde{\bv} & = 0, && \tin Q_{T'}^\eta, \\ 
		\pt^2 \eta + \Delta^2 \eta - \nu_s \Delta \pt \eta & = \tilde{\Pi} \circ \bPhi_{\eta}, && \ton \Gamma_{T'}, \\
		(\tilde{\bv} \cdot \bn^\eta) \circ \bPhi_\eta
		& = \pt \eta, && \ton \Gamma_{T'}, \\
		\tilde{\bv} \cdot \be_3 & = 0, && \ton \Gamma_{T'}^0, \\
		(\tilde{\bv}, \eta, \pt \eta)(\cdot, 0)
		& = (\bv_0 \circ \inv{\bPsi_{0}}, \eta_0, \eta_1), && \tin \Omega^{\eta_0},
	\end{alignat}
\end{subequations}
for $ T' > 0 $ the existence time. 
The above system \eqref{eqs:EulerPlate} was investigated by Kukavica--Tuffaha \cite{KT2022} with a periodic slab-like geometry, and it is shown that \eqref{eqs:EulerPlate} admits a unique strong solution in $(0,T')$ for some small $ T' > 0 $. For simplicity, we assume
\begin{equation}
	\label{eqs:regularity-v-tilde}
	\begin{aligned}
		\tilde{\bv} & \in C^1([0,T']; C^2(\Omega^\eta(t))), \\
		\tilde{\Pi} & \in C^1([0,T']; C^1(\Omega^\eta(t))), \\
		\eta & \in C^2([0,T']; C^3(\Gamma)).
	\end{aligned}
\end{equation}
This is one of the motivations that we want to investigate the limiting behavior of our compressible fluid-structure interaction problems, providing a unique suitable regular strong solution.

\subsection{Scaled relative energy inequality}
\label{sec:relative-energy}
There are different ways to justify the singular limits in the context of compressible fluids, for example, uniform estimates with compactness argument, energy methods for strong solutions, as well as relative entropy methods, which are endowed with different properties. As a starting point to investigating the singular limits for fluid-structure interaction problems, we will employ modified relative entropy methods. 
Before approaching the limit passage, we recall the relative energy depending on the Mach number $\varepsilon$ and Reynolds number $\frac{1}{\nu}$. Following e.g. \cite{FN2017,FKNZ2016}, we define
\begin{equation*}
	\begin{aligned}
		\cE \big( &(\vr^\varepsilon, \bu^\varepsilon, w^\varepsilon)|(r, \bU, W) \big)(t) \\
		& \coloneqq \int_{\Omega^{w^\varepsilon}(t)} 
		\onehalf \vr^\varepsilon \abs{\bu^\varepsilon - \bU}^2 (t) \dx
		+ \frac{1}{\varepsilon^2 (\gamma - 1)} \int_{\Omega^{w^\varepsilon}(t)} \left( p(\vr^\varepsilon) - p'(r) (\vr^\varepsilon - r) - p(r) \right)(t) \dx \\
		& \quad + \int_{\Gamma} \left( \onehalf \abs{\pt w^\varepsilon - \pt W}^2 + \onehalf \abs{\Delta w^\varepsilon - \Delta W}^2 \right)(t) \dH^2,
	\end{aligned}
\end{equation*}
which consists of the differences in kinetic energy of fluids, pressure entropy, kinetic energy of structure, and elastic energy.
Based on it, one can show the following proposition of relative energy inequality, whose proof is postponed in Appendix \ref{sec:proof-REI}.
\begin{proposition}
	\label{prop:relative-energy}
	Let $ (\vr^\varepsilon, \bu^\varepsilon, w^\varepsilon) $ be a weak solution to \eqref{eqs:FSI-model-rescale} constructed in the sense of the Definition \ref{def:bounded-weak}. Then $ (\vr^\varepsilon, \bu^\varepsilon, w^\varepsilon) $ satisfies the following relative energy inequality 
	\begin{align}
		& \cE \big( (\vr^\varepsilon, \bu^\varepsilon, w^\varepsilon)|(r, \bU, W) \big)(\tau)
		+ \nu \int_{0}^{\tau} \int_{\Omega^{w^\varepsilon}(t)} \big( \bbs(\nabla \bu^\varepsilon) - \bbs(\nabla \bU) \big) : (\nabla \bu^\varepsilon - \nabla \bU) \dxdt 
		\nonumber \\
		& \quad + \alpha \nu \int_{0}^{\tau} \int_{\Gamma^{w^\varepsilon}(t)} \abs{(\bu^\varepsilon - \bU - (\pt w^\varepsilon - \pt W) \bn \circ \bPhi_{w^\varepsilon}^{-1})_{\btau^{w^\varepsilon}}}^2 \dH^2 \dt
		\nonumber \\
		& \quad + \alpha_0 \nu \int_{0}^{\tau} \int_{\Gamma_0} \abs{(\bu^\varepsilon - \bU)_{\btau^0}}^2 \dH^2 \dt
		+ \nu_s \int_{0}^{\tau} \int_{\Gamma} \abs{\nabla \pt w^\varepsilon - \nabla \pt W}^2 \dH^2 \dt
		\nonumber \\
		& \qquad \leq \cE \big( (\vr^\varepsilon, \bu^\varepsilon, w^\varepsilon)|(r, \bU, W) \big)(0) 
		+ \int_{0}^{\tau} \cR(\vr^\varepsilon, \bu^\varepsilon, w^\varepsilon, r, \bU, W)(t) \,\d t,
		\label{eqs:relative-energy-inequality-epsilon}
	\end{align}
	for a.e. $ \tau \in (0,T) $ and any pair of test functions $ (r, \bU, W) $ such that $ \bU \in C_c^\infty(\Bar{Q_T^{w^\varepsilon}}) $, $ W \in L^\infty(0,T; H^2(\Gamma)) \cap W^{1,\infty}(0,T; L^2(\Gamma)) \cap H^1(0,T; H^1(\Gamma)) $, $ \bU \cdot \bn^{w^\varepsilon} = \pt W \circ \bPhi_{w^\varepsilon}^{-1} $ on $ \Gamma^{w^\varepsilon} $, $ r \in C_c^\infty(\Bar{Q_T^{w^\varepsilon}}) $, $ r > 0 $, where the reminder term $ \cR $ is given by
	\begin{align*}
		& \quad \cR(\vr^\varepsilon, \bu^\varepsilon, w^\varepsilon, r, \bU, W)(t) \\
		& = - \int_{\Omega^{w^\varepsilon}(t)} \vr^\varepsilon (\bu^\varepsilon - \bU) \cdot ( \pt + \bu^\varepsilon \cdot \nabla ) \bU \dx 
		- \frac{1}{\varepsilon^2} \int_{\Omega^{w^\varepsilon}(t)} (p(\vr^\varepsilon) - p(r)) \Div \bU \dx \\
		& \quad
		+ \frac{1}{\varepsilon^2} \int_{\Gamma^{w^\varepsilon}(t)} (p(r) - p(\bar{\vr})) (\bu^\varepsilon - \bU) \cdot \bn^{w^\varepsilon} \dH^2 \\
		& \quad 
		- \nu \int_{\Omega^{w^\varepsilon}(t)} \bbs(\nabla \bU) : \big(\nabla \bu^\varepsilon - \nabla \bU\big) \dx 
		- \int_{\Gamma} (\pt w^\varepsilon - \pt W) (\pt^2 W + \Delta^2 W - \nu_s \Delta \pt W) \dH^2 \\
		& \quad - \frac{1}{\varepsilon^2} \int_{\Omega^{w^\varepsilon}(t)} \frac{1}{\gamma - 1} \left[ (\vr^{\varepsilon} - r) (\pt + \bU \cdot \nabla) p'(r) + \vr (\bu^\varepsilon - \bU) \cdot \nabla p'(r) \right] \dx \\
		& \quad - \alpha \nu \int_{\Gamma^w(t)} \big( \bU - \pt W \bn \circ \inv{\bPhi_{w^\varepsilon}} \big)_{\btau^{w^\varepsilon}} \cdot \big((\bu^\varepsilon - \bU) - (\pt w^\varepsilon - \pt W) \bn \circ \inv{\bPhi_{w^\varepsilon}}\big)_{\btau^{w^\varepsilon}} \dH^2 \\
		& \quad - \alpha_0 \nu \int_{\Gamma_0} \bU_{\btau^0} \cdot (\bu^\varepsilon - \bU)_{\btau^0} \dH^2.
	\end{align*}
\end{proposition}
\subsection{Uniform bounds}
\label{sec:uniform-bounds}
Simply letting $ r = \bar{\vr} $, $ \bU = 0 $, $ W = 0 $ in \eqref{eqs:relative-energy-inequality-epsilon}, one obtains 
\begin{align}
	\label{eqs:uniform-estimate-epsilon}
	& \cE \big( (\vr^\varepsilon, \bu^\varepsilon, w^\varepsilon)|(\bar{\vr}, 0, 0) \big)(\tau)
		+ \nu \int_{0}^{\tau} \int_{\Omega^{w^\varepsilon}(t)} \bbs(\nabla \bu^\varepsilon) : \nabla \bu^\varepsilon \dxdt
		+ \nu_s \int_{0}^{\tau} \int_{\Gamma} \abs{\nabla \pt w^\varepsilon}^2 \dH^2 \dt \\
		& \quad \qquad + \alpha \nu \int_{0}^{\tau} \int_{\Gamma^{w^\varepsilon}(t)} \abs{(\bu^\varepsilon-\partial_{t}w^{\varepsilon}{\bf n}\circ\bPhi^{-1}_{w^{\varepsilon}})_{\tau^{w^\varepsilon}}}^2 \dH^2 \dt
        + \alpha_0 \nu \int_{0}^{\tau} \int_{\Gamma_0} \abs{\bu^\varepsilon_{\btau^0}}^2 \dH^2 \dt\\
		&\leq \cE \big( (\vr^\varepsilon, \bu^\varepsilon, w^\varepsilon)|(\bar{\vr}, 0, 0) \big)(0),
        \nonumber
\end{align}
for a.e. $ \tau \in (0,T) $, where $ C > 0 $ is independent of $ \varepsilon $. This initial data in this contribution is assumed to be \textit{well-prepared} in the sense of
\begin{equation}
	\label{eqs:intial-well}
	\norm{\vr_0^{\varepsilon, 1}}_{L^\infty(\Omega^{w_0})}
	+ \norm{\bu_0^\varepsilon}_{L^\infty(\Omega^{w_0})}
	+ \norm{\Delta w_0^\varepsilon}_{L^2(\Gamma)}
	+ \norm{w_1^\varepsilon}_{L^2(\Gamma)}
	\leq D,
\end{equation}
and
\begin{subequations}
	\label{eqs:well-prepared-initial}
	\begin{alignat}{3}
		\bu_0^\varepsilon & \to \tilde{\bu}_0, && \tin L^2(\Omega^{\tilde{w}_0}), \label{U}\\
		\vr_0^{\varepsilon,1} & \to 0, && \tin L^\infty(\Omega^{\tilde{w}_0}),\label{rho} \\
		w_0^\varepsilon & \to \tilde{w}_0, && \tin W^{2,2}(\Gamma), \label{W)}\\
		w_1^\varepsilon & \to \tilde{w}_1, && \tin L^2(\Gamma),\label{W1}
	\end{alignat}
\end{subequations}
for some fixed constant $D>0$, which denotes the initial data perturbation. 
It follows from \eqref{eqs:uniform-estimate-epsilon} that
\begin{gather}
	\esssup_{t \in (0,T)} \norm{\sqrt{\vr^\varepsilon} \bu^\varepsilon}_{L^2(\Omega^{w^\varepsilon})} \leq C, 
	\label{eqs:uniform-bounds-rho-u} \\
	\sqrt{\nu} \norm{\nabla \bu^\varepsilon + \tran{(\nabla \bu^\varepsilon)} - \frac{2}{3} \Div \bu^\varepsilon \bbi}_{L^2(\Omega^{w^\varepsilon})} \leq C, 
	\label{eqs:uniform-bounds-nabla-u}\\
	\esssup_{t \in (0,T)} \int_{\bar{\vr}/2 \leq \vr^\varepsilon \leq 2 \bar{\vr}} \abs{\frac{\vr^\varepsilon - \bar{\vr}}{\varepsilon}}^2 \dx \leq C, 
	\label{eqs:uniform-bounds-rho-barrho-epsilon} \\
	\esssup_{t \in (0,T)} \int_{\vr^\varepsilon > 2 \bar{\vr}} (1 + \abs{\vr^\varepsilon}^\gamma) \dx
	+ \esssup_{t \in (0,T)} \int_{\vr^\varepsilon < \bar{\vr}/2} (1 + \abs{\vr^\varepsilon}^\gamma) \dx \leq \varepsilon^2 C, 
	\label{eqs:uniform-bounds-1+rho-epsilon} \\
	\esssup_{t \in (0,T)} \big(\norm{\pt w^\varepsilon}_{L^2(\Omega^{w^\varepsilon})} + \norm{\Delta w^\varepsilon}_{L^2(\Omega^{w^\varepsilon})}\big) \leq C, 
	\label{eqs:uniform-bounds-w} \\
	\sqrt{\nu_s} \norm{\nabla \pt w^\varepsilon}_{L^2(\Omega^{w^\varepsilon})} \leq C.
	\label{eqs:uniform-bounds-nabla-dt-w}
\end{gather}
The inequalities \eqref{eqs:uniform-bounds-rho-barrho-epsilon} and \eqref{eqs:uniform-bounds-1+rho-epsilon} are consequences of the following:
\begin{equation}
p(\rho^{\varepsilon})-p'(\overline{\rho})(\rho^{\varepsilon}-\overline{\rho})-p(\overline{\rho})\geq\left\{\begin{array}{ll}
C(\overline{\rho})(\rho^{\varepsilon}-\overline{\rho})^{2},& \displaystyle \frac{\overline{\rho}}{2}<\rho^{\varepsilon}<2\overline{\rho},\\
C(\overline{\rho})(1+(\rho^{\varepsilon})^{\gamma}),&\displaystyle \mbox{otherwise}.
\end{array}\right.
\end{equation}
Note that \eqref{eqs:uniform-bounds-1+rho-epsilon} implies
\begin{equation}
	\label{eqs:uniform-bounds-rho-epsilon}
	\esssup_{t \in (0,\tau)} \int_{\vr^\varepsilon > 2 \bar{\vr}} \abs{\vr^\varepsilon}^\gamma \dx
	+ \esssup_{t \in (0,\tau)} \int_{\vr^\varepsilon < \bar{\vr}/2} \abs{\vr^\varepsilon}^\gamma \dx \leq \varepsilon^2 C, 
\end{equation}
Here $C > 0$ depends on $D > 0$.

As in \cite{FN2017}, it is convenient to introduce the following notation: for each measurable function $ f(t,x) $ we write $ f = \ess{f} + \res{f} $, where
\begin{equation*}
	\ess{f} = \chi(\vr) f, \quad
	\res{f} = (1 - \chi(\vr)) f, \quad
	\chi(\vr) = \left\{
		\begin{aligned}
			& 1 \quad && \text{for } \bar{\vr}/2 \leq \vr \leq 2 \bar{\vr}, \\
			& 0 \quad && \text{otherwise}.
		\end{aligned}
	\right.
\end{equation*}

\subsection{Transformation between two deformed domains}
\label{sec:tranformation}
To construct suitable test functions based on the strong solution defined in $ \Omega^\eta $, one needs to make sure it is possible to compare two solutions in different domains.
In this section, we are going to construct a simple composition of two maps, which are defined on top of the displacements $ w $ and $ \eta $ respectively. 

\begin{figure}[htbp!]
	\centering
	\begin{tikzpicture}[line width=0.5pt,scale=0.8]
		
		\draw plot[smooth cycle, tension=0.7] coordinates{(0.1,0.1) (1,1.5) (1.8,2) (2.8,1.2) (3.6,1) (3.8,-1) (3,-1.35) (1,-2) (0.5,-1)};
		\draw [dashed] plot[smooth cycle, tension=0.7] coordinates{(0,0) (0.2,1) (1,1.25) (2,2) (3,1.5) (3.75,0.6) (4,-0.3) (3.3,-0.7) (3,-1.75) (2,-2.05) (0.5,-1)};

		\draw [shift ={(8,0)}] plot[smooth cycle, tension=0.7] coordinates{(0.25,0.2) (1,1.4) (2,1.5) (3,1.75) (4,-1) (2,-1.5) (0.5,-1.5)};
		\draw [dashed,shift ={(8,0)}] plot[smooth cycle, tension=0.7] coordinates{(0,0) (0.2,1) (1,1.25) (2,2) (3,1.5) (3.75,0.6) (4,-0.3) (3.3,-0.7) (3,-1.75) (2,-2.05) (0.5,-1)};
		
		\draw [shift ={(4,-4)}] plot[smooth cycle, tension=0.7] coordinates{(0,0) (0.2,1) (1,1.25) (2,2) (3,1.5) (3.75,0.6) (4,-0.3) (3.3,-0.7) (3,-1.75) (2,-2.05) (0.5,-1)};
		
		\draw [-stealth] (3.9,-4.5) to [bend left=35] (1.9,-2.5);
		\draw node at (2.85,-3.35) {$\bPhi_w$}; 
		\draw [-stealth] (8.1,-4.5) to [bend right=35] (10.1,-2.5);
		\draw node at (9.15,-3.35) {$\bPhi_\eta$}; 
		\draw [-stealth] (3.8,1.5) to [bend left=20] (8.2,1.5);
		\draw node at (6,1.25) {$\bPhi_\eta \circ \bPhi_w^{-1}$}; 
		
		\draw node at (6,-4) {$\Omega$};
		\draw node at (8,-6.3) {$\partial \Omega = \Phi(\Gamma)$};
		\draw node at (2,0) {$\Omega^w$};
		\draw node at (-0.5,0) {$\Gamma^w$};
		\draw node at (10,0) {$\Omega^\eta$};
		\draw node at (12.5,0) {$\Gamma^\eta$};
		
	\end{tikzpicture}
	\caption{Construction of the mapping between two deformed configurations}
	\label{fig:mapping-two}
\end{figure}

Let the neighborhood of $ \partial \Omega $ be $ \cN_a^b $ for $ a,b > 0 $ as in Section \ref{sec:geometric-setting}.
Given $ \bPhi_w : \Omega \rightarrow \Omega^w $ and $ \bPhi_\eta : \Omega \rightarrow \Omega^\eta $, which are two transformations with respect to the displacements $ w, \eta $ respectively defined in the same fashion as in Section \ref{sec:geometric-setting}. 
Define
\begin{equation*}
	\cN_w \coloneqq \{\bX^w \in \bbr^3: \bPhi_w^{-1}(\bX^w) \in \cN_a^b\}.
\end{equation*}
Now we construct a mapping $ \bPsi : \Omega^w \rightarrow \Omega^\eta $ such that
\begin{equation}
	\label{eqs:mapping-Psi}
	\bPsi(\bX^w,t) = \bPhi_\eta \circ \bPhi_w^{-1} (\bX^w,t),
\end{equation}
fulfilling $ \bPsi_{|\Gamma^w} = \bX^w + \eta(\inv{\bPhi_w}(\bX^w)) \bn(\inv{\bPhi_w}(\bX^w)) $ and $ \bPsi_{|(\Omega^w(t) \setminus \cN_w)} = \bX^w $ for $\bX^w \in \Omega^w(t)$, see Figure \ref{fig:mapping-two}.
Then we subsequently have the associated quantities as
\begin{equation*}
	\nabla \bPsi = \nabla \bPhi_\eta \circ \inv{\bPhi_w} \nabla \inv{\bPhi_w}, \quad 
	J \coloneqq \det (\nabla \bPsi) = J_\eta \circ \inv{\bPhi_w} \inv{J_w}, \quad 
	\cA \coloneqq \inv{(\nabla \bPsi)}.
\end{equation*}

Under the setting above, we discuss briefly the regularity of the transformations in the manuscript, which guarantees the nonlinear estimates later. 
%
Since $ \eta $ is supposed to be sufficiently smooth, by construction $ \bPsi $ inherits the regularity from $ w $ automatically. Namely, with \eqref{eqs:nabla-Phi_w-1}, \eqref{eqs:pt-Phi_w}, and Proposition \ref{prop:more-spatial-regularity}, we conclude the following lemma. It is noted that in the following, $ L_w^p $ will simply denote the Lebesgue space $ L^p(\Omega^w(t)) $ depending on $ w $.
\begin{lemma}
	\label{lem:bPsi}
	Let $\bPsi$ be defined as in \eqref{eqs:mapping-Psi}. Then it satisfies
	\begin{align*}
		\bPsi 
		& \in L^\infty(0,T; H^2(\Omega^w(t)), 
		\\
		\pt \bPsi 
		& \in L^\infty(0,T; L^2(\Omega^w(t)))\cap L^{2}(0,T;L^r(\Omega^{w}(t)))  \text{ for all } 1 \leq r < 4, 
		\\
		\nabla \bPsi 
		& \in L^\infty(0,T; L^p(\Omega^w(t))) \text{ for all } 1 \leq p < \infty,
	\end{align*}
	Depending on $\Omega$, $\nu_s$, $\gamma$, it follows that
	\begin{align*}
		\pt \bPsi 
		& \in L^2(0,T; H^1(\Omega^w(t))) 
		\cap L^2(0,T; L^p(\Omega^w(t))) \text{ for all } 1 \leq p < \infty, \\
		\nabla \bPsi 
		& \in L^2(0,T; L^\infty(\Omega^w(t))),\\
		\nabla^2 \bPsi 
		& \in L^2(0,T; H^s(\Omega^w(t))) \text{ for some } 0 \leq s < \onequater.
	\end{align*}
	Moreover, for a.e. $ t \in (0,T) $, it holds the following Lipschitz-type estimates
	\begin{align*}
		\norm{\pt \bPsi(t)}_{L_w^q} 
		& + \norm{\nabla \bPsi(t) - \bbi}_{L_w^p}
		+ \norm{\cA(t) - \bbi}_{L_w^p} \\
		& + \norm{\nabla^2 \bPsi(t)}_{L_w^q}
		+ \norm{\nabla \cA(t)}_{L_w^q}
		\leq C \big(\norm{w(t) - \eta(t)}_X\big),
	\end{align*}
	for all $1 \leq q < 2$, $1 \leq p < \infty$.  If $ \nu_s > 0 $,
	\begin{equation*}
		\int_0^t 
		\norm{\pt \bPsi(s)}_{L_w^p}^2 
		\,\d s
		\leq C \int_0^t \Big(
		\norm{\pt w(s) - \pt \eta(s)}_{L_\Gamma^p}^2 
		+ \norm{\nabla w(s) - \nabla \eta(s)}_{L_\Gamma^p}^2
		\Big) \,\d s.
	\end{equation*}
	Here $ C > 0 $ depends on the solution $ (\vr, \bu, w) $ and $ \normm{\pt \eta}_{L_t^\infty W_\Gamma^{2,2}} $, $ \normm{\nabla \eta}_{L_t^\infty W_\Gamma^{2,2}} $. 
\end{lemma}
\begin{proof}
	The proof is a direct consequence of the construction of $ \bPsi $ and the regularities of $ w $ we have. So we omit the details here. Similar arguments can also be found in e.g. \cite{Trifunovic2023}.
\end{proof}
\begin{remark}
	We remark that the dependence on the regularity of $ \eta $ in Lemma \ref{lem:bPsi} is not optimal, in the sense of the lowest regularity requirement. In other words, it is sufficient for us to continue, as we always assume a sufficiently regular strong solution for the argument. In a very recent work by Breit--Mensah--Schwarzacher--Su \cite{BMSS2023}, they considered a Ladyzhenskaya--Prodi--Serrin condition in the context of a fluid-structure interaction problem with incompressible Navier--Stokes equations and a viscoelastic flexible shell.
\end{remark}
By a similar chain rule, we have the following regularities for any smooth function under the transformation $\bPsi$.
\begin{corollary}[Regularity loss]
	\label{coro:regularity-loss}
	Let $ f $ be a function lying in $ L^s(0,T; W^{1,\infty}(\Omega^\eta(t)) \cap W^{2,2}(\Omega^\eta(t)))$, $\pt f \in L^2(0,T; L^s(\Omega^\eta(t)))$ for some $1 \leq s \leq \infty$. Then under the transformation $\bPsi$ defined in \eqref{eqs:mapping-Psi}, the transformed function $g = f \circ \bPsi$ lies in the space
	\begin{equation*}
		\nabla g \in L^s(0,T; L^p(\Omega^w(t))), \quad 
		\nabla^2 g \in L^s(0,T; L^q(\Omega^w(t))), \quad
		\pt g \in L^1(0,T; L^k(\Omega^w(t)))
	\end{equation*}
	for all $ 1 \leq p < \infty $, $1 \leq q < 2$, and $1 \leq k < 4$ if $\nu_s = 0$, $\Omega \in \bbtwo$, while $1 \leq k < \infty$ if $\nu_s > 0$.
\end{corollary}
\begin{remark}
	Here $\pt g = \pt f + \pt \bPsi \cdot \nabla f$. In the case of $\Omega \in \bbtwo$ and $\nu_s = 0$, it follows from Proposition \ref{prop:more-spatial-regularity} that $\pt w \in L^2(0,T;L^k(\Gamma))$, $1 \leq k < 4$, which also holds for $\pt \bPsi \in L^2(0,T;L^k(\Omega^w(t)))$. 
\end{remark}
\subsection{Singular limits}
\label{sec:singular-limits-final}
Now we are ready to state the detailed version of Theorem \ref{thm:singular-limit} and justify it by passing to the limit of low Mach number and high Reynolds number, i,e, $ \varepsilon \rightarrow 0 $, $ \nu \rightarrow 0 $. 
\begin{theorem}
    \label{thm:singular-limit-full}
    Let $\Omega \in \bbtwo$, $ \nu_s > 0 $, $ \gamma > \frac{3}{2} $. Assume that the initial data \eqref{eqs:SL-initial-data} is well-prepared satisfying \eqref{eqs:intial-well} and \eqref{eqs:well-prepared-initial}. Let $ (\vr^\varepsilon, \bu^\varepsilon, w^\varepsilon) $ be a weak solution to the rescaled compressible fluid-structure interaction problem \eqref{eqs:FSI-model-rescale} constructed in Theorem \ref{thm:weak}, possessing an additional regularity
    \begin{equation}
        \label{eqs:w-assumption-thm-3}
        \nabla w^\varepsilon \in L^2(0,T; W^{1,\mathfrak{p}}(\Gamma))
        \cap W^{1,2}(0,T; L^\fp(\Gamma)) \text{ for some } \mathfrak{p} > \max\left\{4,\frac{2\gamma}{\gamma - 1}\right\}.
    \end{equation}
    Let $ (\tilde{\bu}, \tilde{w}) $ be the unique strong solution of \eqref{eqs:EulerPlate} defined on the interval $ (0,T') $ subjected to initial data $ (\tilde{\bu}, \tilde{w}, \pt \tilde{w})(0, \cdot) = (\tilde{\bu}_0 \circ \bPhi_{\tilde{w}_0 - w_0^\varepsilon}, \tilde{w}_0, \tilde{w}_1) $.
    Then it follows
    \begin{align}
    	\nonumber
    	& \sup_{0 \leq t \leq T'}
    		\int_{\Omega^{w^\varepsilon}(t)} \bigg(
	    		\vr^\varepsilon \abs{\bu^\varepsilon - \tilde{\bu}}^2 
	    		+ \abs{\frac{\vr^\varepsilon - \bar{\vr}}{\varepsilon}}^\gamma
	    	\bigg)(t) \dx \\
	    \nonumber
    	& \qquad \qquad + \sup_{0 \leq t \leq T'} \int_\Gamma 
    		\bigg(
    			\abs{\pt w^\varepsilon - \pt \tilde{w}}^2
    			+ \abs{\Delta w^\varepsilon - \Delta \tilde{w}}^2
    		\bigg)(t) \dH^2 \\
    	\label{eqs:singular-Limit-convergence}
    	& \qquad \leq C(T',D) \Bigg[(\varepsilon + \nu)
    		+ \int_{\Omega^{w_0^{\varepsilon}}} \bigg(
    			\abs{\bu_0^\varepsilon - \tilde{\bu}_0 \circ \bPhi_{\tilde{w}_0 - w_0^\varepsilon}}^2
    			+ \abs{\vr_0^{\varepsilon,1}}^2
    		\bigg) \dx \\
    	\nonumber
    	& \phantom{\qquad \leq C(T',D) \Bigg[(\varepsilon + \nu)} \, + \int_\Gamma 
    		\bigg(
    			\abs{w_1^\varepsilon - \tilde{w}_1}^2
    			+ \abs{\Delta w_0^\varepsilon - \Delta \tilde{w}_0}^2
    		\bigg) \dH^2 \Bigg],
    \end{align}
	where the constant $ C = C(T',D) $ depends on the norm of the limit solution $ (\tilde{\bu}, \tilde{\eta}) $, $\nu_s > 0$, and on the size of the initial data perturbation. 
 
\end{theorem}
\begin{remark}
	{As a consequence of Theorem \ref{thm:singular-limit-full}, in the regime $ \varepsilon \to 0 $, $ \nu \to 0 $, since the initial data is \textit{well-prepared} as in \eqref{eqs:well-prepared-initial},
	it implies the convergence. }
\end{remark}
\begin{remark}\label{rmk:regularity-loss}
    Note that, by the interpolation and embeddings, e.g. \cite[Lemma 2.3]{AL2023a} and \cite[Theorem 4.6.1]{Triebel1978}, it follows from \eqref{eqs:w-assumption-thm-3}, that 
    \begin{equation}
        \label{eqs:w-lipschitz}
        \begin{aligned}
            \nabla w^\varepsilon \in C\big([0,T]; (L^\fp(\Gamma), W^{1,\fp}(\Gamma))_{1-1/2,2}\big)
            = C([0,T]; B_{\fp, 2}^{1-1/2}(\Gamma))
            \hookrightarrow
            C([0,T]; C(\Gamma)),
        \end{aligned}
    \end{equation}
    for $1 > 1/2 + 2/\fp$, which means the boundary of $\Omega^w(t)$ is uniform Lipschitz for all $t$. Here $(X,Y)_{\theta,p}$ denotes the real interpolation space of $X,Y$ with index $\theta$, and $B_{\fp \mathfrak q}^{s}(\Omega)$ is the standard Besov space, cf. \cite{Triebel1978}. Hence, for $\fp > 4$, there will be no spatial regularity loss compared to Corollary \ref{coro:regularity-loss}. Namely, the results in Corollary \ref{coro:regularity-loss} can be improved to
    \begin{equation}
        \label{eqs:no-regularity-loss}
        \nabla g \in L^\infty(0,T; L^\infty(\Omega^{w^\varepsilon}(t))) \cap L^2(0,T;W^{1,2}(\Omega^{w^\varepsilon}(t))), \pt g \in L^2(0,T; W^{1,2}(\Omega^{w^\varepsilon}(t)))
    \end{equation}
    uniformly in $\varepsilon > 0$ for any $g = f \circ \bPsi^\varepsilon$, $f \in L^\infty(0,T; W^{1,\infty}(\Omega^\eta(t)) \cap W^{2,2}(\Omega^\eta(t)))$, and $\pt f\in L^2(0,T; W^{1,2}(\Omega^\eta(t)))$, depending on \eqref{eqs:w-assumption-thm-3}. Moreover, with the same $\fp$ as in \eqref{eqs:w-assumption-thm-3}, one has
    \begin{equation*}
        \pt g, \nabla g \in L^2(0,T; W^{1,\fp}(\Omega^{w^\varepsilon}(t)))
    \end{equation*}
    if $\pt f, \nabla f \in L^2(0,T; W^{1,\fp}(\Omega^{\eta}(t)))$.
\end{remark}

\begin{remark}
    The additional regularity assumption \eqref{eqs:w-assumption-thm-3} can be achieved by considering fractional Laplacian $(-\Delta)^{s+1}$ for some $s>0$. In other words, if one substitutes the Laplacian operator and the bi-Laplacian operator in the elastic equation with $(-\Delta)^{s+1}$ and $(-\Delta)^{s+2}$, we would obtain from the energy estimates that $w \in L^\infty(0,T;H^{2+s}(\Gamma)) \cap W^{1,\infty}(0,T;H^{1+s}(\Gamma))$ and no more assumptions needed. Note that the definition of the fractional Laplacian varies depending on the contexts, which agree with each other in the Hilbert space setting, see for example \cite{DPV2012}. Then choosing $s > 1 - \frac{2}{\fp}$ entails the desired regularity by Sobolev embeddings in two dimensions. In this case, the existence of weak solutions in Section \ref{sec:weak} can be extended easily with the same arguments, since the fractional Laplacian is a linear operator. 
    
    Alternatively, one may consider the generalized $\fp$-biLaplacian energy and $\fp$-Laplacian damping with the same $\fp$ as above, 
    \begin{equation*}
        \pt^2 w + \Delta (\Div (\absm{\Delta w}^{\fp-2} \nabla w)) 
        - \nu_s \Div(\absm{\nabla \partial_{t}w}^{\fp-2} \nabla \partial_{t}w) = \cF.
    \end{equation*}
    Then by testing the above equation with $\pt w$, we get
    $w \in L^\infty(0,T;W^{2,\fp}(\Gamma)) \cap W^{1,2}(0,T;W^{1,\fp}(\Gamma))$ fulfilling the requirement. In this situation, the existence of weak solutions is not trivial and cannot be extended naturally, as the solid equation now is nonlinear. Thus, one shall employ more compactness arguments to find the weak solutions.
\end{remark}

\begin{remark}
    It is expected that additional regularity assumptions are necessary up to now in the context of inviscid limit for fluid-structure interaction problems as the regularity of the weak solution is not sufficient to justify the limit, cf. Remark \ref{rmk:not-avoid-assumtion}. If only the incompressible limit is considered, we can get rid of these assumptions with a bit of sacrifice of $\gamma$, cf. Theorem \ref{coro:IncompressibleLimit}.
\end{remark}

\subsubsection{Construct suitable test functions}
\label{sec:construct-test-function}
In order to prove Theorem \ref{thm:singular-limit-full}, we would like to employ the scaled relative energy inequality \eqref{eqs:relative-energy-inequality-epsilon}. 
Since the limit system is defined in the domain $ \Omega^\eta $, one needs to transform it back to $ \Omega^w $ concerning the weak solution, to compare it with the weak solution of the compressible fluid-structure interaction problem. A naive idea is to take the same strategy as in Section \ref{sec:tranformation}, namely, to define the transformation between two domains by
\begin{equation*}
    \bPsi^\varepsilon \coloneqq \bPhi_\eta \circ \inv{\bPhi_{w^\varepsilon}}
    = \bPhi_{\eta - w^\varepsilon}, \quad 
    J^\varepsilon \coloneqq \det(\nabla \bPsi^\varepsilon), \quad 
    \cA^\varepsilon \coloneqq \inv{(\nabla \bPsi^\varepsilon)}.
\end{equation*}
Note that all these share the same uniform (in $\varepsilon > 0$) regularities as in Lemmata \ref{lem:bPsi}.

However, simple analog does not apply here directly and it is twofold. First, observing that in relative entropy inequality, the test function should be admissible in the sense of boundary condition $ (\bU \cdot \bn^{w^\varepsilon}) \circ \bPhi_\varepsilon = \pt W $ holds on $ \Gamma^{w^\varepsilon} $, as in Section \ref{sec:tranformation}. Moreover, the limit system is incompressible, meaning the test function should be divergence-free. Otherwise, when deriving the uniform in $\varepsilon$ Lipschitz-like estimates for the relative entropy, we would need to control terms like $\Div \bv$, which are not quadratic. Hence, one must correct the transform $\bPsi^\varepsilon$ so that the constructed test function is admissible.

To handle this problem, we take the well-known \textit{Piola transform}, cf. \cite{Ciarlet1988}, which is endowed with two important properties: (i) it preserves the nullity of normal components; (ii) it maps divergence-free vectors to divergence-free vectors.
A similar idea can be found, e.g., in \cite{GGP2012,SS2022} for simple 2D/1D fluid-structure interaction problems, in \cite{EH2013} for an isogeometric analysis of a Darcy--Stokes--Brinkman system.
To this end, we define
\begin{equation}
    \label{eqs:new-transformation}
    \bv = J^\varepsilon \cA^\varepsilon \tilde{\bv} \circ \bPsi^\varepsilon, \quad 
    \Pi = \tilde{\Pi} \circ \bPsi^\varepsilon.
\end{equation}
By Remark \ref{rmk:regularity-loss}, \eqref{eqs:regularity-v-tilde}, and Sobolev embedding, we know
\begin{equation}
    \label{eqs:tilde-v-psi}
    \begin{gathered}
        \tilde{\bv} \circ \bPsi^\varepsilon \in L^\infty(0,T'; W^{1,\infty}(\Omega^{w^\varepsilon}(t))), \quad
        \pt (\tilde{\bv} \circ \bPsi^\varepsilon), \nabla (\tilde{\bv} \circ \bPsi^\varepsilon) \in L^2(0,T';W^{1,\mathfrak{p}}(\Omega^{w^\varepsilon}(t))), \\
        \Pi \in L^\infty(0,T'; W^{1,\infty}(\Omega^{w^\varepsilon}(t))),
    \end{gathered}
\end{equation}
for all $\mathfrak{p} > 4$.
In view of the Piola identity $\Div (J^\varepsilon (\cA^\varepsilon)^\top
) = 0$, and the identity $\Div (\mathbf T \mathbf v) = {\mathbf{T}}^\top : \nabla {\mathbf{v}} + \mathbf{v} \cdot \Div \mathbf{T}^\top$ for $\mathbf{T} \in \bbr^{3 \times 3}$, $\mathbf{v} \in \bbr^3$,
\begin{equation}
    \label{eqs:div-v=0}
    \Div \bv 
    = \Div (J^\varepsilon (\cA^\varepsilon)^\top) \tilde{\bv} \circ \bPsi^\varepsilon
    + J^\varepsilon (\cA^\varepsilon)^\top : \nabla (\tilde{\bv} \circ \bPsi^\varepsilon)
    = J^\varepsilon (\Div \tilde{\bv}) \circ \bPsi^\varepsilon = 0.
\end{equation}
In addition, the deformation of the normal vector entails
\begin{equation*}
    \bn^\eta \circ \bPsi^\varepsilon = J^\varepsilon \tran{(\cA^\varepsilon)} \bn^{w^\varepsilon}, \text{ on } \Gamma^{w^{\varepsilon}}.
\end{equation*}
Then one obtains the boundary condition on $\Gamma_{T'}$
\begin{equation}
    \label{eqs:v-n-w}
    \bv \cdot \bn^{w^\varepsilon}
    = J^\varepsilon \cA^\varepsilon (\tilde{\bv} \circ \bPsi^\varepsilon)
    \cdot 
    \bn^{w^\varepsilon}
    = (\tilde{\bv} \circ \bPsi^\varepsilon) \cdot (
    J^\varepsilon (\cA^\varepsilon)^\top \bn^{w^\varepsilon})
    = (\tilde{\bv} \cdot \bn^\eta) \circ \bPsi^\varepsilon
    = \pt \eta \circ \bPhi_{w^\varepsilon}.
\end{equation}

With \eqref{eqs:new-transformation} and identities above, we arrive at the transformed system of \eqref{eqs:EulerPlate} in the time-space cylinder denoted by $ Q_{T'}^{w^\varepsilon} \coloneqq \bigcup_{t \in (0,T')} \Omega^{w^\varepsilon} (t) \times \{t\}$,
\begin{subequations}
	\label{eqs:EulerPlate-transformed}
	\begin{alignat}{3}
		\bar{\vr} (\pt \bv + \bv \cdot \nabla \bv) + \nabla \Pi & = \mathbf{F}^\varepsilon, && \tin Q_{T'}^{w^\varepsilon}, \\
		\Div \bv & = 0, && \tin Q_{T'}^{w^\varepsilon}, \\ 
		\pt^2 \eta + \Delta^2 \eta - \nu_s \Delta \pt \eta & = \Pi \circ \bPhi_{w^\varepsilon}, && \ton \Gamma_{T'}, \\
		\pt \eta - (\bv \cdot \bn^{w^\varepsilon}) \circ \bPhi_{w^\varepsilon}
		& = 0, && \ton \Gamma_{T'}, \\
        \bv \cdot \be_3 & = 0, && \ton \Gamma_{T'}^0, \\
		(\bv, \eta, \pt \eta)(\cdot, 0)
		& = (\bv_0, \eta_0, \eta_1), && \tin \Omega^{w_0^\varepsilon},
	\end{alignat}
\end{subequations}
where
\begin{align*}
    \mathbf{F}^\varepsilon & = \bar{\vr} J^\varepsilon \cA^\varepsilon (\pt \bPsi^\varepsilon \cdot (\nabla \tilde{\bv}) \circ \bPsi^\varepsilon)
    + \bar{\vr} \pt(J^\varepsilon \cA^\varepsilon) (\tilde{\bv} \circ \bPsi^\varepsilon) \\
        & 
        \quad + \bar{\vr} \bv \cdot
         \nabla (J^\varepsilon \cA^\varepsilon) (\tilde{\bv} \circ \bPsi^\varepsilon)
        + \bar{\vr} (J^\varepsilon - 1) J^\varepsilon \cA^\varepsilon (\tilde{\bv} \cdot \nabla \tilde{\bv}) \circ \bPsi^\varepsilon \\
        & 
	\quad + (\inv{(\cA^\varepsilon)} - \bbi) (\nabla \tilde{\Pi}) \circ \bPsi^\varepsilon).
\end{align*}
At this stage, we would like to exploit the uniform estimate of $\bF^\varepsilon$ in terms of $w^\varepsilon - \eta$. In view of Lemma \ref{lem:bPsi}, most of the terms can be handled, except for terms with $\pt(J^\varepsilon \cA^\varepsilon)$ and $\nabla(J^\varepsilon \cA^\varepsilon)$. In general, one cannot expect good control of it if we ignore the viscous damping $\nu_s \Delta \pt w^\varepsilon$. More precisely, with $\nu_s > 0$, we know $\sqrt{\nu_s} \nabla \pt w^\varepsilon \in L_t^2L_\Gamma^2$. Then one concludes from Remark \ref{rmk:regularity-loss} with a bit of calculations
\begin{align*}
    \int_0^t \norm{\pt \cA^\varepsilon(\tau)}_{L^2(\Omega^{w^\varepsilon})}^2 \dtau
    & \leq C \int_0^t \nu_s \normm{\nabla \pt w^\varepsilon(\tau) - \nabla \pt \eta(\tau)}_{L_\Gamma^2}^2 \dtau \\
    & \quad + C \int_0^t \big(
    \normm{\pt w^\varepsilon(\tau) - \pt \eta(\tau)}_{L_\Gamma^2}^2 + \normm{\Delta w^\varepsilon(\tau) - \Delta \eta(\tau)}_{L_\Gamma^2}^2\big) \dtau,
\end{align*}
which holds for $\pt J^\varepsilon$ as well. Here the argument strongly relies on the additional regularity \eqref{eqs:w-assumption-thm-3}, i.e., $\nabla w \in L_t^\infty L_\Gamma^\infty$ (and hence $\cA^\varepsilon, J^\varepsilon \in L_t^\infty L_{w^\varepsilon}^\infty$). 
Therefore, with $\nabla w^\varepsilon \in L^\infty(0,T;W^{1,2}(\Gamma))$,
\begin{equation}
    \label{eqs:F^epsilon-L^2}
    \begin{aligned}
        \int_0^t \norm{\mathbf{F}^\varepsilon(\tau)}_{L^2(\Omega^{w^\varepsilon})}^2 \dtau 
        & \leq C \int_0^t \Big(\normm{\nabla \pt w^\varepsilon(\tau) - \nabla \pt \eta(\tau)}_{L_\Gamma^2}^2 \\
        & \qquad \qquad
        + \normm{\pt w^\varepsilon(\tau) - \pt \eta(\tau)}_{L_\Gamma^2}^2
        + \normm{\Delta w^\varepsilon(\tau) - \Delta \eta(\tau)}_{L_\Gamma^2}^2\Big) \dtau,
    \end{aligned}
\end{equation}
where $C > 0$ depends on \eqref{eqs:w-assumption-thm-3}, but is independent of $\varepsilon > 0$.

Moreover, the assumption \eqref{eqs:w-assumption-thm-3} helps us to exploit better uniform bounds for $\bF^\varepsilon$, which is essentially necessary to handle the residual term in \eqref{eqs:rhovF^varepsilon}. For $ \fp > \max\left\{4,\frac{2\gamma}{\gamma - 1}\right\} $, we have
\begin{equation}
    \label{eqs:F^epsilon-L^p}
    \begin{aligned}
        \int_0^t \norm{\mathbf{F}^\varepsilon(\tau)}_{L_{w^\varepsilon}^\fp}^2 \dtau
        & \leq C \int_0^t \big(
            \normm{\pt \bPsi^\varepsilon(\tau)}_{L_{w^\varepsilon}^\fp}^2
            + \normm{\nabla J^\varepsilon(\tau)}_{L_{w^\varepsilon}^\fp}^2
            + \normm{\nabla \cA^\varepsilon(\tau)}_{L_{w^\varepsilon}^\fp}^2
            + 1
        \big) \dtau \\
        & \leq C \int_0^t \big(
            \normm{\pt \nabla w^\varepsilon(\tau)}_{L_{w^\varepsilon}^\fp}^2
            + \normm{\Delta w^\varepsilon(\tau)}_{L_{w^\varepsilon}^\fp}^2
            + 1
        \big) \dtau
        \leq C,
    \end{aligned}
\end{equation}
where $C > 0$ depends on \eqref{eqs:w-assumption-thm-3}, but is independent of $\varepsilon > 0$.

At the moment we are ready to justify the incompressible inviscid limit.
\subsubsection{Proof of Theorem \ref{thm:singular-limit-full}}
\label{sec:proof-theorem-limit}
Let $ r = \bar{\vr} $, $ \bU = \bv $, $ W = \eta $ in \eqref{eqs:relative-energy-inequality-epsilon}. It follows from \eqref{eqs:EulerPlate-transformed} that
\begin{align*}
	& \cE \big( (\vr^\varepsilon, \bu^\varepsilon, w^\varepsilon)|(\bar{\vr}, \bv, \eta) \big)(\tau) \\
	& \quad + \frac{\nu}{2} \int_{0}^{\tau} \int_{\Omega^{w^\varepsilon}(t)} \big( \bbs(\nabla \bu^\varepsilon) - \bbs(\nabla \bv) \big) : (\nabla \bu^\varepsilon - \nabla \bv) \dxdt \\
	& \quad + \frac{\alpha \nu}{2} \int_{0}^{\tau} \int_{\Gamma^{w^\varepsilon}(t)} \abs{\big(\bu^\varepsilon - \bv - (\pt w^\varepsilon - \pt \eta) \bn \circ \bPhi_{w^\varepsilon}^{-1}\big)_{\btau^{w^\varepsilon}}}^2 \dH^2 \dt \\
	& \quad + \frac{\alpha_0 \nu}{2} \int_{0}^{\tau} \int_{\Gamma_0} \abs{(\bu^\varepsilon - \bv)_{\btau^0}}^2 \dH^2 \dt \\
	& \quad + \nu_s \int_{0}^{\tau} \int_{\Gamma} \abs{\nabla \pt w^\varepsilon - \nabla \pt \eta}^2 \dH^2 \dt \\
	& \leq \cE \big( (\vr^\varepsilon, \bu^\varepsilon, w^\varepsilon)|(\bar{\vr}, \bv, \eta) \big)(0) 
	+ \frac{\nu}{2} \int_0^\tau \int_{\Omega^{w^\varepsilon}(t)} \abs{\nabla \bv}^2 \dxdt\\
	& \quad
        + \frac{\alpha \nu}{2} \int_{0}^{\tau} \int_{\Gamma^{w^\varepsilon}(t)} \abs{(\bv - \pt \eta \bn \circ \bPhi_{w^\varepsilon}^{-1})_{\btau^{w^\varepsilon}}}^2 \dH^2 \dt \\
	& \quad + \frac{\alpha_0 \nu}{2} \int_{0}^{\tau} \int_{\Gamma_0} \abs{\bv_{\btau^0}}^2 \dH^2 \dt 
        - \int_{0}^{\tau} \int_\Gamma (\pt w^\varepsilon - \pt \eta) (\Pi \circ \bPhi_{w^\varepsilon}) \dH^2 \dt \\
	& \quad - \int_0^\tau \int_{\Omega^{w^\varepsilon}(t)}
	\vr^\varepsilon(\bu^\varepsilon - \bv) \cdot (\bu^\varepsilon - \bv) \cdot \nabla \bv \dxdt
	+ \int_0^\tau \int_{\Omega^{w^\varepsilon}(t)} \frac{\vr^\varepsilon}{\bar{\vr}} (\bu^\varepsilon - \bv) \cdot \nabla \Pi \dxdt \\
	& \quad + \int_0^\tau \int_{\Omega^{w^\varepsilon}(t)} \vr^\varepsilon (\bu^\varepsilon - \bv) \cdot \frac{1}{\bar{\vr}} \bF^\varepsilon \dxdt,
\end{align*}
where we apply Young's inequality for the term with respect to $ \nu $, $\alpha$, $\alpha_0$.
Now we proceed with the estimates term by term. First, direct calculation yields
\begin{equation*}
	\int_0^\tau \int_{\Omega^{w^\varepsilon}(t)}
	\vr^\varepsilon(\bu^\varepsilon - \bv) \cdot (\bu^\varepsilon - \bv) \cdot \nabla \bv \dxdt 
	\leq \int_0^\tau \normm{\nabla \bv(t)}_{L_w^\infty} \cE \big( (\vr^\varepsilon, \bu^\varepsilon, w^\varepsilon)|(\bar{\vr}, \bv, \eta) \big)(t) \dt.
\end{equation*}
Moreover, 
\begin{equation*}
	\int_{0}^{\tau} \int_\Gamma (\pt w^\varepsilon - \pt \eta) (\Pi \circ \bPhi_{w^\varepsilon}) \dH^2 \dt
	= \int_{0}^{\tau} \int_{\Gamma^{w^\varepsilon}(t)} (\bu^\varepsilon - \bv) \cdot \frac{\bn^{w^\varepsilon}}{\sqrt{1 + \abs{\nabla w}^2}} \Pi \dH^2 \dt.
\end{equation*}
On the other hand, by means of weak formulation of the continuity equation \eqref{eqs:weak-renormal}, we have
\begin{align*}
	\int_0^\tau & \int_{\Omega^{w^\varepsilon}(t)} \vr^\varepsilon(\bu^\varepsilon - \bv) \cdot \nabla \Pi \dxdt \\
	& = - \varepsilon \int_0^\tau \int_{\Omega^{w^\varepsilon}(t)} \frac{\vr^\varepsilon - \bar{\vr}}{\varepsilon} \pt \Pi \dxdt
	+ \varepsilon \left[ \int_{\Omega^{w^\varepsilon}(t)} \frac{\vr^\varepsilon - \bar{\vr}}{\varepsilon} \Pi \dx \right]_{t = 0}^{t = \tau} \\
	& \quad
	+ \int_0^\tau \int_{\Gamma^{w^\varepsilon}(t)} \bar{\vr} (\bu^\varepsilon - \bv) \cdot \frac{\bn^{w^\varepsilon}}{\sqrt{1 + \abs{\nabla w}^2}} \Pi \dH^2 \dt
	- \varepsilon \int_0^\tau \int_{\Omega^{w^\varepsilon}(t)} \frac{\vr^\varepsilon - \bar{\vr}}{\varepsilon} \bv \cdot \nabla \Pi \dxdt,
\end{align*}
where we employ an observation by the Reynolds' transport theorem as
\begin{align*}
	\left[ \int_{\Omega^{w^\varepsilon}(t)} \bar{\vr} \Pi \dx \right]_{t = 0}^{t = \tau}
	& = \int_0^\tau \ddt \int_{\Omega^{w^\varepsilon}(t)} \bar{\vr} \Pi \dxdt \\
	& = \int_0^\tau \int_{\Omega^{w^\varepsilon}(t)} \bar{\vr} \pt \Pi \dxdt
	+ \int_0^\tau \int_{\Omega^{w^\varepsilon}(t)} \Div (\bar{\vr} \Pi \bu^\varepsilon) \dxdt \\
	& = \int_0^\tau \int_{\Omega^{w^\varepsilon}(t)} \bar{\vr} \pt \Pi \dxdt 
	+ \int_0^\tau \int_{\Gamma^{w^\varepsilon}(t)} \bar{\vr} \bu^\varepsilon \cdot \frac{\bn^{w^\varepsilon}}{\sqrt{1 + \abs{\nabla w}^2}} \Pi \dH^2 \dt.
\end{align*}
In a similar manner as in \cite{FN2017}, we have
\begin{align*}
	- \varepsilon \int_0^\tau \int_{\Omega^{w^\varepsilon}(t)} \frac{\vr^\varepsilon - \bar{\vr}}{\varepsilon} \pt \Pi \dxdt
	& + \varepsilon \left[ \int_{\Omega^{w^\varepsilon}(t)} \frac{\vr^\varepsilon - \bar{\vr}}{\varepsilon} \Pi \dx \right]_{t = 0}^{t = \tau} \\
	& - \varepsilon \int_0^\tau \int_{\Omega^{w^\varepsilon}(t)} \frac{\vr^\varepsilon - \bar{\vr}}{\varepsilon} \bv \cdot \nabla \Pi \dxdt
	\leq \varepsilon C.
\end{align*}
Therefore,
\begin{align*}
    \int_0^\tau \int_{\Omega^{w^\varepsilon}(t)} \vr^\varepsilon(\bu^\varepsilon - \bv) \cdot \nabla \Pi \dxdt 
    & - \int_{0}^{\tau} \int_\Gamma (\pt w^\varepsilon - \pt \eta) (\Pi \circ \bPhi_{w^\varepsilon}) \dH^2 \dt \\
    & \quad \leq \int_{0}^{\tau} C \cE \big( (\vr^\varepsilon, \bu^\varepsilon, w^\varepsilon)|(\bar{\vr}, \bv, \eta) \big)(t) \dt 
	+ \varepsilon C.
\end{align*}

Furthermore, it follows from \eqref{eqs:uniform-bounds-rho-barrho-epsilon} and \eqref{eqs:uniform-bounds-1+rho-epsilon}  that
\begin{align*}
	& \norm{\ess{\vr^\varepsilon} (\bu^\varepsilon - \bv)}_{L^2(\Omega^{w^\varepsilon})} 
	\leq \ess{\sqrt{\bar{\vr}}} \norm{\sqrt{\vr^\varepsilon} (\bu^\varepsilon - \bv)}_{L^2(\Omega^{w^\varepsilon})} \leq C \cE^\onehalf, \\
	& \norm{\res{\vr^\varepsilon} (\bu^\varepsilon - \bv)}_{L_{w^\varepsilon}^{\frac{2 \gamma}{1 + \gamma}}} 
	\leq \norm{\res{\sqrt{\vr^\varepsilon}}}_{L_{w^\varepsilon}^{2 \gamma}} \norm{\sqrt{\vr^\varepsilon} (\bu^\varepsilon - \bv)}_{L^2(\Omega^{w^\varepsilon})} \leq \varepsilon^{\frac{1}{\gamma}} C \cE^\onehalf.
\end{align*}
If $\gamma > \frac{3}{2}$, then $ \frac{2\gamma}{\gamma+1} > \frac{6}{5} $ and $\frac{2\gamma}{\gamma-1}<6$. 
By virtue of \eqref{eqs:F^epsilon-L^2} and \eqref{eqs:F^epsilon-L^p} we have 

\begin{align}
	\int_0^\tau & \int_{\Omega^{w^\varepsilon}(t)} \vr^\varepsilon(\bu^\varepsilon - \bv) \cdot \bF^\varepsilon \dxdt \nonumber\\
    & = \int_0^\tau \int_{\Omega^{w^\varepsilon}(t)} \ess{\vr^\varepsilon}(\bu^\varepsilon - \bv) \cdot \bF^\varepsilon \dxdt 
    + \int_0^\tau \int_{\Omega^{w^\varepsilon}(t)} \res{\vr^\varepsilon}(\bu^\varepsilon - \bv) \cdot \bF^\varepsilon \dxdt
	\nonumber\\
	& \leq \int_0^\tau C \cE^\onehalf \norm{\bF^\varepsilon(t)}_{L^2(\Omega^{w^\varepsilon})} \dt
    + \varepsilon^\frac{1}{\gamma} C \int_0^\tau \cE^\onehalf  \norm{\bF^\varepsilon(t)}_{L_{w^\varepsilon}^{\frac{2\gamma}{\gamma - 1}}} \dt \nonumber\\
	& \leq 
        \frac{\nu_s}{2} \int_0^\tau \normm{\nabla \pt w^\varepsilon(t) - \nabla \pt \eta(t)}_{L_\Gamma^2}^2 \dt
        + C \int_0^\tau \cE \big( (\vr^\varepsilon, \bu^\varepsilon, w^\varepsilon)|(\bar{\vr}, \bv, \eta) \big)(t) \dt 
        + C \varepsilon^{\frac{1}{\gamma}},
        \label{eqs:rhovF^varepsilon}
\end{align}
where $C > 0$ depends on $\nu_s > 0$ but uniformly in $\varepsilon > 0$.

In conclusion,
\begin{align*}
	\cE & \big( (\vr^\varepsilon, \bu^\varepsilon, w^\varepsilon)|(\bar{\vr}, \bv, \eta) \big)(\tau) \\
	& \leq \cE \big( (\vr^\varepsilon, \bu^\varepsilon, w^\varepsilon)|(\bar{\vr}, \bv, \eta) \big)(0) 
	+ (\nu + \varepsilon) C 
	+ \int_{0}^{\tau} C \cE \big( (\vr^\varepsilon, \bu^\varepsilon, w^\varepsilon)|(\bar{\vr}, \bv, \eta) \big)(t) \dt,
\end{align*}
which, in view of Gronwall's inequality and \eqref{eqs:intial-well}, implies
\begin{align}
	\nonumber
	\cE & \big( (\vr^\varepsilon, \bu^\varepsilon, w^\varepsilon)|(\bar{\vr}, \bv, \eta) \big)(\tau) \\
	\nonumber
	& \leq C(\cE \big( (\vr^\varepsilon, \bu^\varepsilon, w^\varepsilon)|(\bar{\vr}, \bv, \eta) \big)(0) 
	+ C (\nu + \varepsilon)) e^{\tau} \\
	\nonumber
    & \leq C(T',D) \big(\cE \big( (\vr^\varepsilon, \bu^\varepsilon, w^\varepsilon)|(\bar{\vr}, \tilde{\bu}, \eta) \big)(0) 
	+ (\nu + \varepsilon)\big)
    + C(T',D) \norm{\cA^\varepsilon(0) - \bbi}_{L^2}^2 \\
    \label{eqs:limit-piola}
    & \leq C(T',D) \big(\cE \big( (\vr^\varepsilon, \bu^\varepsilon, w^\varepsilon)|(\bar{\vr}, \tilde{\bu}, \eta) \big)(0) 
	+ (\nu + \varepsilon)\big),
\end{align}
as $\normm{\cA^\varepsilon(0) - \bbi}_{L^2}^2 \leq C(D) \normm{\nabla w_0^\varepsilon - \nabla \eta_0}_{L^2}^2$.
Invoking the definition \eqref{eqs:new-transformation} and letting $\tilde{\bu} = \tilde{\bv} \circ \bPsi^\varepsilon$, one further derives
\begin{align*}
	\cE & \big( (\vr^\varepsilon, \bu^\varepsilon, w^\varepsilon)|(\bar{\vr}, \tilde{\bu}, \eta) \big)(\tau) \\
	& \leq C(T',D) \big(\cE \big( (\vr^\varepsilon, \bu^\varepsilon, w^\varepsilon)|(\bar{\vr}, \tilde{\bu}, \eta) \big)(0) 
	+ (\nu + \varepsilon)\big)
    + \int_{\Omega^{w^\varepsilon}(\tau)} \vr^\varepsilon \abs{(J^\varepsilon \cA^\varepsilon - \bbi) \tilde{\bu}}^2(\tau) \dx.
\end{align*}
Note that by \eqref{eqs:uniform-bounds-rho-barrho-epsilon}, \eqref{eqs:uniform-bounds-1+rho-epsilon}, Lemma \ref{lem:bPsi}, \eqref{eqs:tilde-v-psi}, and \eqref{eqs:limit-piola},
\begin{align*}
    \int_{\Omega^{w^\varepsilon}(\tau)} & \vr^\varepsilon \abs{(J^\varepsilon \cA^\varepsilon - \bbi) \tilde{\bu}}^2(\tau) \dx \\
    & \leq \int_{\Omega^{w^\varepsilon}(\tau)} (\vr^\varepsilon - \bar{\vr}) \abs{(J^\varepsilon \cA^\varepsilon - \bbi) \tilde{\bu}}^2(\tau) \dx
    + \int_{\Omega^{w^\varepsilon}(\tau)} \bar{\vr} \abs{(J^\varepsilon \cA^\varepsilon - \bbi) \tilde{\bu}}^2(\tau) \dx \\
    & \leq C \varepsilon \norm{(J^\varepsilon \cA^\varepsilon - \bbi)(\tau)}_{L_{w^\varepsilon}^4 \cap L_{w^\varepsilon}^{\frac{2\gamma}{\gamma-1}}}^2
    + C \norm{(J^\varepsilon \cA^\varepsilon - \bbi)(\tau)}_{L^2(\Omega^{w^\varepsilon})}^2 \\
    & \leq C(\varepsilon + 1) \norm{(\Delta w^\varepsilon - \Delta \eta)(\tau)}_{L_\Gamma^2}^2 \\
    & \leq C(T',D) \big(\cE \big( (\vr^\varepsilon, \bu^\varepsilon, w^\varepsilon)|(\bar{\vr}, \tilde{\bu}, \eta) \big)(0) 
	+ (\nu + \varepsilon)\big).
\end{align*}
Hence, one obtains
\begin{equation*}
	\cE \big( (\vr^\varepsilon, \bu^\varepsilon, w^\varepsilon)|(\bar{\vr}, \tilde{\bu}, \eta) \big)(\tau)
    \leq C(T',D) \big(\cE \big( (\vr^\varepsilon, \bu^\varepsilon, w^\varepsilon)|(\bar{\vr}, \tilde{\bu}, \eta) \big)(0) 
	+ (\nu + \varepsilon)\big),
\end{equation*}
which completes the proof.
\qed

\begin{remark}
    \label{rmk:not-avoid-assumtion}
    In the argument above, we take the \emph{Piola transform}, which preserves the divergence-free property and nullity of normal components. However, it requires a slightly more regularity of $w^\varepsilon$. Indeed, among the error term $\bF^\varepsilon$, we would have {to handle} $\pt \cA^\varepsilon$ and $\nabla \cA^\varepsilon$, which are at most $L^2$-integrable in space with $\nu_s>0$ by standard energy and dissipation (namely, $\pt \nabla w^\varepsilon$ and $\Delta w^\varepsilon$). This would cause much trouble when we derive the limit passage and try to control $\bF^\varepsilon$, cf. \eqref{eqs:rhovF^varepsilon}, and it is essentially where we need the additional regularity assumption \eqref{eqs:w-assumption-thm-3} to close the estimates.
\end{remark}

\subsubsection{Incompressible limit}
If we consider only the low Mach number limit, the statement in Theorem \ref{thm:singular-limit-full} can be accordingly improved, as one may expect some regularity of $\bu^\varepsilon$ with viscosity $\nu > 0$. In this case, the compressible fluid-structure interaction problem will tend to a system coupling the incompressible Navier--Stokes equations and a viscoelastic plate with slip boundary conditions, whose strong well-posedness was established by Djebour--Takahashi \cite{DT2019}. The limit system reads
\begin{subequations}
	\label{eqs:NSPlate}
	\begin{alignat}{3}
		\bar{\vr} (\pt \tilde{\bv} + \tilde{\bv} \cdot \nabla \tilde{\bv}) + \nabla \tilde{\Pi} & = \Div(\nu \bbd(\tilde{\bv})), && \tin Q_{T'}^\eta, \\
		\Div \tilde{\bv} & = 0, && \tin Q_{T'}^\eta, \\ 
		\pt^2 \eta + \Delta^2 \eta - \nu_s \Delta \pt \eta & = \tilde{\Pi} \circ \bPhi_{\eta} - \nu \bbd(\tilde{\bv}) \bn^\eta \circ \bPhi_\eta \cdot \be_3, && \ton \Gamma_{T'}, \\
		(\tilde{\bv} \cdot \bn^\eta) \circ \bPhi_\eta
		& = \pt \eta, && \ton \Gamma_{T'}, \\
		(\nu \bbd(\tilde{\bv}) \bn^\eta)_{\btau^\eta} \circ \bPhi_\eta
		& = - \alpha(\tilde{\bv} - \pt \eta \be_3 \circ \bPhi_\eta^{-1})_{\btau^\eta} \circ \bPhi_\eta, && \ton \Gamma_{T'}, \\
        \tilde{\bv} \cdot \be_3 & = 0, && \ton \Gamma_{T'}^0, \\
		(\nu \bbd(\tilde{\bv}) \be_3 + \alpha_0 \tilde{\bv})_{\btau^0}
		& = 0, && \ton \Gamma_{T'}, \\
		(\tilde{\bv}, \eta, \pt \eta)(\cdot, 0)
		& = (\bv_0 \circ \inv{\bPsi_{0}}, \eta_0, \eta_1), && \tin \Omega^{\eta_0}.
	\end{alignat}
\end{subequations}
where $\bbd(\tilde{\bv}) \coloneqq (\nabla \tilde{\bv} + \nabla \tilde{\bv}^\top)$.
\begin{theorem}
    \label{coro:IncompressibleLimit}
     Let $\Omega \in \bbtwo$, $ \nu_s > 0 $, $ \gamma > 3 $. Assume that the initial data \eqref{eqs:SL-initial-data} is well-prepared satisfying  \eqref{eqs:intial-well} and \eqref{eqs:well-prepared-initial}. Let $ (\vr^\varepsilon, \bu^\varepsilon, w^\varepsilon) $ be a weak solution to the rescaled compressible fluid-structure interaction problem \eqref{eqs:FSI-model-rescale} constructed in Theorem \ref{thm:weak}.
    Let $ (\tilde{\bu}, \tilde{w}) $ be the unique strong solution of \eqref{eqs:NSPlate} defined on $ (0,T') $ subjected to initial data $ (\tilde{\bu}, \tilde{w}, \pt \tilde{w})(0, \cdot) = (\tilde{\bu}_0 \circ \bPhi_{\tilde{w}_0 - w_0^\varepsilon}, \tilde{w}_0, \tilde{w}_1) $.
    Then it holds
    \begin{align}
    	\nonumber
    	& \sup_{0 \leq t \leq T'}
    		\int_{\Omega^{w^\varepsilon}(t)} \bigg(
	    		\vr^\varepsilon \abs{\bu^\varepsilon - \tilde{\bu}}^2 
	    		+ \abs{\frac{\vr^\varepsilon - \bar{\vr}}{\varepsilon}}^\gamma
	    	\bigg)(t) \dx \\
	    \nonumber
    	& \qquad \qquad + \sup_{0 \leq t \leq T'} \int_\Gamma 
    		\bigg(
    			\abs{\pt w^\varepsilon - \pt \tilde{w}}^2
    			+ \abs{\Delta w^\varepsilon - \Delta \tilde{w}}^2
    		\bigg)(t) \dH^2 \\
    	\label{eqs:singular-Limit-convergence-coro}
    	& \qquad \leq C(T',D) \Bigg[\varepsilon
    		+ \int_{\Omega^{w_0^{\varepsilon}}} \bigg(
    			\abs{\bu_0^\varepsilon - \tilde{\bu}_0 \circ \bPhi_{\tilde{w}_0 - w_0^\varepsilon}}^2
    			+ \abs{\vr_0^{\varepsilon,1}}^2
    		\bigg) \dx \\
    	\nonumber
    	& \phantom{\qquad \leq C(T',D) \Bigg[\varepsilon} \, + \int_\Gamma 
    		\bigg(
    			\abs{w_1^\varepsilon - \tilde{w}_1}^2
    			+ \abs{\Delta w_0^\varepsilon - \Delta \tilde{w}_0}^2
    		\bigg) \dH^2 \Bigg],
    \end{align}
	where the constant $ C = C(T',D) $ depends on the norm of the limit solution $ (\tilde{\bu}, \tilde{\eta}) $, $\nu,\nu_s>0$ and on the size of the initial data perturbation. 
\end{theorem}
\begin{proof}
    The proof follows from the same arguments as in Sections \ref{sec:construct-test-function} and \ref{sec:proof-theorem-limit}. In this situation, We point out that we may employ the integrability of $\bu^\varepsilon$ with Sobolev embeddings to decrease the integrability requirement of the source term $\bF^\varepsilon$ in \eqref{eqs:rhovF^varepsilon}. More precisely, by \eqref{eqs:uniform-bounds-rho-u}, \eqref{eqs:uniform-bounds-nabla-u}, and the Korn's inequality in H\"older domain, cf. \cite[Lemma 3.9]{MMNRT2022}, one knows
    \begin{equation}
        \norm{\bu^\varepsilon - \bv}_{W_{{w^\varepsilon}}^{1,p}}^2
        \leq C \Big(
            \norm{\bbs(\nabla \bu^\varepsilon) - \bbs(\nabla \bv)}_{L^2(\Omega^{w^\varepsilon})}^2
            + \norm{\sqrt{\vr^\varepsilon} (\bu^\varepsilon - \bv)}_{L^2(\Omega^{w^\varepsilon})}^2
        \Big),
    \end{equation}
    for any $1< p < 2$, which, together with the Sobolev embedding, cf. \cite[Corollary 2.10]{LR2014}, implies
    \begin{equation}
        \norm{\bu^\varepsilon - \bv}_{L_{w^\varepsilon}^q}^2
        \leq C \Big(
            \norm{\bbs(\nabla \bu^\varepsilon) - \bbs(\nabla \bv)}_{L^2(\Omega^{w^\varepsilon})}^2
            + \norm{\sqrt{\vr^\varepsilon} (\bu^\varepsilon - \bv)}_{L^2(\Omega^{w^\varepsilon})}^2
        \Big),
    \end{equation}
    for any $1 < q < 6$.
    Then it follows from \eqref{eqs:uniform-bounds-rho-epsilon} that
    \begin{align}
        \int_0^\tau & \int_{\Omega^{w^\varepsilon}(t)} \vr^\varepsilon(\bu^\varepsilon - \bv) \cdot \bF^\varepsilon \dxdt \nonumber\\
        & = \int_0^\tau \int_{\Omega^{w^\varepsilon}(t)\setminus\{ \bar{\vr}/2 \leq \vr^\varepsilon \leq 2 \bar{\vr}\}} \vr^\varepsilon (\bu^\varepsilon - \bv) \cdot \bF^\varepsilon \dxdt 
        + \int_0^\tau \int_{\bar{\vr}/2 \leq \vr^\varepsilon \leq 2 \bar{\vr}} \bar{\vr} (\bu^\varepsilon - \bv) \cdot \bF^\varepsilon \dxdt
        \nonumber\\
        & \leq \int_0^\tau C \varepsilon^\frac{2}{\gamma} \norm{(\bu^\varepsilon - \bv)(t)}_{L^q} \norm{\bF^\varepsilon(t)}_{L_{w^\varepsilon}^{\frac{\gamma q}{(q - 1)\gamma - q}}} \dt
        + C \int_0^\tau \norm{(\bu^\varepsilon - \bv)(t)}_{L^q} \norm{\bF^\varepsilon(t)}_{L_{w^\varepsilon}^{\frac{q}{q - 1}}} \dt.
        \nonumber
    \end{align}
    Note that without additional regularity, the domain $\Omega^{w^\varepsilon}(t)$ is only uniform H\"older, and hence with $\nu_s > 0$, we have the control similar to \eqref{eqs:F^epsilon-L^2} (recall Lemma \ref{lem:bPsi} and Corollary \ref{coro:regularity-loss}), 
    \begin{equation}
        \label{eqs:F-varepsilon-Lp}
        \begin{aligned}
            \int_0^\tau \norm{\bF^\varepsilon(t)}_{L_{w^\varepsilon}^{p}}^2 \dt
            & \leq C \int_0^\tau \Big(\normm{\nabla \pt w^\varepsilon(t) - \nabla \pt \eta(t)}_{L_\Gamma^2}^2 \\
            & \qquad \qquad
            + \normm{\pt w^\varepsilon(t) - \pt \eta(t)}_{L_\Gamma^2}^2
            + \normm{\Delta w^\varepsilon(t) - \Delta \eta(t)}_{L_\Gamma^2}^2\Big) \dt,
        \end{aligned}
    \end{equation}
    for any $1 < p < 2$.
    Since $1 < \frac{q}{q-1} < 2$ and $1 < \frac{\gamma q}{(q - 1)\gamma - q} < 2$ for $q > 2$, $\gamma > 3$, with \eqref{eqs:F-varepsilon-Lp}, one derives
    \begin{align}
        \int_0^\tau & \int_{\Omega^{w^\varepsilon}(t)} \vr^\varepsilon(\bu^\varepsilon - \bv) \cdot \bF^\varepsilon \dxdt \nonumber\\
        & \leq 
        \frac{\nu}{4} \int_{0}^{\tau} \int_{\Omega^{w^\varepsilon}(t)} \big( \bbs(\nabla \bu^\varepsilon) - \bbs(\nabla \bv) \big) : (\nabla \bu^\varepsilon - \nabla \bv) \dxdt
        + \frac{\nu_s}{2} \int_0^\tau \normm{\nabla \pt w^\varepsilon(t) - \nabla \pt \eta(t)}_{L_\Gamma^2}^2 \dt 
        \nonumber\\
        & \quad
        + C \int_0^\tau \cE \big( (\vr^\varepsilon, \bu^\varepsilon, w^\varepsilon)|(\bar{\vr}, \bv, \eta) \big)(t) \dt 
        + C \varepsilon^{\frac{2}{\gamma}},
        \label{eqs:rhovF^varepsilon-coro}
    \end{align}
    by which then proceeding with the same argument as in Section \ref{sec:proof-theorem-limit} completes the proof of Theorem \ref{coro:IncompressibleLimit}.
\end{proof}

\section*{Acknowledgments}
    
    The authors are grateful to the anonymous referees for the careful reading of the paper and valuable suggestions, which improved remarkably the presentation of the manuscript.
    
    This work was initiated and partially done during the visits at the Institute of Mathematics of the Czech Academy of Science in Prague and Fakult\"{a}t f\"{u}r Mathematik of Universit\"{a}t Regensburg.
    The hospitalities are gratefully appreciated. 
    
    Y.L. is partially supported by the startup funding from Nanjing Normal University, the Natural Science Foundation of Jiangsu Province (Grant No. BK20240572), the Natural Science Foundation of the Jiangsu Higher Education Institutions of China (Grant No. 24KJB110020), and the China Postdoctoral Science Foundation under Grant Number 2025M773078.
    Y.L. was partially supported by the Graduiertenkolleg 2339 \textit{IntComSin} of the Deutsche Forschungsgemeinschaft (DFG, German Research Foundation) -- Project-ID 321821685, when he was a member at Fakult\"at f\"ur Mathematik of Universit\"at Regensburg and started the project. S.M. is supported by the `Prime Minister Early Career Research Grant, Anusandhan National Research Foundation' (Grant No. ANRF/ECRG/2024/002168/PMS) provided by the Govt. of India.     {\v{S}}.N. has been supported by  (GA\v CR) project 22-01591S and by the  Praemium Academiae of {\v{S}}. Ne{\v{c}}asov{\'{a}}. The Institute of Mathematics, CAS is supported by RVO:67985840. 
    The supports are gratefully acknowledged.

\section*{Compliance with Ethical Standards}
\subsection*{Date avability}
Data sharing not applicable to this article as no datasets were generated during the current study.
\subsection*{Conflict of interest}
The authors declare that there are no conflicts of interest.

\appendix
\section{Proof of Proposition \ref{prop:relative-energy}}
\label{sec:proof-REI}
In this section, we give the proof of Proposition \ref{prop:relative-energy}, the relative energy inequality. 
For a weak solution $ (\vr^\varepsilon, \bu^\varepsilon, w^\varepsilon) $ and a pair of test functions $ (r, \bU, W) $ in $ Q_T^{w^\varepsilon} $, let us recall the relative entropy functional
\begin{equation*}
	\begin{aligned}
		\cE \big( &(\vr^\varepsilon, \bu^\varepsilon, w^\varepsilon)|(r, \bU, W) \big)(t) \\
		& \coloneqq \int_{\Omega^{w^\varepsilon}(t)} 
		\onehalf \vr^\varepsilon \abs{\bu^\varepsilon - \bU}^2 (t) \dx
		+ \frac{1}{\varepsilon^2 (\gamma - 1)} \int_{\Omega^{w^\varepsilon}(t)} \left( p(\vr^\varepsilon) - p'(r) (\vr^\varepsilon - r) - p(r) \right)(t) \dx \\
		& \quad\  + \int_{\Gamma} \left( \onehalf \abs{\pt w^\varepsilon - \pt W}^2 + \onehalf \abs{\Delta w^\varepsilon - \Delta W}^2 \right)(t) \dH^2.
	\end{aligned}
\end{equation*}
In order to prove the proposition, we follow \cite{FJN2012} and adapt it to the case with \textit{an elastic structure} and \textit{Navier-slip boundary condition}. Taking $ (\bphi, \psi) = (\bU, \pt W) $ in \eqref{eqs:weak-momentum} satisfying $ \tr_{\Gamma^w} \bU \cdot (\bn^{w^\varepsilon} \circ \bPhi_{w^\varepsilon}) = (\pt W \bn) \cdot (\bn^{w^\varepsilon} \circ \bPhi_{w^\varepsilon}) $ on $ \Gamma_T $ (where the notion of $\tr_{\Gamma}$ is introduced in Lemma \ref{lem:trace-Phi_w}),we have for a.e. $ \tau \in (0,T) $,
\begin{align}
	\nonumber
	& \int_{\Omega^{w^\varepsilon}(\tau)} (\vr^{\varepsilon} \bu^{\varepsilon}) \cdot \bU(\tau) \dx 
	+ \int_{\Gamma} \pt w^\varepsilon \pt W(\tau) \dH^2 \\
	\nonumber
	& \quad = \int_{\Omega^{w_0}} \mathbf{m}_0 \cdot \bU(0) \dx 
	+ \int_{\Gamma} w_{0,1} \pt W(0) \dH^2 \\
	\nonumber
	& \qquad + \int_0^{\tau} \int_{\Omega^{w^\varepsilon}(t)} \big(
	\vr^\varepsilon \bu^\varepsilon \cdot \pt \bU + (\vr^\varepsilon \bu^\varepsilon \otimes \bu^\varepsilon) : \nabla \bU
	+ \frac{1}{\varepsilon^2} p(\vr^\varepsilon) \Div \bU - \nu \bbs(\nabla \bu^\varepsilon) : \nabla \bU
	\big) \dxdt \\
	\nonumber
	& \qquad + \int_0^{\tau} \int_{\Gamma} \big( \pt w^\varepsilon \pt^2 W - \Delta w^\varepsilon \Delta \pt W - \nu_s \nabla \pt w^\varepsilon \cdot \nabla \pt W \big) \dH^2\dt \\
	\nonumber
	& \qquad - \alpha \nu \int_0^{\tau} \int_{\Gamma^{w^\varepsilon}(t)} \big( \bu^\varepsilon - \pt w^\varepsilon \bn \circ \inv{\bPhi_{w^\varepsilon}} \big)_{\btau^{w^\varepsilon}} \cdot \big( \bU - \pt W \bn \circ \inv{\bPhi_{w^\varepsilon}} \big)_{\btau^{w^\varepsilon}} \dH^2 \dt \\
	\label{eqs:bPhi=bv}
	& \qquad - \alpha_0 \nu \int_0^{\tau} \int_{\Gamma_0(t)}  \bu^\varepsilon_{\btau^0} \cdot \bU_{\btau^0} \dH^2 \dt.
\end{align}
Note that
\begin{align*}
	\int_0^{\tau} & \int_{\Gamma} \pt w^\varepsilon \pt^2 W \dH^2\dt \\
	& = \int_0^{\tau} \int_{\Gamma} \pt W \pt^2 W \dH^2 \dt 
	+ \int_0^{\tau} \int_{\Gamma} (\pt w^\varepsilon - \pt W) \pt^2 W \dH^2 \dt \\
	& = \int_0^{\tau} \int_{\Gamma} \onehalf \frac{\d}{\d t} \abs{\pt W}^2 \dH^2 \dt
	+ \int_0^{\tau} \int_{\Gamma} (\pt w^\varepsilon - \pt W) \pt^2 W \dH^2 \dt \\
	& = \int_{\Gamma} \onehalf \abs{\pt W}^2(\tau) \dH^2 - \int_{\Gamma} \onehalf \abs{\pt W}^2(0) \dH^2
	+ \int_0^{\tau} \int_{\Gamma} (\pt w^\varepsilon - \pt W) \pt^2 W \dH^2 \dt.
\end{align*}
Integration by parts over $\Gamma$ yields
\begin{align*}
	- \int_0^{\tau} & \int_{\Gamma} \Delta w^\varepsilon \Delta \pt W \dH^2 \dt \\
	& = - \int_0^{\tau} \int_{\Gamma} \Big(\frac{\d}{\d t} (\Delta w^\varepsilon \Delta W) - \Delta \pt w^\varepsilon \Delta W\Big) \dH^2 \dt \\
	& = - \int_{\Gamma} \big( (\Delta w^\varepsilon \Delta W)(\tau) - (\Delta w^\varepsilon \Delta W)(0) \big) \dH^2  \\
	& \quad + \int_0^{\tau} \int_{\Gamma} \Delta \pt W \Delta W \dH^2 \dt
	+ \int_0^{\tau} \int_{\Gamma} (\pt w^\varepsilon - \pt W) \Delta^2 W \dH^2 \dt \\
	& = - \int_{\Gamma} \big( (\Delta w^\varepsilon \Delta W)(\tau) - (\Delta w^\varepsilon \Delta W)(0) \big) \dH^2 \\
	& \quad + \int_{\Gamma} \onehalf \abs{\Delta W}^2(\tau) \dH^2 - \int_{\Gamma} \onehalf \abs{\Delta W}^2(0) \dH^2
	+ \int_0^{\tau} \int_{\Gamma} (\pt w^\varepsilon - \pt W) \Delta^2 W \dH^2 \dt,
\end{align*}
and 
\begin{equation*}
	\begin{aligned}
		- \nu_s & \int_0^{\tau} \int_{\Gamma} \nabla \pt w^\varepsilon \cdot \nabla \pt W \dH^2 \dt \\
		& = - \nu_s \int_0^{\tau} \int_{\Gamma} \nabla \pt W \cdot \nabla \pt W \dH^2 \dt
		- \nu_s \int_0^{\tau} \int_{\Gamma} \nabla (\pt w^\varepsilon - \pt W) \cdot \nabla \pt W \dH^2 \dt \\
		& = \nu_s \int_0^{\tau} \int_{\Gamma} \abs{\nabla \pt W}^2 \dH^2 \dt 
		- 2 \nu_s \int_0^{\tau} \int_{\Gamma} \abs{\nabla \pt W}^2 \dH^2 \dt \\
		& \quad + \nu_s \int_0^{\tau} \int_{\Gamma} \nabla (\pt w^\varepsilon - \pt W) \nabla \pt W \dH^2 \dt 
		- 2 \nu_s \int_0^{\tau} \int_{\Gamma} \nabla (\pt w^\varepsilon - \pt W) \nabla \pt W \dH^2 \dt \\
		& = \nu_s \int_0^{\tau} \int_{\Gamma} \abs{\nabla \pt W}^2 \dH^2 \dt 
		- 2 \nu_s \int_0^{\tau} \int_{\Gamma} \nabla \pt w^\varepsilon \cdot \nabla \pt W \dH^2 \dt \\
		& \quad - \nu_s \int_0^{\tau} \int_{\Gamma} (\pt w^\varepsilon - \pt W) \Delta \pt W \dH^2 \dt.
	\end{aligned}
\end{equation*}
Then \eqref{eqs:bPhi=bv} becomes
\begin{align}
	\nonumber
	& \bigg[\int_{\Omega^{w^\varepsilon}(t)} (\vr^\varepsilon \bu^\varepsilon) \cdot \bU(t) \dx 
	+ \int_{\Gamma} \big( \pt w^\varepsilon \pt W + \Delta w^\varepsilon \Delta W \big)(t) \dH^2
	- \int_{\Gamma} \big( \onehalf \abs{\pt W}^2 + \onehalf \abs{\Delta W}^2 \big)(t) \dH^2\bigg]_{t = 0}^{t = \tau} \\
	\nonumber
	& \quad = \int_0^{\tau} \int_{\Omega^{w^\varepsilon}(t)} \big( 
	\vr^\varepsilon \bu^\varepsilon \cdot (\pt + \bu^\varepsilon \cdot \nabla) \bU
	+ \frac{1}{\varepsilon^2} p(\vr^\varepsilon) \Div \bU 
	- \nu \bbs(\nabla \bu^\varepsilon) : \nabla \bU 
	\big) \dxdt \\
	\nonumber
	& \qquad + \nu_s \int_0^{\tau} \int_{\Gamma} \abs{\nabla \pt W}^2 \dH^2 \dt
	- 2 \nu_s \int_0^{\tau} \int_{\Gamma} \nabla \pt w^\varepsilon \cdot \nabla \pt W \dH^2 \dt \\
	\nonumber
	& \qquad + \int_0^{\tau} \int_{\Gamma} (\pt w^\varepsilon - \pt W) (\pt^2 W + \Delta^2 W - \nu_s \Delta \pt W) \dH^2 \dt \\
	\nonumber
	& \qquad - \alpha \nu \int_0^{\tau} \int_{\Gamma^{w^\varepsilon}(t)} \big( \bu^\varepsilon - \pt w^\varepsilon \bn \circ \inv{\bPhi_{w^\varepsilon}} \big)_{\btau^{w^\varepsilon}} \cdot \big( \bU - \pt W \bn \circ \inv{\bPhi_{w^\varepsilon}} \big)_{\btau^{w^\varepsilon}} \dH^2 \dt \\
	\label{eqs:bPhi=bv_2}
	& \qquad - \alpha_0 \nu \int_0^{\tau} \int_{\Gamma_0(t)}  \bu^\varepsilon_{\btau^0} \cdot \bU_{\btau^0} \dH^2 \dt.
\end{align}

Let $ b(\vr^\varepsilon) = \vr^\varepsilon $. Setting $ \varphi = \onehalf \abs{\bU}^2 $ and $ \varphi = \frac{1}{\gamma - 1} p'(r) $ in \eqref{eqs:weak-renormal} respectively, one obtains
\begin{equation}
	\label{eqs:phi=U^2}
	\bigg[
	\onehalf \int_{\Omega^{w^\varepsilon}(t)} \vr^\varepsilon \abs{\bU}^2 (t) \dx
	\bigg]_{t = 0}^{t = \tau}
	= \int_0^{\tau} \int_{\Omega^{w^\varepsilon}(t)} \vr^\varepsilon \bU \cdot ( \pt + \bu \cdot \nabla ) \bU \dxdt,
\end{equation}
and
\begin{equation}
	\label{eqs:phi=H'(rho)}
	\bigg[
	\frac{1}{\gamma - 1} \int_{\Omega^{w^\varepsilon}(t)} \vr^\varepsilon p'(r) (t) \dx
	\bigg]_{t = 0}^{t = \tau}
	= \frac{1}{\gamma - 1} \int_0^{\tau} \int_{\Omega^{w^\varepsilon}(t)} \vr^\varepsilon ( \pt + \bu^\varepsilon \cdot \nabla ) p'(r) \dxdt.
\end{equation}
By the Reynolds transport theorem and Gauss's theorem, one gets
\begin{align*}
	\int_{\Omega^{w^\varepsilon}(\tau)} \bU \cdot \nabla p(r) \dx
	& = \int_{\Omega^{w^\varepsilon}(\tau)} \big( \Div (p(r) \bU) - p(r) \Div \bU \big) \dx \\
	& = \int_{\Omega^{w^\varepsilon}(\tau)} \Div ( p(r) \bu^\varepsilon ) \dx \\
	& \quad - \int_{\Omega^{w^\varepsilon}(\tau)} \Div \big( p(r)(\bu^\varepsilon - \bU) \big) \dx 
	- \int_{\Omega^{w^\varepsilon}(\tau)} p(r) \Div \bU \dx \\
	& = \frac{\d}{\d \tau} \int_{\Omega^{w^\varepsilon}(\tau)} p(r) \dx 
	- \int_{\Omega^{w^\varepsilon}(\tau)} \pt p(r) \dx\\
	& \quad - \int_{\Gamma^{w}(\tau)} p(r) (\bu^\varepsilon - \bU) \cdot \bn^{w^\varepsilon} \dH^2
	- \int_{\Omega^{w^\varepsilon}(\tau)} p(r) \Div \bU \dx,
\end{align*}
which implies that
\begin{equation}
	\label{eqs:r^gamma}
	\begin{aligned}
		\int_{\Omega^{w^\varepsilon}(\tau)} p(r)(\tau) \dx
		& = \int_{\Omega^{w_0}} p(r)(0) \dx
		+ \int_0^{\tau} \int_{\Gamma^{w^\varepsilon}(t)} p(r) (\bu^\varepsilon - \bU) \cdot \bn^{w^\varepsilon} \dH^2 \dt \\
		& \quad + \int_0^{\tau} \int_{\Omega^{w^\varepsilon}(t)} p(r) \Div \bU \dxdt + \frac{1}{\gamma - 1} \int_0^{\tau} \int_{\Omega^{w^\varepsilon}(t)} r (\pt + \bU \cdot \nabla) p'(r) \dxdt,
	\end{aligned}
\end{equation}
where we have used the identity
\begin{equation}
	\label{eqs:identity-rho^gamma}
	\partial p(\vr) = \gamma \vr^{\gamma - 1} \partial \vr = \gamma \vr (\vr^{\gamma - 2} \partial \vr) = \frac{1}{\gamma - 1} \vr \partial p'(\vr).
\end{equation}
Here $ \partial \cdot $ is a general first-order differential operator. Note that 
\begin{equation*}
	\frac{1}{\gamma - 1} \int_{\Omega^{w^\varepsilon}(\tau)} \vr^\varepsilon p'(r) \dx
	- \int_{\Omega^{w^\varepsilon}(\tau)} p(r) \dx
	= \frac{1}{\gamma - 1} \int_{\Omega^{w^\varepsilon}(\tau)} (\vr^\varepsilon - r) p'(r) \dx
	+ \frac{1}{\gamma - 1} \int_{\Omega^{w^\varepsilon}(\tau)} p(r) \dx.
\end{equation*}
Now subtracting \eqref{eqs:bPhi=bv_2}, \eqref{eqs:phi=H'(rho)} from \eqref{eqs:weak-energy-inequality} (with parameters $ \varepsilon $ and $ \nu $ correspondingly) and adding \eqref{eqs:phi=U^2}, \eqref{eqs:r^gamma} multiplied by $ \frac{1}{\varepsilon^2} $ together yields the desired relative energy inequality \eqref{eqs:relative-energy-inequality-epsilon}.
\qed

\begin{remark}
    Here the relative energy inequality is derived in the deformed domain $\Omega^{w^\varepsilon}(t)$ depending on time. In \cite{Trifunovic2023}, a version of relative energy inequality in a fixed configuration was considered, which contains the associated quantities due to the transformation, e.g., the deformation gradient.
\end{remark}

\bibliographystyle{siam}
\bibliography{CFSI}

@article {AL2023a,
	AUTHOR = {Abels, Helmut and Liu, Yadong},
	TITLE = {On a fluid-structure interaction problem for plaque growth},
	JOURNAL = {Nonlinearity},
	FJOURNAL = {Nonlinearity},
	VOLUME = {36},
	YEAR = {2023},
	NUMBER = {1},
	PAGES = {537--583},
	ISSN = {0951-7715,1361-6544},
	MRCLASS = {35R35 (35Q30 74F10 74L15)},
	MRNUMBER = {4521953},
}

@article {BH2021,
	AUTHOR = {Behzadan, A. and Holst, M.},
	TITLE = {Multiplication in {S}obolev spaces, revisited},
	JOURNAL = {Ark. Mat.},
	FJOURNAL = {Arkiv f\"{o}r Matematik},
	VOLUME = {59},
	YEAR = {2021},
	NUMBER = {2},
	PAGES = {275--306},
	ISSN = {0004-2080,1871-2487},
	MRCLASS = {46E35 (35J65)},
	MRNUMBER = {4339668},
	MRREVIEWER = {Antonella\ Nastasi},
	DOI = {10.4310/arkiv.2021.v59.n2.a2},
	URL = {https://doi.org/10.4310/arkiv.2021.v59.n2.a2},
}

@article {BdV2004,
	AUTHOR = {Beir\~{a}o da Veiga, H.},
	TITLE = {On the existence of strong solutions to a coupled
	fluid-structure evolution problem},
	JOURNAL = {J. Math. Fluid Mech.},
	FJOURNAL = {Journal of Mathematical Fluid Mechanics},
	VOLUME = {6},
	YEAR = {2004},
	NUMBER = {1},
	PAGES = {21--52},
	ISSN = {1422-6928,1422-6952},
	MRCLASS = {35Q30 (74F10 76D03 76D05 76F10 76Z05)},
	MRNUMBER = {2027753},
	MRREVIEWER = {Shu\ Ming\ Sun},
	DOI = {10.1007/s00021-003-0082-5},
	URL = {https://doi.org/10.1007/s00021-003-0082-5},
}

@article {BKS2023,
    AUTHOR = {Bene\v{s}ov\'a, Barbora and Kampschulte, Malte and
              Schwarzacher, Sebastian},
     TITLE = {A variational approach to hyperbolic evolutions and
              fluid-structure interactions},
   JOURNAL = {J. Eur. Math. Soc. (JEMS)},
  FJOURNAL = {Journal of the European Mathematical Society (JEMS)},
    VOLUME = {26},
      YEAR = {2024},
    NUMBER = {12},
     PAGES = {4615--4697},
      ISSN = {1435-9855,1435-9863},
   MRCLASS = {35Q35 (35Q30 35Q74 49J45 49S05 74B20 74H20 76D03)},
  MRNUMBER = {4780491},
       DOI = {10.4171/jems/1353},
       URL = {https://doi.org/10.4171/jems/1353},
}

@article {Breit2023,
	AUTHOR = {Breit, Dominic},
	TITLE = {Regularity results in 2{D} fluid-structure interaction},
	JOURNAL = {Math. Ann.},
	FJOURNAL = {Mathematische Annalen},
	VOLUME = {388},
	YEAR = {2024},
	NUMBER = {2},
	PAGES = {1495--1538},
	ISSN = {0025-5831,1432-1807},
	MRCLASS = {35B65 (35Q30 74F10 74K25 76D03)},
	MRNUMBER = {4700375},
	DOI = {10.1007/s00208-022-02548-9},
	URL = {https://doi.org/10.1007/s00208-022-02548-9},
}

@article {BMSS2023,
	AUTHOR = {Breit, Dominic and Mensah, Prince Romeo and Schwarzacher, Sebastian and Su, Pei},
	TITLE = {{L}adyzhenskaya-{P}rodi-{S}errin condition for fluid-structure interaction systems},
	JOURNAL = {ANNALI SCUOLA NORMALE SUPERIORE - CLASSE DI SCIENZE},
	FJOURNAL = {ANNALI SCUOLA NORMALE SUPERIORE - CLASSE DI SCIENZE},
	VOLUME = {43},
	YEAR = {2025},
	NUMBER = {},
	PAGES = {1--43},
	ISSN = {},
	MRCLASS = {},
	MRNUMBER = {},
	MRREVIEWER = {},
	NOTE = { },
	DOI = {10.2422/2036-2145.202407_010},
	URL = {},
}

@article {BS2018,
	AUTHOR = {Breit, Dominic and Schwarzacher, Sebastian},
	TITLE = {Compressible fluids interacting with a linear-elastic shell},
	JOURNAL = {Arch. Ration. Mech. Anal.},
	FJOURNAL = {Archive for Rational Mechanics and Analysis},
	VOLUME = {228},
	YEAR = {2018},
	NUMBER = {2},
	PAGES = {495--562},
	ISSN = {0003-9527},
	MRCLASS = {76N10 (74B05 74F10 74K25)},
	MRNUMBER = {3766983},
	MRREVIEWER = {Alessandro Morando},
	DOI = {10.1007/s00205-017-1199-8},
	URL = {https://doi.org/10.1007/s00205-017-1199-8},
}

@article {BS2021,
	AUTHOR = {Breit, Dominic and Schwarzacher, Sebastian},
	TITLE = {Navier-{S}tokes-{F}ourier fluids interacting with elastic
	shells},
	JOURNAL = {Ann. Sc. Norm. Super. Pisa Cl. Sci. (5)},
	FJOURNAL = {Annali della Scuola Normale Superiore di Pisa. Classe di
	Scienze. Serie V},
	VOLUME = {24},
	YEAR = {2023},
	NUMBER = {2},
	PAGES = {619--690},
	ISSN = {0391-173X,2036-2145},
	MRCLASS = {76N10 (35Q30 74K25 76N06)},
	MRNUMBER = {4630041},
	MRREVIEWER = {Bernard\ Ducomet},
	DOI = {10.2422/2036-2145.202105\_090},
	URL = {https://doi.org/10.2422/2036-2145.202105_090},
}

@book {Ciarlet1988,
	AUTHOR = {Ciarlet, Philippe G.},
	TITLE = {Mathematical elasticity. {V}ol. {I}},
	SERIES = {Studies in Mathematics and its Applications},
	VOLUME = {20},
	NOTE = {Three-dimensional elasticity},
	PUBLISHER = {North-Holland Publishing Co., Amsterdam},
	YEAR = {1988},
	PAGES = {xlii+451},
	ISBN = {0-444-70259-8},
	MRCLASS = {73-02 (73Cxx)},
	MRNUMBER = {936420},
	MRREVIEWER = {Cornelius O. Horgan},
}

@article {CDEG2004,
	AUTHOR = {Chambolle, Antonin and Desjardins, Beno\^{i}t and Esteban,
	Maria J. and Grandmont, C\'{e}line},
	TITLE = {Existence of weak solutions for the unsteady interaction of a
	viscous fluid with an elastic plate},
	JOURNAL = {J. Math. Fluid Mech.},
	FJOURNAL = {Journal of Mathematical Fluid Mechanics},
	VOLUME = {7},
	YEAR = {2005},
	NUMBER = {3},
	PAGES = {368--404},
	ISSN = {1422-6928,1422-6952},
	MRCLASS = {35Q35 (35D05 35Q30 74F10 76D03)},
	MRNUMBER = {2166981},
	MRREVIEWER = {Changxing\ Miao},
	DOI = {10.1007/s00021-004-0121-y},
	URL = {https://doi.org/10.1007/s00021-004-0121-y},
}

@article {CCS2007,
	AUTHOR = {Cheng, C. H. Arthur and Coutand, Daniel and Shkoller, Steve},
	TITLE = {Navier-{S}tokes equations interacting with a nonlinear elastic
	biofluid shell},
	JOURNAL = {SIAM J. Math. Anal.},
	FJOURNAL = {SIAM Journal on Mathematical Analysis},
	VOLUME = {39},
	YEAR = {2007},
	NUMBER = {3},
	PAGES = {742--800},
	ISSN = {0036-1410},
	MRCLASS = {74F10 (35Q30 74B20 74H25 76D05 76Z99)},
	MRNUMBER = {2349865},
	MRREVIEWER = {Ana L. Silvestre},
	DOI = {10.1137/060656085},
	URL = {https://doi.org/10.1137/060656085},
}

@article {CS2010,
	AUTHOR = {Cheng, C. H. Arthur and Shkoller, Steve},
	TITLE = {The interaction of the 3{D} {N}avier--{S}tokes equations with a
	moving nonlinear {K}oiter elastic shell},
	JOURNAL = {SIAM J. Math. Anal.},
	FJOURNAL = {SIAM Journal on Mathematical Analysis},
	VOLUME = {42},
	YEAR = {2010},
	NUMBER = {3},
	PAGES = {1094--1155},
	ISSN = {0036-1410},
	MRCLASS = {74F10 (35Q30 35Q74 74H20 74H25 74K25 76D05)},
	MRNUMBER = {2644917},
	MRREVIEWER = {Natalia B. Chinchaladze},
	DOI = {10.1137/080741628},
	URL = {https://doi.org/10.1137/080741628},
}

@article {DT2019,
	AUTHOR = {Djebour, Imene Aicha and Takahashi, Tak\'{e}o},
	TITLE = {On the existence of strong solutions to a fluid structure
	interaction problem with {N}avier boundary conditions},
	JOURNAL = {J. Math. Fluid Mech.},
	FJOURNAL = {Journal of Mathematical Fluid Mechanics},
	VOLUME = {21},
	YEAR = {2019},
	NUMBER = {3},
	PAGES = {Paper No. 36, 30},
	ISSN = {1422-6928,1422-6952},
	MRCLASS = {35Q30 (35D35 74F10 76D03 76D05)},
	MRNUMBER = {3962841},
	DOI = {10.1007/s00021-019-0440-7},
	URL = {https://doi.org/10.1007/s00021-019-0440-7},
}

@article {DPV2012,
    AUTHOR = {Di Nezza, Eleonora and Palatucci, Giampiero and Valdinoci,
              Enrico},
     TITLE = {Hitchhiker's guide to the fractional {S}obolev spaces},
   JOURNAL = {Bull. Sci. Math.},
  FJOURNAL = {Bulletin des Sciences Math\'{e}matiques},
    VOLUME = {136},
      YEAR = {2012},
    NUMBER = {5},
     PAGES = {521--573},
      ISSN = {0007-4497,1952-4773},
   MRCLASS = {46E35 (35A23 35S05 35S30)},
  MRNUMBER = {2944369},
MRREVIEWER = {Lanzhe\ Liu},
       DOI = {10.1016/j.bulsci.2011.12.004},
       URL = {https://doi.org/10.1016/j.bulsci.2011.12.004},
}

@article {EH2013,
	AUTHOR = {Evans, John A. and Hughes, Thomas J. R.},
	TITLE = {Isogeometric divergence-conforming {B}-splines for the
	{D}arcy-{S}tokes-{B}rinkman equations},
	JOURNAL = {Math. Models Methods Appl. Sci.},
	FJOURNAL = {Mathematical Models and Methods in Applied Sciences},
	VOLUME = {23},
	YEAR = {2013},
	NUMBER = {4},
	PAGES = {671--741},
	ISSN = {0218-2025,1793-6314},
	MRCLASS = {76S05 (65N12 65N15 65N30 76D07 76M10)},
	MRNUMBER = {3021778},
	MRREVIEWER = {Gangavamsam\ P.\ Raja Sekhar},
	DOI = {10.1142/S0218202512500583},
	URL = {https://doi.org/10.1142/S0218202512500583},
}

@book {Feireisl2004,
	AUTHOR = {Feireisl, Eduard},
	TITLE = {Dynamics of viscous compressible fluids},
	SERIES = {Oxford Lecture Series in Mathematics and its Applications},
	VOLUME = {26},
	PUBLISHER = {Oxford University Press, Oxford},
	YEAR = {2004},
	PAGES = {xii+212},
	ISBN = {0-19-852838-8},
	MRCLASS = {76N10 (35Q35 76N15)},
	MRNUMBER = {2040667},
	MRREVIEWER = {Piotr\ Bogus\l aw\ Mucha},
}

@article {FJN2012,
	AUTHOR = {Feireisl, Eduard and Jin, Bum Ja and Novotn\'{y},
	Anton\'{\i}n},
	TITLE = {Relative entropies, suitable weak solutions, and weak-strong
	uniqueness for the compressible {N}avier-{S}tokes system},
	JOURNAL = {J. Math. Fluid Mech.},
	FJOURNAL = {Journal of Mathematical Fluid Mechanics},
	VOLUME = {14},
	YEAR = {2012},
	NUMBER = {4},
	PAGES = {717--730},
	ISSN = {1422-6928,1422-6952},
	MRCLASS = {35Q35 (35A02 76N10)},
	MRNUMBER = {2992037},
	MRREVIEWER = {Fa-gui\ Liu},
	DOI = {10.1007/s00021-011-0091-9},
	URL = {https://doi.org/10.1007/s00021-011-0091-9},
}

@article {FKNZ2016,
	AUTHOR = {Feireisl, Eduard and Klein, Rupert and Novotn\'{y},
	Anton\'{\i}n and Zatorska, Ewelina},
	TITLE = {On singular limits arising in the scale analysis of stratified
	fluid flows},
	JOURNAL = {Math. Models Methods Appl. Sci.},
	FJOURNAL = {Mathematical Models and Methods in Applied Sciences},
	VOLUME = {26},
	YEAR = {2016},
	NUMBER = {3},
	PAGES = {419--443},
	ISSN = {0218-2025,1793-6314},
	MRCLASS = {35Q30 (35B25)},
	MRNUMBER = {3458246},
	DOI = {10.1142/S021820251650007X},
	URL = {https://doi.org/10.1142/S021820251650007X},
}

@article {FKNNS2013,
	AUTHOR = {Feireisl, Eduard and Kreml, Ond\v{r}ej and Ne\v{c}asov\'{a},
	{\v{S}}\'{a}rka and Neustupa, Ji\v{r}\'{\i} and Stebel, Jan},
	TITLE = {Weak solutions to the barotropic {N}avier-{S}tokes system with
	slip boundary conditions in time dependent domains},
	JOURNAL = {J. Differential Equations},
	FJOURNAL = {Journal of Differential Equations},
	VOLUME = {254},
	YEAR = {2013},
	NUMBER = {1},
	PAGES = {125--140},
	ISSN = {0022-0396,1090-2732},
	MRCLASS = {35Q35 (35D30)},
	MRNUMBER = {2983046},
	MRREVIEWER = {Yueling\ Jia},
	DOI = {10.1016/j.jde.2012.08.019},
	URL = {https://doi.org/10.1016/j.jde.2012.08.019},
}

@book {FN2017,
	AUTHOR = {Feireisl, Eduard and Novotn\'{y}, Anton\'{\i}n},
	TITLE = {Singular limits in thermodynamics of viscous fluids},
	SERIES = {Advances in Mathematical Fluid Mechanics},
	EDITION = {Second},
	PUBLISHER = {Birkh\"{a}user/Springer, Cham},
	YEAR = {2017},
	PAGES = {xlii+524},
	ISBN = {978-3-319-63780-8; 978-3-319-63781-5},
	MRCLASS = {35-02 (35B25 35Q35 76-02 76N10)},
	MRNUMBER = {3729430},
}

@article {FNP2001,
	AUTHOR = {Feireisl, Eduard and Novotn\'{y}, Anton\'{\i}n and
	Petzeltov\'{a}, Hana},
	TITLE = {On the existence of globally defined weak solutions to the
	{N}avier-{S}tokes equations},
	JOURNAL = {J. Math. Fluid Mech.},
	FJOURNAL = {Journal of Mathematical Fluid Mechanics},
	VOLUME = {3},
	YEAR = {2001},
	NUMBER = {4},
	PAGES = {358--392},
	ISSN = {1422-6928,1422-6952},
	MRCLASS = {35Q30 (76N10)},
	MRNUMBER = {1867887},
	MRREVIEWER = {Cun-Zheng\ Wang},
	DOI = {10.1007/PL00000976},
	URL = {https://doi.org/10.1007/PL00000976},
}

@article {Grandmont2008,
	AUTHOR = {Grandmont, C\'{e}line},
	TITLE = {Existence of weak solutions for the unsteady interaction of a
	viscous fluid with an elastic plate},
	JOURNAL = {SIAM J. Math. Anal.},
	FJOURNAL = {SIAM Journal on Mathematical Analysis},
	VOLUME = {40},
	YEAR = {2008},
	NUMBER = {2},
	PAGES = {716--737},
	ISSN = {0036-1410},
	MRCLASS = {35Q35 (35D05 74F10 76D03)},
	MRNUMBER = {2438783},
	MRREVIEWER = {Dmitry A. Vorotnikov},
	DOI = {10.1137/070699196},
	URL = {https://doi.org/10.1137/070699196},
}

@article {GH2016,
	AUTHOR = {Grandmont, C\'{e}line and Hillairet, Matthieu},
	TITLE = {Existence of global strong solutions to a beam-fluid
	interaction system},
	JOURNAL = {Arch. Ration. Mech. Anal.},
	FJOURNAL = {Archive for Rational Mechanics and Analysis},
	VOLUME = {220},
	YEAR = {2016},
	NUMBER = {3},
	PAGES = {1283--1333},
	ISSN = {0003-9527},
	MRCLASS = {74F10 (35A01 35D35 35Q30 35Q74 74G25)},
	MRNUMBER = {3466847},
	MRREVIEWER = {Merab Svanadze},
	DOI = {10.1007/s00205-015-0954-y},
	URL = {https://doi.org/10.1007/s00205-015-0954-y},
}

@article {GHL2019,
	AUTHOR = {Grandmont, C\'{e}line and Hillairet, Matthieu and Lequeurre,
	Julien},
	TITLE = {Existence of local strong solutions to fluid-beam and
	fluid-rod interaction systems},
	JOURNAL = {Ann. Inst. H. Poincar\'{e} C Anal. Non Lin\'{e}aire},
	FJOURNAL = {Annales de l'Institut Henri Poincar\'{e} C. Analyse Non Lin\'{e}aire},
	VOLUME = {36},
	YEAR = {2019},
	NUMBER = {4},
	PAGES = {1105--1149},
	ISSN = {0294-1449},
	MRCLASS = {76D05 (35D35 35Q35 74F10)},
	MRNUMBER = {3955112},
	DOI = {10.1016/j.anihpc.2018.10.006},
	URL = {https://doi.org/10.1016/j.anihpc.2018.10.006},
}

@article {GGP2012,
	AUTHOR = {Guidoboni, Giovanna and Guidorzi, Marcello and Padula,
	Mariarosaria},
	TITLE = {Continuous dependence on initial data in fluid-structure
	motions},
	JOURNAL = {J. Math. Fluid Mech.},
	FJOURNAL = {Journal of Mathematical Fluid Mechanics},
	VOLUME = {14},
	YEAR = {2012},
	NUMBER = {1},
	PAGES = {1--32},
	ISSN = {1422-6928,1422-6952},
	MRCLASS = {74F10 (76D03 76D05)},
	MRNUMBER = {2891187},
	MRREVIEWER = {Natalia\ B.\ Chinchaladze},
	DOI = {10.1007/s00021-010-0031-0},
	URL = {https://doi.org/10.1007/s00021-010-0031-0},
}

@article {Hoff1998,
	AUTHOR = {Hoff, David},
	TITLE = {The zero-{M}ach limit of compressible flows},
	JOURNAL = {Comm. Math. Phys.},
	FJOURNAL = {Communications in Mathematical Physics},
	VOLUME = {192},
	YEAR = {1998},
	NUMBER = {3},
	PAGES = {543--554},
	ISSN = {0010-3616,1432-0916},
	MRCLASS = {35Q30 (76N10)},
	MRNUMBER = {1620511},
	MRREVIEWER = {Beno\^{i}t\ P.\ Desjardins},
	DOI = {10.1007/s002200050308},
	URL = {https://doi.org/10.1007/s002200050308},
}

@article {KMN2023,
	AUTHOR = {Kalousek, Martin and Mitra, Sourav and Ne\v{c}asov\'{a},
	{\v{S}}\'{a}rka},
	TITLE = {The existence of a weak solution for a compressible
	multicomponent fluid structure interaction problem},
	JOURNAL = {J. Math. Pures Appl. (9)},
	FJOURNAL = {Journal de Math\'{e}matiques Pures et Appliqu\'{e}es.
	Neuvi\`eme S\'{e}rie},
	VOLUME = {184},
	YEAR = {2024},
	PAGES = {118--189},
	ISSN = {0021-7824,1776-3371},
	MRCLASS = {76T06 (35D30 35Q30 74F10)},
	MRNUMBER = {4717264},
	DOI = {10.1016/j.matpur.2024.02.007},
	URL = {https://doi.org/10.1016/j.matpur.2024.02.007},
}

@book {KR2019,
	AUTHOR = {Kru\v{z}\'{\i}k, Martin and Roub\'{\i}\v{c}ek, Tom\'{a}\v{s}},
	TITLE = {Mathematical methods in continuum mechanics of solids},
	SERIES = {Interaction of Mechanics and Mathematics},
	PUBLISHER = {Springer, Cham},
	YEAR = {2019},
	PAGES = {xiii+617},
	ISBN = {978-3-030-02064-4; 978-3-030-02065-1},
	MRCLASS = {74-01 (35Q74 74Bxx)},
	MRNUMBER = {3890020},
	MRREVIEWER = {Corina-\c{S}tefania Drapaca},
	DOI = {10.1007/978-3-030-02065-1},
	URL = {https://doi.org/10.1007/978-3-030-02065-1},
}

@article {KT2022,
    AUTHOR = {Kukavica, Igor and Tuffaha, Amjad},
     TITLE = {A free boundary inviscid model of flow-structure interaction},
   JOURNAL = {J. Differential Equations},
  FJOURNAL = {Journal of Differential Equations},
    VOLUME = {413},
      YEAR = {2024},
     PAGES = {851--912},
      ISSN = {0022-0396,1090-2732},
   MRCLASS = {35Q31 (35B65 35L57 35R35 76B07)},
  MRNUMBER = {4795678},
       DOI = {10.1016/j.jde.2024.08.045},
       URL = {https://doi.org/10.1016/j.jde.2024.08.045},
}

@article {Kukucka2009,
	AUTHOR = {Kuku\v{c}ka, Peter},
	TITLE = {On the existence of finite energy weak solutions to the
	{N}avier-{S}tokes equations in irregular domains},
	JOURNAL = {Math. Methods Appl. Sci.},
	FJOURNAL = {Mathematical Methods in the Applied Sciences},
	VOLUME = {32},
	YEAR = {2009},
	NUMBER = {11},
	PAGES = {1428--1451},
	ISSN = {0170-4214,1099-1476},
	MRCLASS = {35Q35 (35A35 35D30 76D05)},
	MRNUMBER = {2535699},
	MRREVIEWER = {Ricardo\ Ruiz Baier},
	DOI = {10.1002/mma.1101},
	URL = {https://doi.org/10.1002/mma.1101},
}

@book {Lee2013,
	AUTHOR = {Lee, John M.},
	TITLE = {Introduction to smooth manifolds},
	SERIES = {Graduate Texts in Mathematics},
	VOLUME = {218},
	EDITION = {Second},
	PUBLISHER = {Springer, New York},
	YEAR = {2013},
	PAGES = {xvi+708},
	ISBN = {978-1-4419-9981-8},
	MRCLASS = {58-01 (53-01 57-01)},
	MRNUMBER = {2954043},
}

@book {Lions1998,
	AUTHOR = {Lions, Pierre-Louis},
	TITLE = {Mathematical topics in fluid mechanics. {V}ol. 2},
	SERIES = {Oxford Lecture Series in Mathematics and its Applications},
	VOLUME = {10},
	NOTE = {Compressible models,
	Oxford Science Publications},
	PUBLISHER = {The Clarendon Press, Oxford University Press, New York},
	YEAR = {1998},
	PAGES = {xiv+348},
	ISBN = {0-19-851488-3},
	MRCLASS = {76-02 (35-02 35Q30 76N10)},
	MRNUMBER = {1637634},
	MRREVIEWER = {Denis\ Serre},
}

@article {LMN-WSU,
	AUTHOR = {Liu, Yadong and Mitra, Sourav and Ne\v{c}asov\'{a}, {\v{S}}\'{a}rka},
	TITLE = {Weak-strong uniqueness of a class of fluid-structure interaction problem},
	JOURNAL = {},
	FJOURNAL = {},
	VOLUME = {},
	YEAR = {2025},
	NUMBER = {},
	PAGES = {},
	ISSN = {},
	MRCLASS = {},
	MRNUMBER = {},
	MRREVIEWER = {},
	NOTE = {in preparation},
	DOI = {},
	URL = {},
}

@article {LR2014,
	AUTHOR = {Lengeler, Daniel and R\r{u}\v{z}i\v{c}ka, Michael},
	TITLE = {Weak solutions for an incompressible {N}ewtonian fluid
	interacting with a {K}oiter type shell},
	JOURNAL = {Arch. Ration. Mech. Anal.},
	FJOURNAL = {Archive for Rational Mechanics and Analysis},
	VOLUME = {211},
	YEAR = {2014},
	NUMBER = {1},
	PAGES = {205--255},
	ISSN = {0003-9527},
	MRCLASS = {35Q35 (35D30 35Q74 74F10 74K25 76D05)},
	MRNUMBER = {3147436},
	MRREVIEWER = {Zhaoyin Xiang},
	DOI = {10.1007/s00205-013-0686-9},
	URL = {https://doi.org/10.1007/s00205-013-0686-9},
}

@article {Lequeurre2011,
	AUTHOR = {Lequeurre, Julien},
	TITLE = {Existence of strong solutions to a fluid-structure system},
	JOURNAL = {SIAM J. Math. Anal.},
	FJOURNAL = {SIAM Journal on Mathematical Analysis},
	VOLUME = {43},
	YEAR = {2011},
	NUMBER = {1},
	PAGES = {389--410},
	ISSN = {0036-1410},
	MRCLASS = {35Q30 (74F10 74L15 76D03 76D05 92C35)},
	MRNUMBER = {2765696},
	MRREVIEWER = {Natalia B. Chinchaladze},
	DOI = {10.1137/10078983X},
	URL = {https://doi.org/10.1137/10078983X},
}

@article {Lequeurre2013,
	AUTHOR = {Lequeurre, Julien},
	TITLE = {Existence of strong solutions for a system coupling the
	{N}avier--{S}tokes equations and a damped wave equation},
	JOURNAL = {J. Math. Fluid Mech.},
	FJOURNAL = {Journal of Mathematical Fluid Mechanics},
	VOLUME = {15},
	YEAR = {2013},
	NUMBER = {2},
	PAGES = {249--271},
	ISSN = {1422-6928},
	MRCLASS = {35Q30 (35A01 35D35 74F10 76D05)},
	MRNUMBER = {3061763},
	DOI = {10.1007/s00021-012-0107-0},
	URL = {https://doi.org/10.1007/s00021-012-0107-0},
}

@article {MMNRT2022,
	AUTHOR = {M\'{a}cha, V\'{a}clav and Muha, Boris and Ne\v{c}asov\'{a},
	{\v{S}}\'{a}rka and Roy, Arnab and Trifunovi\'{c}, Sr\dj an},
	TITLE = {Existence of a weak solution to a nonlinear fluid-structure
	interaction problem with heat exchange},
	JOURNAL = {Comm. Partial Differential Equations},
	FJOURNAL = {Communications in Partial Differential Equations},
	VOLUME = {47},
	YEAR = {2022},
	NUMBER = {8},
	PAGES = {1591--1635},
	ISSN = {0360-5302,1532-4133},
	MRCLASS = {35D30},
	MRNUMBER = {4462188},
	DOI = {10.1080/03605302.2022.2068425},
	URL = {https://doi.org/10.1080/03605302.2022.2068425},
}

@article {MRR2020,
	AUTHOR = {Maity, Debayan and Raymond, Jean-Pierre and Roy, Arnab},
	TITLE = {Maximal-in-time existence and uniqueness of strong solution of
	a 3{D} fluid-structure interaction model},
	JOURNAL = {SIAM J. Math. Anal.},
	FJOURNAL = {SIAM Journal on Mathematical Analysis},
	VOLUME = {52},
	YEAR = {2020},
	NUMBER = {6},
	PAGES = {6338--6378},
	ISSN = {0036-1410},
	MRCLASS = {35Q35 (35Q30 74F10 76D05)},
	MRNUMBER = {4189724},
	DOI = {10.1137/18M1178451},
	URL = {https://doi.org/10.1137/18M1178451},
}

@article {MT2021NARWA,
	AUTHOR = {Maity, Debayan and Takahashi, Tak\'{e}o},
	TITLE = {Existence and uniqueness of strong solutions for the system of
	interaction between a compressible {N}avier--{S}tokes-{F}ourier
	fluid and a damped plate equation},
	JOURNAL = {Nonlinear Anal. Real World Appl.},
	FJOURNAL = {Nonlinear Analysis. Real World Applications. An International
	Multidisciplinary Journal},
	VOLUME = {59},
	YEAR = {2021},
	PAGES = {Paper No. 103267, 34},
	ISSN = {1468-1218},
	MRCLASS = {35A01 (76N06)},
	MRNUMBER = {4183914},
	MRREVIEWER = {Xinghong Pan},
	DOI = {10.1016/j.nonrwa.2020.103267},
	URL = {https://doi.org/10.1016/j.nonrwa.2020.103267},
}

@article {MRT2021,
	AUTHOR = {Maity, Debayan and Roy, Arnab and Takahashi, Tak\'{e}o},
	TITLE = {Existence of strong solutions for a system of interaction
	between a compressible viscous fluid and a wave equation},
	JOURNAL = {Nonlinearity},
	FJOURNAL = {Nonlinearity},
	VOLUME = {34},
	YEAR = {2021},
	NUMBER = {4},
	PAGES = {2659--2687},
	ISSN = {0951-7715,1361-6544},
	MRCLASS = {35Q30 (35R37 76N10)},
	MRNUMBER = {4253566},
	DOI = {10.1088/1361-6544/abe696},
	URL = {https://doi.org/10.1088/1361-6544/abe696},
}

@article {Mitra2020,
	AUTHOR = {Mitra, Sourav},
	TITLE = {Local existence of strong solutions of a fluid-structure
	interaction model},
	JOURNAL = {J. Math. Fluid Mech.},
	FJOURNAL = {Journal of Mathematical Fluid Mechanics},
	VOLUME = {22},
	YEAR = {2020},
	NUMBER = {4},
	PAGES = {Paper No. 60, 38},
	ISSN = {1422-6928},
	MRCLASS = {35Q35 (35B65 35D35 35R37 74D99 74F10 76N06)},
	MRNUMBER = {4152298},
	MRREVIEWER = {Alexey V. Samokhin},
	DOI = {10.1007/s00021-020-00520-8},
	URL = {https://doi.org/10.1007/s00021-020-00520-8},
}

@article {MC2013,
	AUTHOR = {Muha, Boris and {\v{C}}ani\'{c}, Sun\v{c}ica},
	TITLE = {Existence of a weak solution to a nonlinear fluid-structure
	interaction problem modeling the flow of an incompressible,
	viscous fluid in a cylinder with deformable walls},
	JOURNAL = {Arch. Ration. Mech. Anal.},
	FJOURNAL = {Archive for Rational Mechanics and Analysis},
	VOLUME = {207},
	YEAR = {2013},
	NUMBER = {3},
	PAGES = {919--968},
	ISSN = {0003-9527},
	MRCLASS = {35Q35 (35D30 76D05 76Z05 92C35)},
	MRNUMBER = {3017292},
	MRREVIEWER = {Nader Masmoudi},
	DOI = {10.1007/s00205-012-0585-5},
	URL = {https://doi.org/10.1007/s00205-012-0585-5},
}

@article {MC2016,
	AUTHOR = {Muha, Boris and {\v{C}}ani\'{c}, Sun\v{c}ica},
	TITLE = {Fluid-structure interaction between an incompressible, viscous
	3{D} fluid and an elastic shell with nonlinear {K}oiter
	membrane energy},
	JOURNAL = {Interfaces Free Bound.},
	FJOURNAL = {Interfaces and Free Boundaries. Mathematical Analysis,
	Computation and Applications},
	VOLUME = {17},
	YEAR = {2015},
	NUMBER = {4},
	PAGES = {465--495},
	ISSN = {1463-9963},
	MRCLASS = {74F10 (35B40 35D30 74K15 76D03)},
	MRNUMBER = {3450736},
	MRREVIEWER = {John M. Stockie},
	DOI = {10.4171/IFB/350},
	URL = {https://doi.org/10.4171/IFB/350},
}

@article {MS2022,
	AUTHOR = {Muha, Boris and Schwarzacher, Sebastian},
	TITLE = {Existence and regularity of weak solutions for a fluid
	interacting with a non-linear shell in three dimensions},
	JOURNAL = {Ann. Inst. H. Poincar\'{e} C Anal. Non Lin\'{e}aire},
	FJOURNAL = {Annales de l'Institut Henri Poincar\'{e} C. Analyse Non
	Lin\'{e}aire},
	VOLUME = {39},
	YEAR = {2022},
	NUMBER = {6},
	PAGES = {1369--1412},
	ISSN = {0294-1449,1873-1430},
	MRCLASS = {74F10 (35D30 35Q30 74K25 76D05)},
	MRNUMBER = {4540753},
	MRREVIEWER = {Viatcheslav\ I.\ Pri\u{\i}menko},
	DOI = {10.4171/aihpc/33},
	URL = {https://doi.org/10.4171/aihpc/33},
}

@book {NS2004,
	AUTHOR = {Novotn\'{y}, A. and Stra\v{s}kraba, I.},
	TITLE = {Introduction to the mathematical theory of compressible flow},
	SERIES = {Oxford Lecture Series in Mathematics and its Applications},
	VOLUME = {27},
	PUBLISHER = {Oxford University Press, Oxford},
	YEAR = {2004},
	PAGES = {xx+506},
	ISBN = {0-19-853084-6},
	MRCLASS = {35Q35 (35Q30 76N10)},
	MRNUMBER = {2084891},
	MRREVIEWER = {Rapha\"{e}l\ Danchin},
}

@article {RT2019,
	AUTHOR = {Remond-Tiedrez, Antoine and Tice, Ian},
	TITLE = {The viscous surface wave problem with generalized surface
	energies},
	JOURNAL = {SIAM J. Math. Anal.},
	FJOURNAL = {SIAM Journal on Mathematical Analysis},
	VOLUME = {51},
	YEAR = {2019},
	NUMBER = {6},
	PAGES = {4894--4952},
	ISSN = {0036-1410,1095-7154},
	MRCLASS = {35Q30 (35B40 35R35 74F10 74K15 76D45 76E17)},
	MRNUMBER = {4039524},
	DOI = {10.1137/18M1195851},
	URL = {https://doi.org/10.1137/18M1195851},
}

@article {SS2022,
	AUTHOR = {Schwarzacher, Sebastian and Sroczinski, Matthias},
	TITLE = {Weak-strong uniqueness for an elastic plate interacting with
	the {N}avier-{S}tokes equation},
	JOURNAL = {SIAM J. Math. Anal.},
	FJOURNAL = {SIAM Journal on Mathematical Analysis},
	VOLUME = {54},
	YEAR = {2022},
	NUMBER = {4},
	PAGES = {4104--4138},
	ISSN = {0036-1410,1095-7154},
	MRCLASS = {35Q30 (34A12 35A02 35Q74 35R37)},
	MRNUMBER = {4450291},
	DOI = {10.1137/21M1443509},
	URL = {https://doi.org/10.1137/21M1443509},
}

@book {Triebel1978,
	AUTHOR = {Triebel, Hans},
	TITLE = {Interpolation theory, function spaces, differential operators},
	SERIES = {North-Holland Mathematical Library},
	VOLUME = {18},
	PUBLISHER = {North-Holland Publishing Co., Amsterdam-New York},
	YEAR = {1978},
	PAGES = {528},
	ISBN = {0-7204-0710-9},
	MRCLASS = {46E35 (35Jxx 46M35)},
	MRNUMBER = {503903},
	MRREVIEWER = {Robert D. Brown},
	URL = {https://www.sciencedirect.com/bookseries/north-holland-mathematical-library/vol/18},
}

@article {Trifunovic2023,
	AUTHOR = {Trifunovi\'{c}, Sr{\dj}an},
	TITLE = {Compressible fluids interacting with plates: regularity and
	weak-strong uniqueness},
	JOURNAL = {J. Math. Fluid Mech.},
	FJOURNAL = {Journal of Mathematical Fluid Mechanics},
	VOLUME = {25},
	YEAR = {2023},
	NUMBER = {1},
	PAGES = {Paper No. 13, 28},
	ISSN = {1422-6928,1422-6952},
	MRCLASS = {76N10 (35A02 35R37 74F10)},
	MRNUMBER = {4526391},
	MRREVIEWER = {Beno\^{i}t\ P.\ Desjardins},
	DOI = {10.1007/s00021-022-00759-3},
	URL = {https://doi.org/10.1007/s00021-022-00759-3},
}

@article {KM1,
    AUTHOR = {Klainerman, Sergiu and Majda, Andrew},
     TITLE = {Singular limits of quasilinear hyperbolic systems with large
              parameters and the incompressible limit of compressible
              fluids},
   JOURNAL = {Comm. Pure Appl. Math.},
  FJOURNAL = {Communications on Pure and Applied Mathematics},
    VOLUME = {34},
      YEAR = {1981},
    NUMBER = {4},
     PAGES = {481--524},
      ISSN = {0010-3640,1097-0312},
   MRCLASS = {35L40 (35Q10 76N10 76W05)},
  MRNUMBER = {615627},
       DOI = {10.1002/cpa.3160340405},
       URL = {https://doi.org/10.1002/cpa.3160340405},
}

@article{EB1,
    AUTHOR = {Ebin, David G.},
     TITLE = {The motion of slightly compressible fluids viewed as a motion
              with strong constraining force},
   JOURNAL = {Ann. of Math. (2)},
  FJOURNAL = {Annals of Mathematics. Second Series},
    VOLUME = {105},
      YEAR = {1977},
    NUMBER = {1},
     PAGES = {141--200},
      ISSN = {0003-486X},
   MRCLASS = {58D05 (76.58)},
  MRNUMBER = {431261},
MRREVIEWER = {Hideki\ Omori},
       DOI = {10.2307/1971029},
       URL = {https://doi.org/10.2307/1971029},
}

@article{SCH2,
    AUTHOR = {Schochet, Steven},
     TITLE = {The mathematical theory of low {M}ach number flows},
   JOURNAL = {M2AN Math. Model. Numer. Anal.},
  FJOURNAL = {M2AN. Mathematical Modelling and Numerical Analysis},
    VOLUME = {39},
      YEAR = {2005},
    NUMBER = {3},
     PAGES = {441--458},
      ISSN = {0764-583X,1290-3841},
   MRCLASS = {35Q35 (76G25 76N10)},
  MRNUMBER = {2157144},
       DOI = {10.1051/m2an:2005017},
       URL = {https://doi.org/10.1051/m2an:2005017},
}

@article{Gallag,
    AUTHOR = {Gallagher, Isabelle},
     TITLE = {R\'{e}sultats r\'{e}cents sur la limite incompressible},
      NOTE = {S\'{e}minaire Bourbaki. Vol. 2003/2004},
   JOURNAL = {Ast\'{e}risque},
  FJOURNAL = {Ast\'{e}risque},
    NUMBER = {299},
      YEAR = {2005},
     PAGES = {Exp. No. 926, vii, 29--57},
      ISSN = {0303-1179,2492-5926},
   MRCLASS = {35Q30 (35B30 35D05 35D10 76D03 76D05 76N10)},
  MRNUMBER = {2167201},
MRREVIEWER = {Bruno\ Scheurer},
}

@article{Ho,
    AUTHOR = {Hoff, David},
     TITLE = {Dynamics of singularity surfaces for compressible, viscous
              flows in two space dimensions},
   JOURNAL = {Comm. Pure Appl. Math.},
  FJOURNAL = {Communications on Pure and Applied Mathematics},
    VOLUME = {55},
      YEAR = {2002},
    NUMBER = {11},
     PAGES = {1365--1407},
      ISSN = {0010-3640,1097-0312},
   MRCLASS = {76N10 (35Q30)},
  MRNUMBER = {1916987},
MRREVIEWER = {Denis\ Serre},
       DOI = {10.1002/cpa.10046},
       URL = {https://doi.org/10.1002/cpa.10046},
}

@article{Da,
    AUTHOR = {Danchin, R.},
     TITLE = {Zero {M}ach number limit for compressible flows with periodic
              boundary conditions},
   JOURNAL = {Amer. J. Math.},
  FJOURNAL = {American Journal of Mathematics},
    VOLUME = {124},
      YEAR = {2002},
    NUMBER = {6},
     PAGES = {1153--1219},
      ISSN = {0002-9327,1080-6377},
   MRCLASS = {35Q35 (35Q30 76D05 76N10)},
  MRNUMBER = {1939784},
MRREVIEWER = {Gerhard\ O.\ Str\"{o}hmer},
       URL =
              {http://muse.jhu.edu/journals/american_journal_of_mathematics/v124/124.6danchin.pdf},
}

@article{DesGre,
    AUTHOR = {Desjardins, Benoit and Grenier, E.},
     TITLE = {Low {M}ach number limit of viscous compressible flows in the
              whole space},
   JOURNAL = {R. Soc. Lond. Proc. Ser. A Math. Phys. Eng. Sci.},
  FJOURNAL = {The Royal Society of London. Proceedings. Series A.
              Mathematical, Physical and Engineering Sciences},
    VOLUME = {455},
      YEAR = {1999},
    NUMBER = {1986},
     PAGES = {2271--2279},
      ISSN = {1364-5021,1471-2946},
   MRCLASS = {76N10 (35Q30)},
  MRNUMBER = {1702718},
MRREVIEWER = {Kevin\ R.\ Zumbrun},
       DOI = {10.1098/rspa.1999.0403},
       URL = {https://doi.org/10.1098/rspa.1999.0403},
}

@article{DGLM,
    AUTHOR = {Desjardins, B. and Grenier, E. and Lions, P.-L. and Masmoudi,
              N.},
     TITLE = {Incompressible limit for solutions of the isentropic
              {N}avier-{S}tokes equations with {D}irichlet boundary
              conditions},
   JOURNAL = {J. Math. Pures Appl. (9)},
  FJOURNAL = {Journal de Math\'{e}matiques Pures et Appliqu\'{e}es.
              Neuvi\`eme S\'{e}rie},
    VOLUME = {78},
      YEAR = {1999},
    NUMBER = {5},
     PAGES = {461--471},
      ISSN = {0021-7824},
   MRCLASS = {35Q30 (76D05 76N10)},
  MRNUMBER = {1697038},
MRREVIEWER = {J\"{u}rgen\ Socolowsky},
       DOI = {10.1016/S0021-7824(99)00032-X},
       URL = {https://doi.org/10.1016/S0021-7824(99)00032-X},
}

@article{Kl,
    AUTHOR = {Klein, R. and Botta, N. and Schneider, T. and Munz, C. D. and
              Roller, S. and Meister, A. and Hoffmann, L. and Sonar, T.},
     TITLE = {Asymptotic adaptive methods for multi-scale problems in fluid
              mechanics},
      NOTE = {Special issue on practical asymptotics},
   JOURNAL = {J. Engrg. Math.},
  FJOURNAL = {Journal of Engineering Mathematics},
    VOLUME = {39},
      YEAR = {2001},
    NUMBER = {1-4},
     PAGES = {261--343},
      ISSN = {0022-0833,1573-2703},
   MRCLASS = {76M25},
  MRNUMBER = {1826065},
       DOI = {10.1023/A:1004844002437},
       URL = {https://doi.org/10.1023/A:1004844002437},
}

@article{D, 
    AUTHOR = {Dafermos, C. M.},
     TITLE = {The second law of thermodynamics and stability},
   JOURNAL = {Arch. Rational Mech. Anal.},
  FJOURNAL = {Archive for Rational Mechanics and Analysis},
    VOLUME = {70},
      YEAR = {1979},
    NUMBER = {2},
     PAGES = {167--179},
      ISSN = {0003-9527},
   MRCLASS = {73B30 (35L65)},
  MRNUMBER = {546634},
MRREVIEWER = {C.-C.\ Wang},
       DOI = {10.1007/BF00250353},
       URL = {https://doi.org/10.1007/BF00250353},
}

\end{document}